\documentclass[reqno,12pt]{amsart}%
\usepackage[body={17cm,22cm}]{geometry}

\usepackage[T1]{fontenc}
\usepackage{graphicx}
\usepackage{mathrsfs}
\usepackage[french]{babel}
\usepackage{amscd,latexsym,amsthm,amsfonts,amssymb,amsmath,amsxtra}
\usepackage[colorlinks, urlcolor=blue,  citecolor=blue]{hyperref}
\usepackage[all]{xy}%
\providecommand{\U}[1]{\protect\rule{.1in}{.1in}}
\RequirePackage{amsmath}
\RequirePackage{amssymb}

\newcommand{\BA}{{\mathbb {A}}}

\newcommand{\BG}{{\mathbb {G}}}

\newcommand{\BQ}{{\mathbb {Q}}}
\newcommand{\BR}{{\mathbb {R}}}

\newcommand{\BZ}{{\mathbb {Z}}}

\newcommand{\CC}{{\mathcal {C}}}
\renewcommand{\CD}{{\mathcal {D}}}

\newcommand{\CF}{{\mathcal {F}}}

\newcommand{\CO}{{\mathcal {O}}}

\newcommand{\CT}{{\mathcal {T}}}
\newcommand{\CU}{{\mathcal {U}}}

\newcommand{\CX}{{\mathcal {X}}}

\newcommand{\CZ}{{\mathcal {Z}}}
\newcommand{\RA}{{\mathbf {A}}}

\newcommand{\Br}{{\mathrm{Br}}}

\newcommand{\codim}{{\mathrm{codim}}}

\newcommand{\coker}{{\mathrm{coker}}}

\newcommand{\Div}{{\mathrm{Div}}}
\renewcommand{\div}{{\mathrm{div}}}

\newcommand{\Gal}{{\mathrm{Gal}}}

\newcommand{\Hom}{{\mathrm{Hom}}}

\renewcommand{\Im}{{\mathrm{Im}}}

\newcommand{\Ker}{{\mathrm{Ker}}}

\newcommand{\Mor}{{\mathrm{Mor}}}

\newcommand{\Pic}{\mathrm{Pic}}

\newcommand{\Res}{{\mathrm{Res}}}

\newcommand{\Spec}{{\mathrm{Spec}}}

\newcommand{\Triv}{{\mathrm{Triv}}}

\newcommand{\tor}{{\mathrm{tor}}}

\font\cyr=wncyr10
\newcommand{\Sha}{\hbox{\cyr X}}

\newcommand{\iso}{\stackrel{\sim}{\rightarrow} }

\newcommand{\sbt}{\subset}

\newcommand{\bk}{\bar{k}}

\newcommand{\Ok}{\Omega_k}

\numberwithin{equation}{section}

\theoremstyle{remark}

\newtheorem{defi}{\rm{\textbf{D\'efinition}}}[section]

\newtheorem{exam}[defi]{\rm{\textbf{Example}}}

\newtheorem{rem}[defi]{\rm{\textbf{Remarque}}}

\newtheorem{ques}[defi]{\rm{\textbf{Question}}}

\theoremstyle{plain}

\newtheorem{thm}[defi]{\rm{\textbf{Th\'eor\`eme}}}

\newtheorem{cor}[defi]{\rm{\textbf{\textbf{Corollaire}}}}

\newtheorem{lem}[defi]{\rm{\textbf{Lemme}}}

\newtheorem{prop}[defi]{\rm{\textbf{\textbf{Proposition}}}}

\begin{document}

\title[]
{Approximation forte pour les vari\'et\'es avec une action d'un groupe lin\'eaire}

\author{Yang CAO}

\address{Yang CAO \newline Laboratoire de Math\'ematiques d'Orsay
\newline Univ. Paris-Sud, CNRS, Univ. Paris-Saclay \newline 91405 Orsay, France}

\email{yang.cao@math.u-psud.fr}

\date{\today.}

\maketitle

\begin{abstract}
 Soit $G$ un groupe lin\'eaire connexe sur un corps de nombres.   Soit $U \hookrightarrow X$ une inclusion $G$-\'equivariante
 d'un $G$-espace homog\`ene \`a stabilisateurs connexes dans une $G$-vari\'et\'e lisse.
 On montre que $X$ satisfait l'approximation forte  avec condition de  Brauer-Manin
 hors d'un ensemble   $S$ de places de $k$ dans chacun des cas suivants :
 
 (i) $S$ est l'ensemble des places archim\'ediennes;
 
 (ii) $S$ est un ensemble fini non vide quelconque, et $\bk^{\times}= \bk[X]^{\times}$.
 
  La d\'emonstration utilise le cas  $X=U$,  qui a fait l'objet de divers travaux.
  
  \medskip

\indent
Summary. 
  Let $G$ be a connected linear algebraic group over a number field. Let $U \hookrightarrow X$ be a $G$-equivariant open embedding
  of a $G$-homogeneous space with connected stabilizers into a smooth $G$-variety. 
We prove that $X$ satisfies strong approximation with  Brauer-Manin  condition  off a  set $S$ of places of $k$ under either of the following hypotheses :

(i) $S$ is the set of archimedean places;

(ii) $S$ is a nonempty finite set and $\bar{k}^{\times}= \bar{k}[X]^{\times}$.

 The proof builds upon the case $X=U$, which has been the object of several works.
  \end{abstract}

\tableofcontents

\section{Introduction}

Soit $k$ un corps de nombres. On note $\Omega_k $ l'ensemble des places de $ k $ et $ \infty_k $ l'ensemble des places archim\'ediennes de $ k $. 
On note $v <\infty_k $ pour $ v \in \Omega_k \setminus \infty_k $. 
Notons $\CO_k$
 l'anneau des entiers 
 de $k$ et $\CO_S $ l'anneau des   $S$-entiers de $ k $ pour un sous-ensemble fini $S$ de $\Omega_k $ contenant $\infty_k$. 
Pour chaque $ v \in \Omega_k $, on note $k_v$ le compl\'et\'e de $k$ en $v$ et  $ \CO_v\sbt k_v$ l'anneau des entiers ($\CO_v=k_v$ pour $v\in \infty_k$).  
Soit $ \RA_k $ l'anneau des ad\`eles de $k$.
Pour un sous-ensemble fini $S\sbt \Omega_k$, soit $\RA_k^S$   l'anneau des ad\`eles hors de $S$ et $k_S:=\prod_{v\in S}k_v$. 
Par exemple, soit $\RA_k^{\infty}$ l'anneau des ad\`eles finis et $k_{\infty}:=\prod_{v\in \infty_k}k_v$.

On rappelle les d\'efinitions de \cite{CTX09}, \cite[\S 2]{CTX13}, \cite{BD} et \cite{CX1}.

Pour $B$ sous-ensemble de $\Br (X)$, on d\'efinit
$$ X ({\RA}_k)^ B = \{(x_v)_{v \in \Omega_k} \in X ({\RA}_k): \ \ \sum_{v \in \Omega_k} \ inv_v (\xi (x_v)) = 0\in \BQ/\BZ, \ \ \forall \xi \in B \}.$$ 
Comme l'a remarqu\'e Manin, la th\'eorie du corps de classes donne $ X (k) \subseteq X (\RA_k)^B $.

D\'efinissons
\begin{equation}\label{Xbullet}
 X(\RA_k)_{\bullet} = \pi_0(X(k_{\infty})) \times  X(\RA_k^{\infty})
 \end{equation}
la projection, o\`u $ \pi_0 (X (k_{\infty})) $ est l'ensemble des composantes connexes de $X(k_{\infty})$. 

Puisque, pour tout $ v \in \infty_k $, 
chaque \'el\'ement de $ \Br (X) $ prend une valeur constante sur chaque composante connexe de $ \pi_0 (X (k_v)) $,  pour tout sous-ensemble $ B \sbt \Br (X)$ on peut d\'efinir :
$$ X ({\RA}_k)_{\bullet} ^ B = \{(x_v)_{v \in \Omega_k} \in X ({\RA}_k)_{\bullet}: \ \ \sum_{v \in \Omega_k} \ inv_v (\xi (x_v)) = 0\in \BQ/\BZ, \ \ \forall \xi \in B \}.$$

\begin{defi} (\cite{CTX09,CTX13}) Soient $k$ un corps de nombres, $ X $ une $k$-vari\'et\'e et $S\sbt \Omega_k$ un sous-ensemble fini.
Notons $Pr^S: X(\RA_k)\to X(\RA_k^S) $ la projection. 

(1) Si $X(k)$ est dense dans $Pr^S( X(\RA_k))$, on dit que $ X $ satisfait \emph{l'approximation forte} hors de $ S $.

(2) Si $X(k)$ est dense dans $Pr^S( X(\RA_k)^{B} )$ pour un sous-ensemble $ B$ de $ \Br (X) $, on dit que $X$ satisfait \emph{l'approximation forte  par rapport \`a $B$} hors de $S$.
On dit aussi alors que $X$ satisfait l'approximation forte de Brauer-Manin   par rapport \`a $B$ hors de $S$.
\end{defi}

Si $S=\infty_k$, on peut s'int\'eresser \`a  une question un peu plus pr\'ecise tenant compte des composantes connexes r\'eelles :  $X(k)$ est-il dense dans $X(\RA_k)^B_{\bullet}$ pour un sous-groupe $B\sbt \Br(X)$?

Pour les espaces homog\`enes de groupes alg\'ebriques connexes, ces questions ont fait ces derni\`eres ann\'ees  l'objet d'une s\'erie de travaux \cite{CTX09,Ha08, BD}
prolongeant des travaux classiques.

Lorsqu'on cherche \`a \'etendre la classe des vari\'et\'es satisfaisant les propri\'et\'es ci-dessus, il est naturel de poser les questions \ref{ques2} et \ref{ques1} suivantes.

\begin{ques}\label{ques2}
Soient $X$ une $k$-vari\'et\'e lisse g\'eom\'etriquement int\`egre et $U$ un ouvert de $X$. Si $U(k)$ est dense dans $U(\RA_k)^{\Br(U)}_{\bullet}$, sous quelles conditions peut-on \'etablir que  $X(k)$ est dense dans $X(\RA_k)^{\Br(X)}_{\bullet}$? 
\end{ques}

Le cas qui nous int\'eresse est 
celui o\`u 
$U$ est un $G$-espace homog\`ene.
En g\'en\'eral, l'inclusion $ \bk^{\times} \subset  \bk[U]^{\times}$ n'est pas un isomorphisme 
et, dans ce cas, il existe des exemples pour lesquels $U$ ne satisfait pas l'approximation forte par rapport \`a $\Br(U)$ 
hors d'un sous-ensemble fini $S\sbt \Omega_k$ non vide quelconque
(par exemple $U\cong \BG_m$, $k=\BQ$ et $S=\{v_0\}$ avec $v_0$ une place non archim\'edienne).

\begin{ques}\label{ques1}
Soit $S\sbt \Omega_k$ un sous-ensemble fini non vide. Soient $X$ une $k$-vari\'et\'e lisse g\'eom\'etriquement int\`egre et $U$ un ouvert de $X$ fibr\'e sur un certain $G$-espace homog\`ene $Z$. Si toute fibre de $U$ au-dessus d'un $k$-point de $Z$ satisfait l'approximation forte hors de $S$, sous quelles conditions peut-on \'etablir l'approximation forte pour $X$ par rapport \`a $\Br(X)$ hors de $S$? 
\end{ques}

Le principal r\'esultat de cet article est le :

\begin{thm}\label{thmgroupic}(cf. Th\'eor\`eme \ref{main1cor2})
Soit $G$ un groupe lin\'eaire connexe sur un corps de nombres $k$, et soit $X$ une $G$-vari\'et\'e lisse g\'eom\'etriquement int\`egre sur $k$. 
Supposons qu'il existe un $G$-ouvert $U\sbt X$ tel que $U\cong G/G_0$, o\`u $G_0\sbt G$ est un sous-groupe ferm\'e connexe.
Soit $S\sbt \Omega_k$ un sous-ensemble fini non vide et supposons que $G'(k_S)$ est non compact pour chaque facteur simple $G'$ du groupe $G^{sc}$. 

(1) Si $S\sbt \infty_k $, alors $X(k)$ est dense dans $X(\RA_k)_{\bullet}^{\Br(X)}$.

(2) Si $\bk^{\times}= \bk[X]^{\times}$, 
alors $X$ satisfait l'approximation forte de Brauer-Manin   par rapport \`a $\Br(X)$ hors de $S$.
\end{thm}

Ce th\'eor\`eme est une cons\'equence directe d'un r\'esultat plus g\'en\'eral mais d'\'enonc\'e technique  (Th\'eor\`eme \ref{main2thm}).

Le  th\'eor\`eme  \ref{thmgroupic} (1)  avait d\'ej\`a \'et\'e \'etabli dans les cas suivants :

(i)  $X$  est une vari\'et\'e torique, c'est-\`a-dire que $G$ est un tore  (F. Xu et l'auteur \cite{CX}).

(ii) $X$ est une vari\'et\'e groupique,  c'est-\`a-dire que $G$ est un groupe lin\'eaire connexe et que $G_{0}=1$ (F.~Xu et l'auteur \cite{CX1}).

(iii) $X$ est comme dans le th\'eor\`eme, avec $G_{0}$ r\'esoluble connexe (r\'esultat tout r\'ecent de D.~Wei \cite{wei}).

Au  \S \ref{4}, on donne une d\'emonstration totalement nouvelle  de ce r\'esultat,  laquelle est plus simple que celles des trois travaux ci-dessus.

Le  th\'eor\`eme  \ref{thmgroupic} (2)  avait d\'ej\`a \'et\'e \'etabli dans les cas suivants :

(iv) $X$  est une vari\'et\'e torique (D.~Wei \cite{wei1}).

(v) $X$ est une vari\'et\'e groupique (F.~Xu et l'auteur \cite{CX1}).

Les d\'emonstrations de ce r\'esultat,  comme celle de \cite{CX1}, 
reposent d'une part sur des constructions  g\'eom\'etriques sur un corps quelconque, d'autre part sur  les th\'eor\`emes
d'approximation forte avec condition de Brauer-Manin pour les espaces homog\`enes \'etablis dans \cite{CTX09, Ha08, BD}.
Les constructions g\'eom\'etriques du pr\'esent article s'inspirent  de celles de \cite{CX1} et 
la partie arithm\'etique du pr\'esent article utilise une g\'en\'eralisation de \cite{CX2} (cf. \S \ref{5}).

Dans les  quatre articles cit\'es, on a  la conclusion plus pr\'ecise :  dans les \'enonc\'es,
on peut remplacer $\Br(X)$ par le sous-groupe ``alg\'ebrique'' $\Br_1(X) \subset \Br(X)$, dont la d\'efinition est rappel\'ee ci-dessous.
Dans le cadre g\'en\'eral o\`u nous nous  pla\c{c}ons, il faut utiliser tout le groupe de Brauer $\Br(X)$. 
L'id\'ee cl\'e est la notion de sous-groupe de Brauer invariant (cf. \S 3), qui g\'en\'eralise \cite[\S 1.2]{BD}.

\medskip

Soit $G$ un groupe lin\'eaire connexe.
Le plan de l'article est le suivant: 

La section \ref{2} est consacr\'ee \`a divers pr\'eliminaires g\'eom\'etriques, notamment sur les $G$-vari\'et\'es et leurs torseurs sous un tore ou un groupe de type multiplicatif.
Au \S \ref{2.3}, on \'etudie les torseurs sous un groupe de type multiplicatif sur une $G$-vari\'et\'e.  
On  montre que tout tel torseur peut    \^etre muni canoniquement  d'une action d'un groupe lin\'eaire $H$ qui s'envoie  sur $G$ (Th\'eor\`eme \ref{thmaction}).
Pour cela, on utilise  un th\'eor\`eme de Colliot-Th\'el\`ene \cite[Thm. 5.6]{CT07} sur les torseurs au-dessus d'un groupe lin\'eaire connexe.

Au  \S \ref{3}, on  d\'efinit la notion  du  sous-groupe de Brauer  $G$-invariant $\Br_G(X)$
 d'une $G$-vari\'et\'e lisse (cf. D\'efinition \ref{def-invariant}).
On donne des caract\'erisations  \'equivalentes sur  les sous-groupes de $\Br_G(X)$ (Proposition \ref{prop-binvariant}). 
Ensuite, on d\'efinit la notion  d'homomorphisme de Sansuc (cf. D\'efinition \ref{defsansuc})
 et on g\'en\'eralise la suite exacte de Sansuc (cf. (\ref{braueralg-e1})).
On \'etudie la propri\'et\'e de $\Br_G(X)$ par rapport au
 passage \`a la fibre d'un $G$-morphisme vers un tore (Proposition \ref{propbrauersuj}).
On g\'en\'eralise la notion de $G$-espace homog\`ene \`a stabilisateur g\'eom\'etrique connexe
 et on d\'efinit la notion de pseudo $G$-espace homog\`ene (D\'efinition \ref{defpseudo}) qui sera utilis\'ee aux \S 4 et \S 6.
 
La section \ref{4} est consacr\'ee \`a l'approximation forte hors des places archim\'ediennes.
On \'etablit le th\'eor\`eme \ref{mainthminfty} 
sur l'approximation  d'un point ad\'elique de $X$ satisfaisant une condition de Brauer-Manin
par un tel point  situ\'e sur $U$.
Comme cons\'equence, on r\'epond \`a la  question \ref{ques2}   (Corollaire \ref{main1cor1}) dans ce cas.
Ceci donne le th\'eor\`eme \ref{thmgroupic} (1).

Au  \S \ref{5}, en utilisant la notion de sous-groupe   $G$-invariant
du groupe de Brauer,  on combine la suite exacte (\ref{braueralg-e1})
et la m\'ethode de \cite{CX2},  et 
on g\'en\'eralise \cite[Thm. 1.2]{CX2}.
On \'etablit un  th\'eor\`eme de descente pour un torseur sous un groupe lin\'eaire connexe quelconque (Th\'eor\`eme \ref{Dmain}).
 Comme cons\'equence, 
 on \'etablit une formule sur les points ad\'eliques  d'un certain $G$-espace homog\`ene
 satisfaisant une condition de Brauer-Manin (Corollaire \ref{Dmaincor2}).

Au  \S \ref{6}, pour une $G$-vari\'et\'e $X$ munie d'un $G$-ouvert $U$ fibr\'e sur l'autre $G$-vari\'et\'e $Z$, 
on consid\`ere l'approximation  d'un point ad\'elique de $X$ satisfaisant une condition de Brauer-Manin
par un tel point  situ\'e dans la fibre au-dessus  d'un point rationnel  de $Z$ (Question \ref{ques4}).
Pour r\'epondre \`a la question \ref{ques4}, 
on g\'en\'eralise \cite[Prop. 3.6]{CX1} 
et on \'etablit  au \S \ref{6.1} un th\'eor\`eme plus fin sur l'approximation forte   d'un espace affine (th\'eor\`eme \ref{strangapptoricstandthm}).
Ensuite, on combine ce th\'eor\`eme et  l'\'etude du sous-groupe de Brauer  $G$-invariant (\S 3) 
avec la m\'ethode de fibration de Colliot-Th\'el\`ene et Harari \cite{CTH}, 
et on r\'epond \`a la question \ref{ques4}
 dans le cas o\`u $Z$ est un certain tore (Th\'eor\`eme  \ref{thmfibration} et Th\'eor\`eme \ref{mainlem}).
Au cas o\`u $Z$ est un pseudo $G$-espace homog\`ene, on r\'esoud la question \ref{ques4} (Th\'eor\`eme \ref{main1thm})
\`a l'aide de tous les r\'esultats ci-dessus (sauf ceux du \S \ref{4}).

La section \ref{7} est consacr\'ee \`a la preuve des r\'esultats principaux \`a l'aide de tous les r\'esultats ci-dessus.
En utilisant la descente (\S \ref{5}), on \'etablit d'abord un r\'esultat (Proposition \ref{main2prop}) sur 
l'approximation forte pour les $G$-espaces homog\`enes \`a stabilisateur g\'eom\'etrique connexe, ce qui g\'en\'eralise un r\'esultat de Borovoi et Demarche (\cite[Thm. 1.4]{BD}).
Ensuite, on combine ce r\'esultat avec les r\'esultats des  \S \ref{4} et \S \ref{6}, 
 et on \'etablit le th\'eor\`eme principal (Th\'eor\`eme \ref{main1cor2})
sur l'approximation forte d'une  $G$-vari\'et\'e munie d'un $G$-ouvert fibr\'e sur un certain $G$-espace homog\`ene.
Ceci donne imm\'ediatement   le th\'eor\`eme \ref{thmgroupic}.

\bigskip

\textbf{Conventions et notations}. 

Soit $k$ un corps quelconque de caract\'eristique $0$. On note $\overline{k}$ une cl\^oture alg\'ebrique et $\Gamma_k:=\Gal (\bk/k)$.

Tous les groupes de cohomologie sont des groupes de cohomologie \'etale.
Soit $X$ un sch\'ema int\`egre.  On note $\eta_X$ le point g\'en\'erique de $X$.

Une $k$-vari\'et\'e $X$ est un $k$-sch\'ema s\'epar\'e de type fini. 
 Pour $X$ une telle vari\'et\'e, on note $k[X]$ son anneau des fonctions globales,
$k[X]^{\times}$ son groupe des fonctions inversibles,
$\Pic(X):=H^1_{\text{\'et}}(X,\BG_m)$ son groupe de Picard et
$\Br(X):=H_{\text {\'et}}^2 (X, \BG_m)$ son groupe de Brauer. Notons
$$\Br_1 (X) := \Ker [\Br (X) \to\Br (X_ {\bk})]\ \ \text{ et}\ \ \Br_a(X):=\Br_1(X)/\Im\Br(k).$$
Le groupe $\Br_1 (X)$ est le sous-groupe ``alg\'ebrique'' du groupe de Brauer de $X$.
Si $X$ est int\`egre, on note $k(X)$ son  corps des fonctions rationnelles.

On note  $X_{sing}$ le lieu singulier de $X$,
 et pour un sous-ensemble ferm\'e $D\sbt X$, on note $D_{sing }$ le lieu singulier de $D$ comme sous-vari\'et\'e ferm\'ee r\'eduite.

Un $k$-groupe alg\'ebrique $G$ est une $k$-vari\'et\'e qui est un $k$-sch\'ema en groupes. 
On note $e_G$ l'unit\'e de $G$ et $G^*$ le groupe des caract\`eres de $G_{\bk}$.
C'est un module galoisien de type fini.
Si $G$ est connexe, on note 
\begin{equation}\label{Bredef}
\Br_{e}(G):=\Ker (\Br_1(G)\xrightarrow{e_G^*} \Br(k))\cong \Br_a(G).
\end{equation}

Soit $G$ un groupe lin\'eaire connexe. On note $G^{tor}$ son  quotient torique maximal, $G^u$ son radical unipotent,
$G^{red}:=G/G^u$ son quotient r\'eductif maximal,
  $G^{ssu} :=\Ker (G\to G^{tor})$, $G^{ss}:=G^{ssu}/G^u$ et $G^{sc}$ le rev\^etement simplement connexe du
groupe semi-simple $G^{ss}$. Alors on a $G^*=(G^{tor})^*$.

 Soit $G$ un $k$-groupe alg\'ebrique. Une \emph{$G$-vari\'et\'e} $(X,\rho )$ (ou $X$) est une $k$-vari\'et\'e $X$ munie d'une action \`a gauche $G\times_k X\xrightarrow{\rho}X$. 
 Un ouvert $U\sbt X$ est un \emph{$G$-ouvert} si $U$ est invariant sous l'action de $G$.
Un $k$-morphisme de $G$-vari\'et\'es est appel\'e \emph{$G$-morphisme} s'il est compatible avec l'action de $G$.

Soit $G$ un $k$-groupe alg\'ebrique connexe et $X$ une $G$-vari\'et\'e lisse g\'eom\'e\-triquement int\`egre.
Notons $\Br_G(X)$ \emph{le sous-groupe de Brauer $G$-invariant} (cf: D\'efinition \ref{def-invariant}).
Dans le cas o\`u $X\cong G/G_0$ avec $G$ lin\'eaire et $G_0\sbt G$ un sous-groupe ferm\'e connexe,
Borovoi et Demarche ont  d\'efini $\Br_1(X,G):=\Ker(\Br(X)\to \Br(G_{\bk}))$ (\cite[\S 1.2]{BD})
En fait, on a (Proposition \ref{corbraueralgebraic1}): $\Br_G(X)\cong \Br_1(X,G)$.

Soit $T$ un tore. 
\emph{Une vari\'et\'e torique} $(T\hookrightarrow X)$ est une $T$-vari\'et\'e lisse int\`egre $X$ 
munie d'une immersion ouverte fix\'ee $T\hookrightarrow X$ de $T$-vari\'et\'es.
Pour une $k$-alg\`ebre finie s\'eparable $K$, on a une vari\'et\'e torique canonique:
$(\Res_{K/k}\BG_m\hookrightarrow \Res_{K/k}\BA^1).$
 
Soit $k$ un corps de nombres. Soit $X$ une $k$-vari\'et\'e. On note $X(\RA_k)$ l'ensemble des points ad\'eliques de $X$
et on note $X(\RA_k)_{\bullet}$  comme en (\ref{Xbullet}). 
Soit $G$ un $k$-groupe alg\'ebrique. On note $G(k_{\infty})^+$ la composante connexe de l'identit\'e de $G(k_{\infty}):=\prod_{v\in \infty_k}G(k_v)$.

\section{Pr\'eliminaires sur les $G$-vari\'et\'es et les torseurs}\label{2}

Dans toute cette section,  $k$ est un corps quelconque de caract\'eristique $0$. 
Sauf  mention explicite du contraire,  une vari\'et\'e est une $k$-vari\'et\'e.
Dans \cite{CX1}, le r\'esultat de structure g\'eom\'etrique est \cite[Prop. 3.12]{CX1}. 
On en donne une g\'en\'eralisation (Proposition \ref{propopensubset} et Proposition \ref{propextension}). 
Dans  \cite[Lem. 1.6.2]{CTS87},
Colliot-Th\'el\`ene et Sansuc  ont \'etudi\'e la structure des torseurs $Y \to X$ sous un tore   et obtenu la suite exacte (\ref{tor-e1}).
On pr\'ecise ici les morphismes de la suite exacte (\ref{tor-e1}) (Proposition \ref{proptor}).
Colliot-Th\'el\`ene (\cite[Thm. 5.6]{CT07})
a \'etudi\'e la structure des   torseurs $Y \to X$ sous un groupe de type mutiplicatif lorque la base $X$
est un groupe lin\'eraire connexe. 
On  consid\`ere ici  le cas plus g\'en\'eral  des torseurs sur une vari\'et\'e  $X$ munie d'une action d'un groupe $G$.  
 On \'etablit 
 un  th\'eor\`eme g\'en\'eral 
(Th\'eor\`eme  \ref{thmaction}) sur la structure de ces torseurs.

\subsection{Pr\'eliminaires sur les $G$-vari\'et\'es}

On rappelle un r\'esultat pour les vari\'et\'es toriques lisses (\cite[Prop. 2.10]{CX}).

\begin{lem}\label{lemtorique}
Soient $T$ un tore et $(T\hookrightarrow Z)$ une vari\'et\'e torique lisse. Soit $Z_1:=Z\setminus [(Z\setminus T)_{sing}]$. 
Alors $(T\hookrightarrow Z_1)$ est une vari\'et\'e torique, $\codim(Z\setminus Z_1,Z)\geq 2$ 
et chaque $T_{\bk}$-orbite de  $(Z_1\setminus T)_{\bk}$ est de codimension $1$ et son stabilisateur g\'eom\'etrique est isomorphe \`a $\BG_{m,\bk}$.
\end{lem}

\begin{proof}
Puisque $(Z\setminus T)_{sing}$ est $T$-invariant et $\dim((Z\setminus T)_{sing})< \dim(Z\setminus T) $, 
la vari\'et\'e $(T\hookrightarrow Z_1)$ est une vari\'et\'e torique et $\codim(Z\setminus Z_1,Z)\geq 2$.

Supposons que $k=\bk$. Dans ce cas,  toutes les vari\'et\'es toriques lisses de dimension $d$ ont un recouvrement par des vari\'et\'es toriques $(\BG_m^n\hookrightarrow \BG_m^i\times \BA^{n-i})_i$ (cf. \cite[Lem. 2.1]{CX}), 
et donc $Z_1$ admet un recouvrement par des vari\'et\'es toriques $\BG_m^{\dim(Z)-1}\times \BA^1$ et $\BG_m^{\dim(Z)}$. 
Le r\'esultat en d\'ecoule.
\end{proof}

\begin{prop}\label{propopensubset}
Soient $T\sbt Z_1 \sbt Z$ comme dans le lemme \ref{lemtorique}.
Soit $G$ un groupe lin\'eaire connexe muni d'un homomorphisme surjectif $G\xrightarrow{\varphi} T$ de noyau connexe.
Soit $X$ une $G$-vari\'et\'e lisse g\'eom\'e\-tri\-que\-ment int\`egre munie d'un $G$-morphisme dominant $X\xrightarrow{f}Z$. 
Soit $U$ un $G$-ouvert de $X$ tel que $f(U)\sbt T$. On a:

(1) le morphisme $f^{-1}(Z_1)\xrightarrow{f|_{Z_1}}Z_1$ est plat;

(2) si $f$ induit un isomorphisme 
$\Div_{Z_{\bk}\setminus T_{\bk}}(Z_{\bk})\xrightarrow{f^*_{\Div}}\Div_{X_{\bk}\setminus U_{\bk}}(X_{\bk})$,
alors il existe un $G$-ouvert $X_1$ de $X$ tel que $f(X_1)\sbt Z_1$, $X_1\cap f^{-1}(T)=U$, $\codim(X\setminus X_1,X)\geq 2$ et $X_1\xrightarrow{f|_{X_1}}Z_1$ soit lisse, surjectif \`a fibres g\'eom\'etriquement int\`egres.
\end{prop}

\begin{proof}
  Par le lemme \ref{lemtorique}, $Z_1$ est $T$-invariant, $\codim(Z\setminus Z_1,Z)\geq 2$ et $(Z_1\setminus T)_{\bk}=\coprod_iF_i$, o\`u $F_i$ est $T_{\bk}$-orbite de codimension $1$.
 Donc  $f^{-1}(Z_1)$ et $\overline{f^{-1}(T)\setminus U}\sbt X$ sont $G$-invariants.
 
 Pour (1), on peut supposer que $X=f^{-1}(Z_1)$ et $k=\bk$. 
Puisque  $Z_1$ est lisse, $X$ est int\`egre et $f$ est dominant, il existe un ouvert $V\sbt Z_1$ tel que $\codim(Z_1\setminus V,Z_1)\geq 2$ et $f^{-1}(V)\to V$ soit plat. Donc pour chaque $i$, $V\cap F_i\neq\emptyset$. 
En utilisant l'action de $G$, on a que $f|_{Z_1}$ est plat.
 
  Pour (2), puisque $f^*_{\Div}$ est un isomorphisme, on a 
  $$\codim(\overline{f^{-1}(T)\setminus U},X)\geq 2 \ \ \ \text{ et}\ \ \ \codim(f^{-1}(Z\setminus Z_1),X)\geq 2.$$
  On note $X_2:=f^{-1}(Z_1)\setminus \overline{f^{-1}(T)\setminus U}$ et $X_3:=X_2\setminus (X_2\setminus U)_{sing}$.
  Alors $(X_3\setminus U)=D\coprod E$ avec $\codim(E,X_3)\geq 2$ et toutes les composantes connexes de $D$ sont de dimension $\dim(X)-1$.
On note $X_1:=X_3\setminus E$.
Alors $X_1$ est $G$-invariant, $f(X_1)\sbt Z_1$, $ X_1\cap f^{-1}(T)=U$, $\codim(X\setminus X_1,X)\geq 2$ et
chaque composante connexe de  $(X_1\setminus U)_{\bk}$ est lisse, int\`egre de dimension $\dim(X)-1$.

Puisque $f^*_{\Div}$ est un isomorphisme, $(X_1\setminus U)_{\bk}\cong \coprod_iD_i$ avec $D_i=f^{-1}(F_i)\cap X_1$. 
Alors chaque $D_i$ est une $G_{\bk}$-vari\'et\'e lisse int\`egre de dimension $\dim(X)-1$. 
Puisque $\Ker(\varphi)$ est connexe, le stabilisateur de $F_i$ comme $G_{\bk}$-vari\'et\'e est connexe. 
Par \cite[Prop. 2.2]{CX1}, les morphismes $U\to T$ et $D_i\to F_i$ sont lisses  surjectifs \`a fibres g\'eom\'etriquement int\`egres. Le r\'esultat en d\'ecoule.
\end{proof}

\begin{prop}\label{propextension}
Soient $T$ et $T_0$ deux tores avec $T_0$ quasi-trivial.
Soient $X$ une vari\'et\'e lisse g\'eom\'etriquement int\`egre et $U\sbt X$ un ouvert muni d'un morphisme $U\xrightarrow{f}T_0\times T$.
Supposons que la composition
$$\Lambda: \ T_0^* \xrightarrow{p_1^*}T_0^*\times T^* \xrightarrow{f^*} \bk[U]^{\times}/\bk^{\times} \xrightarrow{\div_X} \Div_{X_{\bk}\setminus U_{\bk}}(X_{\bk})$$
est un isomorphisme. 

(1) Alors il existe un homomorphisme $T_0\xrightarrow{\phi}T$ et une vari\'et\'e torique
 $(T_0\hookrightarrow \BA^l)$ tels que:
 
 (a) il existe  une $k$-alg\`ebre finie s\'eparable $K$ et un isomorphisme de vari\'et\'es toriques:
 \begin{equation}\label{exten-e1}
 (T_0\hookrightarrow \BA^l)\cong (\Res_{K/k}\BG_m\hookrightarrow \Res_{K/k}\BA^1_K );
 \end{equation}
 
(b) si l'on note
$$
T_0\times T\xrightarrow{\widetilde{\phi}}T_0\times T:\ (t_0,t)\mapsto (t_0,t-\phi(t_0)),
$$
alors le morphisme $\widetilde{\phi} \circ f$ peut \^etre \'etendu \`a un  unique morphisme $X\xrightarrow{f'}\BA^l\times T$;

(c) on a un isomorphisme $\Div_{(\BA^l\times T)_{\bk}\setminus (T_0\times T)_{\bk}}((\BA^l\times T)_{\bk}) \xrightarrow{{f'}^*}\Div_{X_{\bk}\setminus U_{\bk}}(X_{\bk}).$

(2) Soit $G$ un groupe lin\'eaire muni d'un homomorphisme $G\xrightarrow{\varphi} T_0\times T$.
Supposons que $X$ est une $G$-vari\'et\'e, $U\sbt X$ est un $G$-ouvert et $f$ est un $G$-morphisme. 
Alors  le morphisme $f'$ est un $G$-morphisme, o\`u l'action de $G$ sur $\BA^l\times T$ est induite par $\widetilde{\phi}\circ \varphi$.
\end{prop}

\begin{proof}
Soient $\{D_i\}_{i=1}^l$ les composantes irr\'eductibles  de $(X\setminus U)_{\bk}$ dont la dimension est $\dim(X)-1$. 
Alors $\Div_{X_{\bk}\setminus U_{\bk}}(X_{\bk})\cong \oplus_i\BZ\cdot [D_i]$. 
Puisque $\Lambda$ est un isomorphisme, il existe une base $\{x_i\}$ de $ T_0^*$ telle que $\Lambda(x_i)=D_i$. 
Puisque $\bk[T_0]^{\times}=\bk^{\times}\oplus T_0^*$, on peut supposer que $x_i\in \bk[T_0]^{\times}$ et $\{x_i\}_{i=1}^l$ est globalement  $\Gal(\bk/k)$-invariant. 
Soit $K$ la  $k$-alg\`ebre finie s\'eparable qui correspond au $\Gal(\bk/k)$-ensemble $\{x_i\}_{i=1}^l$.
Alors $T_0\cong \Res_{K/k}\BG_{m,K}$,  et l'immersion ouverte
$$T_{0,\bk}\cong \Spec\ \bk[x_1,x_1^{-1},\cdots ,x_l,x_l^{-1}]\hookrightarrow \Spec\ \bk[x_1,\cdots ,x_l]\cong \BA^l_{\bk}\cong (\Res_{K/k}\BA^1_{K})_{\bk}$$
est exactement l'immersion ouverte $(\Res_{K/k}\BG_{m,K}\hookrightarrow \Res_{K/k}\BA^1_{K})_{\bk}$.
Puisque $\BA^l\cong \Res_{K/k}\BA^1_{K}$, on obtient canoniquement une vari\'et\'e torique $T_0\hookrightarrow \BA^l$.

 Notons 
 $T_{\bk}\cong \Spec\ \bk[t_1,t_1^{-1},\cdots ,t_n,t_n^{-1}]$ et donc 
 $$  (\BA^l\times T)_{\bk}\cong \Spec\ \bk[x_1,\cdots ,x_l, t_1,t_1^{-1},\cdots ,t_n,t_n^{-1}].$$
 Soit $T_0\xrightarrow{\phi}T$ l'homomorphisme correspondant \`a
  la composition
$$\xymatrix{\phi^*:&T^* \ar[r]^-{p_2^*}&T_0^*\times T^*\ar[r]^-{f^*} &\bk[U]^{\times}/\bk^{\times}\ar[r]^-{\div_X}& \Div_{X_{\bk}\setminus U_{\bk}}(X_{\bk})\ar[r]^-{\Lambda^{-1}}_-{\cong}&T_0^*.
}$$
 Puisque l'homomorphisme $T_0^*\times T^*\xrightarrow{\widetilde{\phi}^*}T_0^*\times T^*$  est
  d\'efini par $\ (t_0,t)\mapsto (t_0-\phi^*(t),t)$,
 on a 
 $$(\div_X\circ f^*\circ \widetilde{\phi}^*)(x_i)=(\div_X\circ f^*)(x_i)=\Lambda (x_i)=D_i$$
 et 
 $$(\div_X\circ f^*\circ \widetilde{\phi}^*)(t_i)=((\div_X\circ f^*)-(\div_X\circ f^*\circ p_1^*\circ \phi^*))(t_i)=((\div_X\circ f^*)-(\Lambda \circ \phi^*))(t_i)=0.$$
 Puisque $X$ est lisse, et donc normale, 
 on a  
 $$( \widetilde{\phi}\circ f)^*(\bk[\BA^l\times T])\sbt \bk[X]\ \ \ \text{ et donc} \ \ \ ( \widetilde{\phi}\circ f)^*(k[\BA^l\times T])\sbt k[X].$$
 Alors $ \widetilde{\phi}\circ f$ s'\'etend en un morphisme $X\xrightarrow{f'}\BA^l\times T$.
 
 Pour (c), puisque dans le diagramme commutatif
 $$\xymatrix{T_0^*\ar[d]^=\ar[r]^-{p_1^*}&T_0^*\times T^*\ar[r]^-{=}\ar[d]&\bk[T_0\times T]^{\times}/\bk^{\times} \ar[d]^-{(\widetilde{\phi}\circ f)^*}\ar[r]^-{\div}& \Div_{(\BA^l\times T)_{\bk}\setminus (T_0\times T)_{\bk}}((\BA^l\times T)_{\bk})\ar[d]^{(f')^*}\\
 T_0^*\ar[r]^-{p_1^*}&T_0^*\times T^*\ar[r]^-{(\widetilde{\phi}\circ f)^*} &\bk[U]^{\times}/\bk^{\times}\ar[r]^{\div_X} &\Div_{X_{\bk}\setminus U_{\bk}}(X_{\bk}),
 }$$
 les compositions horizontales sont  des isomorphismes, $(f')^*$ est un isomorphisme.
 
 Pour (2), puisque le diagramme
 $$\xymatrix{G\times U\ar@{^{(}->}[r]&G\times X\ar[d]^{\rho_X}\ar[r]^-{(\widetilde{\phi}\circ f,f')}&(T_0\times T)\times (\BA^l\times T)\ar[d]\\
 &X\ar[r]^{f'}&\BA^l\times T
 }$$
 est commutatif en $G\times U$, ce diagramme est commutatif et $f'$ est un $G$-morphisme.
\end{proof}

\medskip

\subsection{Trivialisation d'un torseur}
 Les torseurs sous un tore ou sous un groupe de type multiplicatif sont  \'etudi\'es par  Colliot-Th\'el\`ene et Sansuc \cite{CTS87}.
Soit $T$ un tore. Soient $X$ une vari\'et\'e lisse g\'eom\'etriquement int\`egre et $U\sbt X$ un ouvert.
 Par \cite[Lem. 1.6.2]{CTS87}, on a une suite exacte canonique:
 \begin{equation}\label{tor-e1}
 \xymatrix{H^0(U,T)\ar[r]^-{\Phi}&\Hom_k(T^*,\Div_{X_{\bk}\setminus U_{\bk}}(X_{\bk}))\ar[r]^-{\Psi}&H^1(X,T)\ar[r]&H^1(U,T),
 }\end{equation}
 et pour chaque $\chi\in H^0(U,T)=\Mor(U,T)$, on a 
 $\Phi(\chi): T^*\xrightarrow{\chi^*}\bk[U]^{\times}/\bk^{\times}\xrightarrow{\div_X}\Div_{X_{\bk}\setminus U_{\bk}}(X_{\bk})$.
 
 Pour un $T$-torseur $Y\xrightarrow{p} X$ tel que $Y|_U$ soit trivial, on note $V:=Y\times_X U$ et $\Triv_U(V,T)$ l'ensemble des trivialisations $V\xrightarrow{\tau} T\times U$. On a une application canonique
 \begin{equation}\label{tor-e2}
 \Triv_U(V,T)\xrightarrow{\Upsilon}\Hom_k(T^*,\Div_{X_{\bk}\setminus U_{\bk}}(X_{\bk}))
 \end{equation}
 d\'efinie par: pour chaque $\tau\in \Triv_U(V,T)$, le morphisme $\Upsilon( \tau)$ est la composition:
 $$ T^*\xrightarrow{p_1^*} \bk[T\times U]^{\times}/\bk^{\times}\xrightarrow{\tau^*}\bk[V]^{\times}/\bk^{\times} \xrightarrow{\div_Y} \Div_{Y_{\bk}\setminus V_{\bk}}(Y_{\bk})\xrightarrow{(p^*)^{-1}}\Div_{X_{\bk}\setminus U_{\bk}}(X_{\bk}),$$
o\`u $\Div_{X_{\bk}\setminus U_{\bk}}(X_{\bk})\xrightarrow{p^*}\Div_{Y_{\bk}\setminus V_{\bk}}(Y_{\bk})$ est un isomorphisme par \cite[Lem. B.1]{CT07}. 
On a aussi une application canonique
 \begin{equation}\label{tor-e3}
 H^0(U,T)\times \Triv_U(V,T)\xrightarrow{\Theta}\Triv_U(V,T):\ \ \ (\chi,\tau)\mapsto \widehat{\chi}\circ \tau
\end{equation} 
o\`u $T\times U\xrightarrow{\widehat{\chi}}T\times U: (t,u)\mapsto (t+\chi(u),u)$.

 \begin{lem}\label{lemtor}
 L'application $\Theta$ induit une action transitive de $H^0(U,T)$ sur $\Triv_U(V,T)$, et pour chaque $\tau\in \Triv_U(V,T)$, $\chi\in H^0(U,T)$, on a $\Upsilon (\Theta(\chi,\tau))=\Phi(\chi)+\Upsilon (\tau)$.
 \end{lem}
 
 \begin{proof}
 Pour $\chi_1,\chi_2\in H^0(U,T)$, on a $\widehat{\chi_1+\chi_2}=\widehat{\chi_2}\circ \widehat{\chi_1}$, et donc $\Theta$ est une action.
 
 Notons $T\times U\xrightarrow{p_1}T$. Pour deux trivialisations $\tau_1,\tau_2\in \Triv_U(V,T)$, le morphisme $\tau_2\circ\tau_1^{-1}$ est un $T$-morphisme sur $U$,
  i.e. pour chaque $t\in T$ et $u\in U$, on a
  $$(\tau_2\circ\tau_1^{-1})(t,u)=((p_1\circ \tau_2\circ\tau_1^{-1})(t,u),u)=(t+(p_1\circ \tau_2\circ\tau_1^{-1})(e_T,u),u).$$
  Notons $\chi:=(p_1\circ \tau_2\circ\tau_1^{-1})(e_T,-)\in H^0(U,T)$, alors $\tau_2=\widehat{\chi}\circ \tau_1$ et donc $\Theta$ est transitive.
  
Notons $T\times U\xrightarrow{p_2}U$.  Pour tout $\tau\in \Triv_U(V,T)$ et tout $\chi\in H^0(U,T)$, on a 
  $$p_1\circ \Theta(\chi,\tau)=p_1\circ \widehat{\chi}\circ \tau=(p_1+\chi\circ p_2)\circ \tau=p_1\circ \tau +\chi\circ p_2\circ \tau= p_1\circ\tau + \chi\circ p|_U $$
  dans $\Mor(V,T)$.
 Puisque les deux morphismes compos\'es
$$\bk[U]^{\times}/\bk^{\times}\xrightarrow{p|_U^*} \bk[V]^{\times}/\bk^{\times}\xrightarrow{\div_Y}\Div_{Y_{\bk}\setminus V_{\bk}}(Y_{\bk}) \ \ \ \text{et}\ \ \  
\bk[U]^{\times}/\bk^{\times}\xrightarrow{\div_X}\Div_{X_{\bk}\setminus U_{\bk}}(X_{\bk})\xrightarrow{p^*}\Div_{Y_{\bk}\setminus V_{\bk}}(Y_{\bk})$$
co\"incident, on a $\div_Y\circ (\chi\circ p|_U)^*=p^*\circ \div_X\circ \chi^*$ et le r\'esultat en d\'ecoule.
 \end{proof}

 \begin{prop}\label{proptor}
 Avec les notations des (\ref{tor-e1}) et (\ref{tor-e2}), pour un $\phi\in \Hom_k(T^*,\Div_{X_{\bk}\setminus U_{\bk}}(X_{\bk}))$, 
 soit $Y\to X$ un torseur correspondant \`a $\Psi (\phi)$ et  soit $V:=Y\times_XU$. 
 Alors apr\`es avoir bien choisi le signe de $\Psi$, il existe une trivialisation $\tau\in \Triv_U(V,T)$ telle que $\Upsilon (\tau)=\phi$ et que, pour chaque $\tau'\in \Triv_U(V,T)$, on ait $\Psi (\Upsilon (\tau'))=[Y]$.
 \end{prop}
 
 \begin{proof}
 D'apr\`es le lemme \ref{lemtor}, l'\'enonc\'e est \'equivalent \`a  l'existence d'une trivialisation $\tau\in \Triv_U(V,T)$ telle que $\Psi (\Upsilon (\tau))=[Y]$.
 
 \'Etape (1). Supposons que $k=\bk$ et $T=\BG_m$. Dans ce cas, $T^*\cong \BZ$.
 La suite exacte (\ref{tor-e1}) est obtenue en appliquant $H^i(X,-)=Ext^i_{X_{\text{\'et}}}(\BZ,-)$ \`a la suite exacte de faisceaux (cf. \cite[Lem. 1.6.2]{CTS87})
 $$0\to \BG_{m,X}\to j_*\BG_{m,U}\to {\textbf{Div}}_{X\setminus U}(X)\to 0$$
 o\`u $U\xrightarrow{j}X$. 
 Donc, apr\`es avoir choisi $1_{T^*}\in T^*$, le morphisme 
 $$\Div_{X\setminus U}(X)\xleftarrow{\sim}\Hom_k(T^*,\Div_{X\setminus U}(X))\xrightarrow {\Psi} H^1(X,T)\iso \Pic(X)$$
 est le morphisme canonique $\Div_{X\setminus U}(X)\rightarrow \Pic(X)$.
 Par \cite[Lem. B.1]{CT07}, pour chaque trivialisation $\tau$, on a un diagramme commutatif:
 $$\xymatrix{T^*\ar[r]^-{\tau^*\circ\ p_1^*}\ar[rd]_=&k[V]^{\times}/k^{\times}\ar[r]^{div_Y}\ar[d]&\Div_{Y\setminus V}(Y)&\Div_{X\setminus U}(X)\ar[l]_{\cong }\ar[d]\\
 &T^*\ar[rr]^{type}&&\Pic(X).
 }$$
 Puisque $type(1_{T^*})=[Y]$, le r\'esultat en d\'ecoule.
 
 \'Etape (2). Supposons $k=\bk$. Notons $n=\dim(T)$. Alors $T\cong \BG_m^n$, $T^*\cong \BZ^n$, $H^1(X,T)\cong \Pic(X)^{\oplus n}$ et $\Hom(T^*,\Div_{X\setminus U}(X))\cong \Div_{X\setminus U}(X)^{\oplus n}$.
 On se r\'eduit ainsi \`a l'\'etape (1).
 
 \'Etape (3). Supposons que $T$ est quasi-trivial et $X$ est projective.
  Dans ce cas, le morphisme $H^1(X,T)\to H^1(X_{\bk},T)$ est injectif (une cons\'equence de \cite[(2.0.2)]{CTS87}). 
  Par l'\'etape (2), pour chaque $\tau\in \Triv_U(V,T)$, on a $\Psi(\Upsilon (\tau))|_{\bk}=[Y_{\bk}]$, et donc $\Psi(\Upsilon (\tau)) =[Y]$.
 
 \'Etape (4). En g\'en\'eral, soit $X^c$ une compactification lisse de $X$ et soit $T_0$ un tore avec un isomorphisme $T_0^*\xrightarrow{\iota}\Div_{X_{\bk}\setminus U_{\bk}}(X_{\bk})$. 
 Notons $\sigma^*:=\iota^{-1}\circ \phi$ et $T_0\xrightarrow{\sigma}T$ le morphisme correspondant de tores.
L'\'egalit\'e $\Div_{X^c_{\bk}\setminus U_{\bk}}(X^c_{\bk})\cong \Div_{X^c_{\bk}\setminus X_{\bk}}(X^c_{\bk})\oplus \Div_{X_{\bk}\setminus U_{\bk}}(X_{\bk})$ donne des morphismes
$$\iota_c: T_0^*\xrightarrow{\iota}\Div_{X_{\bk}\setminus U_{\bk}}(X_{\bk})\hookrightarrow \Div_{X^c_{\bk}\setminus U_{\bk}}(X^c_{\bk})\ \ \ \text{et}\ \ \ \pi: \Div_{X^c_{\bk}\setminus U_{\bk}}(X^c_{\bk})\to \Div_{X_{\bk}\setminus U_{\bk}}(X_{\bk})$$
tels que $\pi\circ \iota_c=\iota $.
 Par la suite exacte (\ref{tor-e1}), on a un diagramme commutatif 
 $$\xymatrix{\Hom_k(T_0^*,\Div_{X^c_{\bk}\setminus U_{\bk}}(X^c_{\bk}))\ar[r]^{\pi\circ -}\ar[d]^{\Psi_c}&
  \Hom_k(T_0^*,\Div_{X_{\bk}\setminus U_{\bk}}(X_{\bk}))\ar[d]^-{\Psi_0}\ar[r]^{-\circ \sigma^*}&
  \Hom_k(T^*,\Div_{X_{\bk}\setminus U_{\bk}}(X_{\bk}))\ar[d]^-{\Psi}\\
  H^1(X^c,T_0)\ar[r]^{(-)|_X}&H^1(X,T_0)\ar[r]^{\sigma_*}&H^1(X,T).
 }$$
 Alors $\Psi_c(\iota_c)$ donne un $T_0$-torseur $Y_0^c$ sur $X^c$ tel que $\pi\circ \iota_c\circ \sigma^*=\phi $, $\sigma_*[Y_0^c|_X]=[Y]\in H^1(X,T)$ 
 et que $V_0:=Y_0^c|_U\cong T_0\times U$. Ceci donne un diagramme commutatif:
 $$\xymatrix{\Triv_U(V_0,T_0)\ar[d]^-{\Upsilon_c }\ar[r]^{=}&
  \Triv_U(V_0,T_0)\ar[d]^-{\Upsilon_0 }\ar[r]^{\sigma_*}&
 \Triv_U(V,T)\ar[d]^-{\Upsilon }\\
 \Hom_k(T_0^*,\Div_{X^c_{\bk}\setminus U_{\bk}}(X^c_{\bk}))\ar[r]^{\pi\circ -}&
  \Hom_k(T_0^*,\Div_{X_{\bk}\setminus U_{\bk}}(X_{\bk}))\ar[r]^{-\circ \sigma^*}&
  \Hom_k(T^*,\Div_{X_{\bk}\setminus U_{\bk}}(X_{\bk}))
 }$$
 et le r\'esultat d\'ecoule de l'\'etape (3).
 \end{proof}

\medskip

\subsection{L'action d'un groupe sur un torseur}\label{2.3}
Soit $S$ un groupe de type multiplicatif et soit $G$ un groupe lin\'eaire connexe.
Colliot-Th\'el\`ene a montr\'e que tout $S$-torseur $H$ sur $G$ peut \^etre muni d'une structure de groupe telle que $H\to G$ soit un homomorphisme de groupes,
s'il poss\`ede un point rationnel au-dessus de $e_G$ (cf. \cite[Thm. 5.6]{CT07}). On donne une g\'en\'eralisation de ce r\'esultat (Th\'eor\`eme \ref{thmaction}).

On commence par une g\'en\'eralisation de  \cite[Lem. 5.5]{CT07}.

\begin{lem}\label{lemaction}
Soient $X_1$, $X_2$ deux vari\'et\'es lisses g\'eom\'etriquement int\`egres.
Soit $S$ un groupe de type multiplicatif.
Supposons que $X_1$ est g\'eom\'etriquement rationnelle et $X_1(k)\neq \emptyset$.
Alors pour chaque $e\in X_1(k)$ et $i=0$ ou $1$, on a un isomorphisme canonique:
$$H^i_e(X_1,S)\oplus H^i(X_2,S)\iso H^i(X_1\times X_2,S)$$
o\`u $H^i_e(X_1,S):=\Ker (H^i(X_1,S)\xrightarrow{e^*} H^i(k,S))$.
\end{lem}

\begin{proof}
Si $S=\BG_m$, l'\'enonc\'e d\'ecoule de \cite[Lem. 6.5 et Lem. 6.6]{S}.
Si $S$ est quasi-trivial, i.e. il existe une $k$-alg\`ebre finie  \'etale $K$ tel que $S=\Res_{K/k}\BG_m$, 
  l'\'enonc\'e d\'ecoule du fait  $H^1(-,S)\cong H^1((-)_K,\BG_m)$.

En g\'en\'eral, on note $H_1^i(S):=H^i_e(X_1,S)\oplus H^i(X_2,S)$ et $H^i_2(S):=H^i(X_1\times X_2,S)$ pour $i=0$ ou $1$.
Il existe un tore $T$, un tore quasi-trivial $T_0$ et une suite exacte:
$$0\to S\to T_0\to T\to 0.$$
 Elle induit  un diagramme commutatif de suites exactes:
$$
\xymatrix{0\ar[r]&H^0_1(S)\ar[r]\ar[d]&H^0_1(T_0)\ar[r]\ar[d]^{\cong}&H^0_1(T)\ar[r]\ar[d]
&H^1_1(S)\ar[r]\ar[d]&H^1_1(T_0)\ar[r]\ar[d]^{\cong}&H^1_1(T)\ar[d]\\
0\ar[r]&H^0_2(S)\ar[r]&H^0_2(T_0)\ar[r]&H^0_2(T)\ar[r]
&H^1_2(S)\ar[r]&H^1_2(T_0)\ar[r]&H^1_2(T).
}$$
Le r\'esultat d\'ecoule du lemme des cinq.
\end{proof}

\begin{thm}\label{thmaction}
Soient $G$ un $k$-groupe lin\'eaire connexe et $(X,\rho )$ une $G$-vari\'et\'e lisse g{\'eom\'e}\-triquement int\`egre.
Soient $S$ un groupe de type multiplicatif et $Y\xrightarrow{p}X$ un $S$-torseur.
Alors il existe  un groupe lin\'eaire $H$ et un homomorphisme $ H \xrightarrow{\psi}G$  
 tels que $\psi$ soit surjectif de noyau $S$ central et que
l'action $\rho_S$ de $S$ en $Y$ s'\'etende  une action $\rho_H$ de $H$ sur $Y$ qui fait $p$  un $H$-morphisme.

De plus :

(1) le groupe $H$ est unique, i.e. si $H_1\xrightarrow{\psi_1}G$ satisfait les conditions ci-dessus,
 alors il existe un isomomorphisme 
de $k$-groupes $H_1\xrightarrow{\vartheta}H $ tel que l'on ait un diagramme commutatif:
\begin{equation}\label{diaactione1}
\xymatrix{1\ar[r]&S\ar[d]^=\ar[r]&H_1\ar[r]^{\psi_1}\ar[d]^{\vartheta}&G\ar[d]^=\ar[r]&1\\
1\ar[r]&S\ar[r]&H\ar[r]^{\psi}&G\ar[r]&1;
}
\end{equation}

(2) si $\rho^*[Y]=p_2^*[Y]$ dans $H^1(G\times X,S)$, alors $H\cong S\times G$;

(3) pour chaque choix d'une action $\rho_H$, chaque homomorphisme $G\xrightarrow{\phi}S$ induit une nouvelle action $H\times Y\xrightarrow{\rho_{H,\phi}}Y$ satisfaisant les conditions ci-dessus, o\`u
$$\rho_{H,\phi}(h,y)=(\phi\circ \psi)(h)\cdot \rho_H(h,y)\ \ \ \text{pour chaque}\ \ \ h\in H,\ y\in Y,$$
et toutes les actions satisfaisant les conditions ci-dessus sont obtenues de cette fa\c{c}on.
\end{thm}

\begin{proof}
On suit la d\'emonstration de Colliot-Th\'el\`ene \cite[Thm. 5.6]{CT07}. 
D'apr\`es le lemme \ref{lemaction}, pour chaque \'el\'ement $\alpha\in H^1(X,S)$, puisque $\rho^*(\alpha)|_{e_G\times X}=\alpha$, 
il existe un unique
$\beta\in H^1_{e_G}(G,S)$ tel que $\rho^*(\alpha)=p_2^*(\alpha)+p_1^*(\beta)$, o\`u $p_1,\ p_2$ sont les projections.
Si $\alpha=[Y]$, on note $\beta=(H\xrightarrow{\psi}G)$.  
Par \cite[Thm. 5.6 et Cor. 5.7]{CT07}, il existe une structure unique de $k$-groupe lin\'eaire  sur $H$ (\`a (\ref{diaactione1}) pr\`es)
telle que $\psi$ soit un homomorphisme de noyau $S$ central.

Notons $ S\xrightarrow{i} H$ l'immersion. L'\'egalit\'e $\rho^*[Y]=p_1^*[H]+p_2^*[Y]$ donne un diagramme commutatif:
\begin{equation}\label{diaaction}
\xymatrix{S\times Y\ar[r]_{i\times id_Y}\ar[d]\ar@{}[rd]|{\square}&H\times Y\ar[r]_{\rho_+}\ar[d]^{\psi\times p}\ar@/^1pc/[rr]^{\rho_H} &\rho^*Y\ar[r]_{\rho_Y}\ar[d]\ar@{}[rd]|{\square}&Y\ar[d]^p\\
e_G\times X\ar[r]&G\times X\ar[r]^=&G\times X\ar[r]^{\rho}&X,
}
\end{equation}
tel que $\rho_+$ induise un isomorphisme $H\times^SY\iso \rho^*Y$ et $\rho_H:=\rho_Y\circ \rho_+$ soit un $S\times S$-morphisme,
o\`u l'action $S\times S\curvearrowright Y: (s_1,s_2,y)\mapsto \rho_S(s_1\cdot s_2, y)$. 
Donc $\rho_H|_{e_S\times Y}\in \Hom_X(Y,Y)$ est un $S$-morphisme, et donc un isomorphisme. 
Rempla\c{c}ant $\rho_Y$ par $(\rho_H|_{e_S\times Y})^{-1}\circ \rho_Y$, on peut supposer que  $\rho_H|_{e_S\times Y}$ est l'identit\'e, et donc $\rho_S=\rho_H\circ (i\times id_Y)$.

Montrons que $\rho_H$ est une action. Notons $m_H$ la multiplication sur $H$. 
Puisque $\rho_G$ est une action, les morphismes $\rho_1:=\rho_H\circ (m_H\times id_Y)$ et $\rho_2:=\rho_H\circ (id_H\times \rho_H)$ induisent un morphisme
$$\Psi: H\times H\times Y\xrightarrow{(\rho_1,\rho_2)} Y\times_XY\iso Y\times S\xrightarrow{p_2} S,\ \ \ \text{i.e.}\ \ \  \rho_H(h_1\cdot h_2,y)=\Psi (h_1,h_2,y)\cdot \rho_H(h_1,\rho_H(h_2,y))$$
pour chaque $h_1, h_2\in H$ et $y\in Y$.
Puisque $\rho_H$ est un $S\times S$-morphisme, il existe un morphisme $G\times G\times X\xrightarrow{\Psi_1}S$ tel que 
$\Psi=\Psi_1\circ (\psi\times \psi\times p)$.
Par le lemme de Rosenlicht (voir \cite[Lem. 6.5]{S}), il existe des homomorphismes $\chi_1,\chi_2: G\to S$ et un morphisme $X\xrightarrow{\chi_0}S$ tels que 
$$\Psi_1(g_1,g_2,x)=\chi_1(g_1)\cdot \chi_2(g_2)\cdot \chi_0 (x)$$
pour chaque $g_1, g_2\in G$ et $x\in X$. Puisque $\rho_H(e_H,y)=y$ pour chaque $y\in Y$, on a $\chi_1=0$, $\chi_2=0$ et $\chi_0(x)=e_S$ pour chaque $x\in X$. Donc $\rho_H$ est une action.

Pour (1), s'il existe un groupe lin\'eaire $H$ qui satisfait l'\'enonc\'e du th\'eor\`eme, alors $H$ et l'action $\rho_H$ en $Y$ induisent le diagramme (\ref{diaaction}) tel que $\rho_H:=\rho_Y\circ \rho_+$  et $H\times^SY\iso \rho^*Y$. 
Alors dans $H^1(G\times X,S)$, on a $p_1^*[H]+p_2^*[Y]=\rho^*[Y]$, 
ce qui d\'etermine uniquement $H$ par l'argument ci-dessus. 

Pour (2), si $\rho^*[Y]=p_2^*[Y]$, on a $[H]=0$ et, d'apr\`es (1), on a $H\cong S\times G$.

Pour (3), il est clair que $\rho_{H,\phi}|_{S\times Y}=\rho_S$ et $p\circ \rho_{H,\phi}=\rho\circ (\psi\times p)$. 
Par ailleurs, soit $\rho_H'$ une action satisfaisant les conditions. 
Alors on a un morphisme 
$$\Phi: H\times Y\xrightarrow{(\rho_H,\rho_H')}Y\times_XY\iso Y\times S\xrightarrow{p_2}S\ \ \ \text{i.e.}\ \ \ 
\rho_H'(h,y)=\Phi (y,h)\cdot \rho_H(h,y)$$
pour chaque $h\in H$ et $y\in Y$. 
Puisque $\rho_H'|_{S\times Y}=\rho_S$, il existe $G\times X\xrightarrow{\Phi_1}S$ tel que $\Phi=\Phi_1\circ (\psi\times p)$.
Par le lemme de Rosenlicht (voir \cite[Lem. 6.5]{S}), il existe un homomorphisme $G\xrightarrow{\phi}S$ et un morphisme $X\xrightarrow{\chi}S$ tels que $\Phi_1(g,x)=\phi(g)\cdot \chi (x)$ pour chaque $g\in G$ et $x\in X$. 
Puisque $\rho_H'(e_H,y)=y$ pour chaque $y\in Y$, on a $\chi (x)=e_S$ pour tout $x\in X$. Donc $\rho_H'=\rho_{H,\phi}$.
\end{proof}

\begin{cor}\label{cor1action}
Soient $G$ un $k$-groupe lin\'eaire connexe 
et $X_1, $ $X_2$ deux $G$-vari\'et\'es lisses g{\'eom\'e}\-triquement int\`egres munies d'un $G$-morphisme $X_1\to X_2$.
Soit $S$ un groupe de type multiplicatif.
Pour $i=1,2$, soient $Y_i\to X_i$ un $S$-torseur et $H_i$ le groupe lin\'eaire donn\'e par le th\'eor\`eme \ref{thmaction}. 
Si $Y_1\cong Y_2\times_{X_2}X_1$ comme $S$-torseurs, alors $H_1\cong H_2$ 
et,  apr\`es avoir chang\'e l'action de $H_1$ sur $Y_1$ ou l'action de $H_2$ sur $Y_2$, 
on peut imposer que $Y_1\cong Y_2\times_{X_2}X_1\to Y_2$ soit un $H_1$-morphisme.
\end{cor}

\begin{proof}
L'action de $H_2$ sur $Y_2$ induit canoniquement une action de $H_2$ sur $Y_2\times_{X_2}X_1$ telle que $Y_2\times_{X_2}X_1\to X_1$ soit un $H_2$-morphisme. 
Par l'unicit\'e dans le th\'eor\`eme \ref{thmaction}, $H_1\cong H_2$.
Le r\'esultat d\'ecoule du fait que la diff\'erence de deux actions est un homomorphisme $G\to S$, qui ne d\'epend pas de $X_i$.
\end{proof}

\begin{cor}\label{coraction}
Sous les hypoth\`eses du th\'eor\`eme \ref{thmaction}, si le $S$-torseur $[Y]$ est trivial sur un $G$-ouvert $U$ de $X$, alors  $H\cong S\times G$.
\end{cor}

\begin{proof}
 Le r\'esultat d\'ecoule du th\'eor\`eme \ref{thmaction} (1) et (2).
\end{proof}

\begin{cor}\label{cor2action}
Sous les hypoth\`eses du th\'eor\`eme \ref{thmaction}, supposons que $X\cong G/G_0$, o\`u $G_0\sbt G$ est un sous-groupe ferm\'e connexe. 
Si $Y(k)\neq\emptyset$, alors il existe un sous-groupe connexe ferm\'e $H_0$ de $ H$ tel que $Y\cong H/H_0$ et $H_0\cong G_0$.
\end{cor}

\begin{proof}
Soient $y\in Y(k)$, $x:=p(y)$ et $G\xrightarrow{\pi}X$ le morphisme induit par $x$. 
Alors $Y_G:=Y\times_XG$ a un $k$-point sur $e_G$. 
Par \cite[Thm. 5.6]{CT07}, $Y_G$ est un groupe satisfaisant les conditions du th\'eor\`eme \ref{thmaction}. 
Par le corollaire \ref{cor1action}, $Y_G\cong H$ et $H\xrightarrow{\pi_Y}Y$ est un $H$-morphisme.
 Donc $H_0:=\pi_Y^{-1}(y)\cong \pi^{-1}(x)\cong G_0$.
\end{proof}

\section{Groupe de Brauer invariant}\label{3}

Dans toute cette section,  $k$ est un corps quelconque de caract\'eristique $0$. 
Sauf  mention explicite, une vari\'et\'e est une $k$-vari\'et\'e.
Soit $G$ un groupe alg\'ebrique et soit $X$ une $G$-vari\'et\'e lisse. 
On d\'efinit la notion  de sous-groupe $G$-invariant de $\Br(X)$ (D\'efinition \ref{def-invariant}). 
Ensuite on \'etablit ``l'alg\'ebricit\'e'' de $\Br_G(X)$ (Proposition \ref{propbraueralgebraic}).
On d\'efinit la notion  d'homomorphisme de Sansuc (cf. D\'efinition \ref{defsansuc}) et
on obtient un diagramme commutatif canonique de suites exactes de Sansuc (Th\'eor\`eme \ref{sansucthm}).

\subsection{D\'efinitions et propri\'et\'es}

\begin{defi}\label{def-invariant}
Soient $G$ un  groupe alg\'ebrique connexe et $(X,\rho )$ une $G$-vari\'et\'e lisse g\'eom\'e\-triquement int\`egre.  
\emph{Le sous-groupe de Brauer $G$-invariant} de $X$ est le sous-groupe 
\begin{equation}\label{defBr_GX}
\Br_G(X):=\{b\in \Br(X)\ :\ (\rho^*(b)-p_2^*(b))\in p_1^*\Br(G)\}
\end{equation}
de $\Br(X)$, o\`u $G\times X\xrightarrow{p_1}G$, $G\times X\xrightarrow{p_2}X$ sont les projections. 
\end{defi}

Dans la proposition \ref{prop-binvariant}, pour un sous-groupe $B\sbt \Br(X)$, on montre que $B\sbt \Br_G(X)$ si et seulement si $B$ est ``$G$-invariant".

 Le lemme suivant est bien connu.
 
\begin{lem}\label{brauertorseur}
Soient $X$, $Y$ deux vari\'et\'es lisses {g\'eom\'e}triquement int\`egres et $Y\xrightarrow{p}X$ un morphisme fid\`element plat \`a fibres g\'eom\'etriquement int\`egres. Soient $U$ un ouvert  non vide de $X$ et $V:=p^{-1}(U)$. 
Soit $B\sbt \Br(U)$ un sous-groupe.
Alors 
$$p^*(B\cap \Br(X))=(p^*B)\cap \Br(Y).$$
\end{lem}

\begin{prop}\label{prop-binvariant}
Sous les hypoth\`eses de la D\'efinition \ref{def-invariant}, pour un sous-groupe $B\sbt \Br(X)$, les \'enonc\'es ci-dessous sont \'equivalents:

(1) $B\sbt \Br_G(X)$;

(2) $\rho^*B+p_1^*\Br(G)=p_2^*B+p_1^*\Br(G)$;

(3) pour toute extension de corps $K/k$ et tout $g\in G(K)$, l'action $\rho_g: X_K\xrightarrow{g\cdot (-)}X_K$ induit 
un morphisme $ \Br(X_K)\xrightarrow{\rho_g^*}\Br(X_K)$ tel que
\begin{equation}\label{prop-binv-e}
\rho_g^*\pi^*B+\Im \Br(K)=\pi^*B+\Im \Br(K),
\end{equation}
o\`u $X_K\xrightarrow{\pi}X$;

(4) pour tout $b\in B$, toute extension de corps $K/k$  et tout  $g\in G(K)$, 
on a 
$$(\rho_g^*\pi^* (b)-\pi^*(b)) \in \Im\Br(K),$$
o\`u $\pi$ et $\rho_g^*$ sont d\'efinis dans (3);

(5) pour toute extension de corps $K/k$ telle que  $k$ soit alg\'ebriquement clos dans $K$, et tout $g\in G(K)$,
on a (\ref{prop-binv-e}),
o\`u $\pi$ et $\rho_g^*$ sont d\'efinis dans (3).
\end{prop}

\begin{proof}
Les implications $(4)\Rightarrow (3)$, $(3)\Rightarrow (5)$ sont claires.

Pour $(2)\Rightarrow (1)$, on note $i_e:X\xrightarrow{e_G\times id_X}G\times X$ l'immersion ferm\'ee.
Pour $b\in B$, il existe $b'\in B$ et $a\in \Br(G)$ tels que $\rho^*(b)=p_2^*(b')+p_1^*(a)$.
Puisque $\rho\circ i_e=p_2\circ i_e=id_X$ et  que $p_1\circ i_e$ se factorise par $\Spec\ k$, 
on a $b-b'\in \Im\Br(k) $ et donc 
$$(\rho^*(b)-p_2^*(b))\in (p_1^*\Br(G)+\Im\Br(k))=p_1^*\Br(G).$$

Pour $(1)\Rightarrow (4)$, on note $i_g:X\xrightarrow{g\times id_X}G\times X$ le morphisme et $A:=\Im \Br(K)$.
Alors $\pi\circ \rho_g=\rho\circ i_g $ et $\pi=p_2\circ i_g $.
Puisque $i_g^*(p_1^*\Br(G))\sbt A$, on a
$$\rho_g^*\pi^* (b)+A=i_g^*(p_1^*\Br(G)+\rho^*(b))+A=i_g^*(p_1^*\Br(G)+p_2^*(b))+A=\pi^*(b)+A.$$

Pour $(5)\Rightarrow (2)$, notons $X_{\eta_G}\xrightarrow{\pi}X$ la projection et $X_{\eta_G}\xrightarrow{i_{\eta}}G\times X$ l'immersion canonique. Alors $\Br(G\times X)\to \Br(X_{\eta_G})$ est injectif.
Par le lemme \ref{brauertorseur}, $p_1^*\Br(\eta_G)\cap \Br(G\times X)=p_1^*\Br(G)$.
Donc il suffit de montrer que:
$$(i_{\eta}^*\rho^*)B+\Im \Br(\eta_G)=(i_{\eta}^*p_2^*)B+ \Im \Br(\eta_G).$$
Le r\'esultat d\'ecoule de $p_2\circ i_{\eta}=\pi$ et $\rho\circ i_{\eta}=\rho_{\eta_G}\circ \pi$.
\end{proof}

\begin{prop}  \label{prop-invariant}
Sous les hypoth\`eses de la D\'efinition \ref{def-invariant}, alors:

(1) pour tout groupe alg\'ebrique connexe $G_0$ muni d'un homomorphisme $G_0\xrightarrow{\phi}G$,  
on a $\Br_G(X)\sbt \Br_{G_0}(X)$;

(2) pour toute $G$-vari\'et\'e $Y$ lisse g\'eom\'etriquement int\`egre munie d'un $G$-morphisme $Y\xrightarrow{p}X$,  
on a $p^*\Br_G(X)\sbt \Br_G(Y)$, o\`u $\Br(X)\xrightarrow{p^*}\Br(Y)$;

(3) pour tout $G$-ouvert dense $U\sbt X$, on a $\Br_G(X)=\Br_G(U)\cap \Br(X)$;

(4) si $G$ est lin\'eaire, on a $\Br_1(X)\sbt \Br_G(X)$;

(5) pour toute $G$-vari\'et\'e $Y$ munie d'un $G$-morphisme $Y\xrightarrow{p}X$, 
si $p$ est un torseur sous un groupe lin\'eaire connexe $H$,
on a $(p^*)^{-1}\Br_G(Y)\sbt \Br_G(X)$, o\`u $\Br(X)\xrightarrow{p^*}\Br(Y)$;

(6) sous les hypoth\`eses de (5),  si $G$ est lin\'eaire, on a
$$\Br_1(X,Y):=\Ker(\Br(X)\to \Br(Y_{\bk}))\sbt \Br_G(X);$$
\end{prop}

\begin{proof}
Les \'enonc\'es (1), (2) et (3) d\'ecoulent de la d\'efinition.

Pour (4),  par \cite[Lem. 6.6]{S}, $\Br_1(G\times X)=\Br_a(G)\oplus \Br_1(X)$, et donc
$$\Br_1(G\times X)=p_2^*\Br_1(X)+p_1^*\Br_1(G)=\rho^*\Br_1(X)+p_1^*\Br_1(G).$$
Le r\'esultat  d\'ecoule de la proposition \ref{prop-binvariant}.

Pour (5), puisque $G\times Y\xrightarrow{id_G\times p}G\times X$ est aussi un $H$-torseur, par la suite exacte de Sansuc (\cite[Prop. 6.10]{S}), on a un diagramme commutatif de suites exactes
$$\xymatrix{\Pic(H)\ar[r]\ar[d]^=&\Br(X)\ar[r]^{p^*}\ar[d]^{p_2^*}&\Br(Y)\ar[d]^{p_{2,Y}^*}\\
\Pic(H)\ar[r]&\Br(G\times X)\ar[r]^-{(id_G\times p)^*}&\Br(G\times Y).
}$$
 Donc, pour tout $\alpha\in (p^*)^{-1}\Br_G(Y)$, on a
$$\rho^*(\alpha)-p_2^*(\alpha)\in p_1^*\Br(G)+\Im \Pic(H)\sbt p_1^*\Br(G)+p_2^*\Br(X) .$$ 
D'apr\`es calculer sur $X\xrightarrow{e_G\times id_X}G\times X$, on peut voir que $\rho^*(\alpha)-p_2^*(\alpha)\in p_1^*\Br(G)$.

Par (4) et (5), $\Br_1(X,Y)\sbt \Br_G(X)$.
\end{proof}

Soient $X$ une vari\'et\'e lisse g\'eom\'etriquement int\`egre et $U\sbt X$ un ouvert non vide. 
Supposons que $D:=(X\setminus U)$ est lisse de codimension $1$. Par le th\'eor\`eme de puret\'e pour la cohomologie \'etale
\`a support  dans un ferm\'e lisse (cf. \cite[\S VI.5]{Mi80}),
on a une suite exacte:
$$H^2(X,\BQ/\BZ(1))\to H^2(U,\BQ/\BZ(1))\to H^1(D,\BQ/\BZ)\to H^3(X,\BQ/\BZ (1))\to H^3(U,\BQ/\BZ (1)).$$
Puisque $\Pic(X)\to \Pic(U)$ est surjectif et $\Br(X)\to \Br(U)$ est injectif, 
d'apr\`es la suite exacte de Kummer, on a la suite exacte (cf. Grothendieck  \cite{Gro})
\begin{equation}\label{purity}
0\to \Br(X)\to \Br(U)\xrightarrow{\partial}H^1(D,\BQ/\BZ)\to H^3(X,\BQ/\BZ (1))\to H^3(U,\BQ/\BZ (1)).
\end{equation}

Soit $G$ un groupe alg\'ebrique connexe. 
Si  $X$ est munie d'une $G$-action $\rho : G \times X \to X$ respectant $U$, on a un diagramme de suites exactes:
$$\xymatrix{0\ar[r]&\Br(X)\ar[r]\ar[d]^{\theta^*}&\Br(U)\ar[r]^-{\partial}\ar[d]^{\theta_U^*}&H^1(D,\BQ/\BZ)\ar[d]^{\theta_D^*}\\
0\ar[r]&\Br(G\times X)\ar[r]^{\Res}&\Br(G\times U)\ar[r]&H^1(G\times D,\BQ/\BZ),
}$$
o\`u $G\times X\xrightarrow{\theta}X$ est soit $p_2$  soit $\rho$.
 On a:

\begin{lem}\label{lem-fibreinvariant}
Avec les notations et hypoth\`eses ci-dessus, on a: 

(1)  $p_{2,D}^*|_{\partial (\Br_G(U))}=\rho^*_D |_{\partial (\Br_G(U))} $;

(2) pour tout $b\in \Br_G(U)$, il existe un rev\^etement fini \'etale galoisien ab\'elien $D'\xrightarrow{\pi}D$ tel que 
$D'$ soit une $G$-vari\'et\'e, $\pi$ soit un $G$-morphisme et que $\pi^*(\partial (b))=0\in H^1(D',\BQ/\BZ)$.
\end{lem}

\begin{proof}
Puisque $p_{1,U}^*\Br(G)\sbt \Im (\Res) $, l'\'enonc\'e (1) d\'ecoule de la proposition \ref{prop-binvariant} (2).
Pour $b\in \Br_G(U)$, soit $n\in \BZ$ l'ordre de $\partial (b)$.
Alors $\partial (b)\in H^1(D,\BZ/n)$. 
Soit $D'\xrightarrow{\pi}D$ un $\BZ/n$-torseur tel que $[D']=\partial (b)$.
Par (1), $p_{2,D}^*[D']=\rho^*_D[D']\in H^1(G\times D,\BZ/n)$.
D'apr\`es le th\'eor\`eme \ref{thmaction} (2), $D'$ est une $G$-vari\'et\'e.
\end{proof}

\begin{lem}\label{lembrauerkbar}
 Soit $G$ un groupe lin\'eaire connexe.
 Alors $\Br_G(G)=\Br_1(G)$.
\end{lem}

\begin{proof}
D'apr\`es la proposition \ref{prop-invariant} (2) et (4), il suffit de montrer que $\Br_G(G_{\bk})=0$ et on peut supposer  $k=\bk$.

Si $G$ est un tore de dimension $n$, alors $G\cong \BG_m^n$ est un ouvert de $ \BA^n$ canoniquement.
Soit $X:=\BA^n\setminus [(\BA^n\setminus G)_{sing}]$.
Alors $\codim (\BA^n\setminus X,\BA^n)\geq 2$, $\Br(X)=0$, $G$ est un ouvert de $X$ et $X\setminus G=\sqcup_{i=1}^nD_i$, chaque   $D_i$ \'etant  un $G$-espace homog\`ene de stabilisateur $\BG_m$.
D'apr\`es \cite[Prop. 2.2]{CX1}, pour tout  rev\^etement fini \'etale galoisien $D'_i\xrightarrow{\pi_i}D_i$ tel que $D'_i$ soit une $G$-vari\'et\'e int\`egre et que $\pi_i$ soit un $G$-morphisme,  le morphisme $\pi_i$ est un isomorphisme.
D'apr\`es le lemme \ref{lem-fibreinvariant} (2) et (\ref{purity}),  $\Br_G(G)\sbt \Br(X)=0$.

Si $G$ est r\'eductif, soit $T$ le tore maximal de $G$.
Par la d\'ecomposition de Bruhat, il existe un ouvert $U$ de $G$ tel que $U\cong \BA^{n}\times T\times \BA^{n}$, o\`u $2n=\dim (U)-\dim (T)$ (voir la d\'emonstration de \cite[Prop. 4.2]{CTb}). Donc $\Br(G)\to \Br(T)$ est injectif.
Le r\'esultat d\'ecoule de la proposition \ref{prop-invariant} (2). 

En g\'en\'eral, par \cite[Lem. 2.1]{CX2}, le morphisme $\Br(G^{red})\to \Br(G)$ est un isomorphisme.  
Le r\'esultat d\'ecoule de la proposition \ref{prop-invariant} (2). 
\end{proof}

Rappelons la d\'efinition de  $\Br_e(G)$ (cf. (\ref{Bredef})).

\begin{prop}\label{propbraueralgebraic}
Soient $G$ un  groupe lin\'eaire connexe et $(X,\rho )$ une $G$-vari\'et\'e lisse g\'eom\'e\-triquement int\`egre.  
 Pour tout $b\in \Br_G(X)$, on a 
$(\rho^*(b)-p_2^*(b))\in p_1^*\Br_{e}(G).$
\end{prop}

\begin{proof}
Notons $X\xrightarrow{e_G\times id_X}G\times X$. 
Puisque $p_2\circ (e_G\times id_X)=\rho\circ (e_G\times id_X)$, il suffit de montrer que, pour tout $b\in \Br_G(X)$, on a 
$(\rho^*(b)-p_2^*(b))\in p_1^*\Br_1(G)$.

On peut supposer  $k=\bk$. 
Un point $x\in X(k)$ induit un morphisme $G\xrightarrow{i_x} X$. 
Notons $m$ la multiplication sur $G$.
Alors on a un diagramme commutatif:
$$\xymatrix{\Br(X)\ar[r]^-{p_2^*-\rho^*}\ar[d]^{i_x^*}&\Br(G\times X)\ar[d]^{(id_G\times i_x)^*}&\Br(G)\ar[d]^=\ar[l]_{p_1^*}\\
\Br(G)\ar[r]_-{p_2^*-m^*}&\Br(G\times G)&\Br(G)\ar[l]^{p_1^*}.
}$$
Le r\'esultat d\'ecoule du lemme \ref{lembrauerkbar} et de l'injectivit\'e de $p_1^*$.
\end{proof}

\subsection{L'homomorphisme de Sansuc}

Soient $G$ un  groupe alg\'ebrique connexe et $(X,\rho )$ une $G$-vari\'et\'e lisse g\'eom\'e\-triquement int\`egre.  
Notons $G\times X\xrightarrow{p_1}G$, $G\times X\xrightarrow{p_2}X$ les deux projections.
Soit $\Br_G(X)$ le sous-groupe de Brauer $G$-invariant de $X$ (D\'efinition \ref{def-invariant}).
Rappelons la d\'efinition de  $\Br_e(G)$ (cf. (\ref{Bredef})).

Par \cite[Lem. 6.6]{S}, on a $\Br_a(G\times X)=\Br_a(G)\oplus \Br_a(X)$, et donc $p_1^*|_{\Br_e(G)}$ est injectif.
Par la proposition \ref{propbraueralgebraic}, il existe un unique homomorphisme $\Br_G(X)\xrightarrow{\lambda}\Br_e(G) $ tel que 
\begin{equation}\label{def-sansuc-e}
p_1^* \circ \lambda=\rho^*-p_2^*: \Br_G(X)\to \Br(G\times X).
\end{equation}

\begin{defi}\label{defsansuc}
Soient $G$ un  groupe lin\'eaire connexe et $(X,\rho )$ une $G$-vari\'et\'e lisse g\'eom\'e\-triquement int\`egre. 
L'unique homomorphisme $\Br_G(X)\xrightarrow{\lambda}\Br_e(G) $ satisfaisant (\ref{def-sansuc-e})
est appel\'e \emph{l'homomorphisme de Sansuc}.
\end{defi}

\begin{prop}\label{corbraueralgebraic1}
Sous les hypoth\`eses de la d\'efinition \ref{defsansuc}, on a:

(1) pour toute extension de corps $K/k$, et tous $x\in X(K)$, $g\in G(K)$, $\alpha\in \Br_G(X)$, on a
$$(g\cdot x)^*(\alpha )= g^*(\lambda (\alpha))+x^*(\alpha)\in \Br(K) ;$$

(2) si $X(k)\neq\emptyset$, alors pour tout $x\in X(k)$, on a
\begin{equation}\label{defsansuc-e}
 \lambda=\rho_x^*-x^*,
 \end{equation} 
 o\`u $G\xrightarrow{\rho_x} X: g\mapsto g\cdot x$, $\Spec\ k\xrightarrow{x}X$ est le point $x$
  et $\Br(X)\xrightarrow{x^*}\Br(k)\sbt \Br(G)$.

(3) si $X\cong G/G_0$ avec $G_0\sbt G$ un sous-groupe ferm\'e connexe, alors 
$$\Br_G(X)\cong \Br_1(X,G):=\Ker (\Br(X)\to \Br(G_{\bk})).$$
\end{prop}

\begin{proof}
Pour (1), on a 
$$(g\cdot x)^*(\alpha )=(g,x)^*(\rho^*(\alpha))=(g,x)^*(p_2^*(\alpha)+p_1^*(\lambda(\alpha)))= g^*(\lambda (\alpha))+x^*(\alpha)\in \Br(K).$$

Dans le cas (2), pour tout $\alpha\in \Br_G(X)$, on obtient
 $ (\rho^*-p_2^*)(\alpha)|_{G\times x}=(\rho_x^*-x^*)(\alpha)$.

L'\'enonc\'e (3) r\'esulte de la proposition \ref{prop-invariant} (2) (5) et  du lemme \ref{lembrauerkbar}.
\end{proof}

 Si $X\xrightarrow{f}Z$ est un $G$-torseur, par la d\'efinition, l'homomorphisme de Sansuc $\lambda|_{\Br_1(X)}$ est exactement le morphisme dans la suite exacte de Sansuc (cf: \cite[Thm. 2.8]{BD}). 

\begin{thm}\label{sansucthm}
Soient $G$ un groupe lin\'eaire connexe, $Z$ une vari\'et\'e lisse g\'eom\'etriquement int\`egre et $X\xrightarrow{f}Z$ un $G$-torseur. 
 Notons $G\times X\xrightarrow{\rho}X$ l'action de $G$.
 Alors l'homomorphisme de Sansuc $\lambda$ induit un diagramme commutatif de suites exactes:
\begin{equation}\label{braueralg-e1}
 \xymatrix{\Pic(G)\ar[r]\ar[d]^=&\Br(Z)\ar[r]^{f^*}\ar[d]^=&\Br_G(X)\ar[r]^-{\lambda}\ar@{^{(}->}[d]\ar@{}[rd]|{\square} &\Br_e(G)\cong \Br_a(G)\ar@{^{(}->}[d]^{p_1^*}\\
 \Pic(G)\ar[r]&\Br(Z)\ar[r]&\Br(X)\ar[r]^-{\rho^*-p_2^*}&\Br(G\times X).
 }
  \end{equation}
  \end{thm}
 
\begin{proof}
D'apr\`es  Borovoi et Demarche \cite[Thm. 2.8]{BD}, on a une suite exacte
 $$\Pic(G)\to \Br(Z)\xrightarrow{f^*} \Br(X)\xrightarrow{\rho^*-p_2^*}\Br(G\times X)$$
telle que $(\rho^*-p_2^*)(\Br_1(X))\sbt p_1^*\Br_e(G)$.
Le r\'esultat d\'ecoule de la proposition \ref{propbraueralgebraic}.
\end{proof}

\begin{cor}\label{corbraueralgebraic}
Soient $1\to N\to H\xrightarrow{\psi}G\to 1$ une suite de groupes lin\'eaires connexes et
$(X,\rho )$ une $G$-vari\'et\'e lisse g\'eom\'e\-triquement int\`egre.  

(1) On a $\Br_G(X)=\Br_H(X)$.

(2) S'il existe une $H$-vari\'et\'e $Y$  et un $H$-morphisme $Y\xrightarrow{p}X$ tels que $Y\to X $ soit un $N$-torseur,
alors $\Br(X)\xrightarrow{p^*}\Br(Y)$ satisfait $p^*\Br_G(X)\sbt \Br_H(Y)$, $(p^*)^{-1}\Br_H(Y)=\Br_G(X)$ 
et on a  une suite exacte (o\`u $\lambda$ est l'homomorphisme de Sansuc):
$$ \Pic(N)\to \Br_G(X)\xrightarrow{p^*} \Br_H(Y)\xrightarrow{\lambda}\Br_a(N).$$
\end{cor}

\begin{proof}
 On a $\Br_1(H)\sbt \Br_N(H) $ et $H\times X\xrightarrow{\psi_X} G\times X$ est un $N$-torseur.
Par la suite exacte de Sansuc (\cite[Prop. 6.10]{S}) et le diagramme (\ref{braueralg-e1}),  on a un diagramme commutatif de suites exactes:
$$\xymatrix{\Pic(N)\ar[d]^=\ar[r]&\Br_1(G)\ar[r]\ar[d]^{p_1^*}&\Br_1(H)\ar[r]\ar[d]^{p_1^*}&\Br_a(N)\ar[d]^=\\
\Pic(N)\ar[r]&\Br(G\times X)\ar[r]^{\psi_X^*}&\Br_N(H\times X)\ar[r]&\Br_a(N)
}$$
 Donc $ (\psi_X^*)^{-1}(p_1^*\Br_{e}(H))=p_1^*\Br_{e}(G)$ et on obtient (1).
 
Une application de la proposition \ref{prop-invariant} (5) et du diagramme (\ref{braueralg-e1}) donne (2).
\end{proof}

\begin{lem}\label{Dlemsuite1}
Soient $G$, $N$ deux groupes lin\'eaires connexes et
$X$ une $G$-vari\'et\'e lisse g\'eom\'e\-triquement int\`egre.  
Soient $H:=N\times G$ et $P$ une $H$-vari\'et\'e telle que $P$ soit un $N$-torseur sur $k$.
Soient $Y\iso P\times X$ et $Y\xrightarrow{p_1}P$, $Y\xrightarrow{p_2}X$  les deux projections.
 Si $P(k)\neq\emptyset$ ou $H^3(k,\bk^{\times})=0$,  on a un isomorphisme:
$$\Br_a(P)\oplus \Br_G(X)/\Im\Br(k) \xrightarrow{(p_1^*,p_2^*)}\Br_{H}(Y)/\Im\Br(k).  $$

De plus, cet isomorphisme induit un isomorphisme:
$\Br_a(P)\oplus \Br_a(X) \xrightarrow{(p_1^*,p_2^*)}\Br_a(Y).$
\end{lem}

\begin{proof}
Par la proposition \ref{prop-invariant} (1) et le lemme \ref{lembrauerkbar}, 
on a $\Br_{H}(P)=\Br_1(P)$.
Par la suite exacte de Sansuc (\cite[Prop. 6.10]{S}) et le corollaire \ref{corbraueralgebraic} (2), on a un diagramme commutatif de suites exactes
$$\xymatrix{\Pic(P)\ar[r]^{\vartheta_1}\ar[d]&\Pic(N)\ar[d]^=\ar[r]&\Br(k)\ar[r]\ar[d] 
&\Br_1(P)\ar[d]^{p_1^*}\ar[r]^{\vartheta_2}&\Br_a(N)\ar[d]^=\\
\Pic(Y)\ar[r]&\Pic(N) \ar[r]&\Br_{H}(X)\ar[r]^-{p_2^*}&
 \Br_{H}(Y)\ar[r] &\Br_a(N).
}$$
Par \cite[Lem. 6.7 et 6.8]{S}, $\vartheta_1 $ et $\vartheta_2$ sont surjectifs.
Alors on a une suite exacte 
\begin{equation}\label{Dlemsuite1-e}
0\to \Br(k)\to \Br_1(P)\oplus \Br_H(X) \xrightarrow{(p_1^*,p_2^*)} \Br_{H}(Y)\to 0. 
\end{equation}
Puisque le morphisme $\Br(X_{\bk})\to \Br((P\times X)_{\bk})\cong \Br(Y_{\bk})$ est injectif, on a une suite exacte:
$$
0\to \Br(k)\to \Br_1(P)\oplus \Br_1(X) \xrightarrow{(p_1^*,p_2^*)} \Br_1( Y)\to 0. 
$$
Le r\'esultat en d\'ecoule.
\end{proof}

\begin{prop}\label{propbrauersuj}
Soient $T$ un tore et $1\to G_0\to G\xrightarrow{\psi}T\to 1$ une suite exacte de groupes lin\'eaires connexes.
Soient $X$ une $G$-vari\'et\'e lisse, g\'eom\'e\-triquement int\`egre et $X\xrightarrow{f}T$ un $G$-morphisme.
Notons $\Br_a(G)\xrightarrow{\vartheta}\Br_a(G_0)$ l'homomorphisme induit par $G_0\sbt G$.
Alors, pour tout $t\in T(k)$, la fibre $X_t$ est $G_0$-invariante et 
on a un isomorphisme naturelle $\Pic(X_{\bk})\cong \Pic(X_{t,\bk})$ 
et deux suites exactes naturelles 
$$0\to T^*\xrightarrow{f^*}\bk[X]^{\times}/\bk^{\times}\to \bk[X_t]^{\times}/\bk^{\times}\to 0\ \ \ \text{et}\ \ \ 
\Br_e(T)\to \Br_G(X)\to \Br_{G_0}(X_t)\to \coker (\vartheta).$$
\end{prop}

\begin{proof}
D'apr\`es \cite[Prop. 2.2]{CX1}, $X_t$ est lisse, g\'eom\'etriquement int\`egre.
Notons:
 $$X_t\xrightarrow{i}G\times X_t: x\mapsto (e_G,x)\ \ \ \text{et}\ \ \  G\times X_t\xrightarrow{\rho}X: (g,x)\mapsto g\cdot x .$$
Alors $\rho\circ i $ est l'immersion $X_t\sbt X$. On fixe des actions
$$G\times G_0\curvearrowright G\times X_t: (g,g_0)\times (g',x)\mapsto (gg'g_0^{-1},g_0\cdot x)\ \ \ \text{et}\ \ \ 
G\times G_0\curvearrowright G: (g,g_0)\times g'\mapsto gg'g_0^{-1}. $$
Par d\'efinition, $X\cong G\times^{G_0}X_t$ et on a un diagramme commutatif de $G\times G_0$-morphismes
$$\xymatrix{X_t&G\times X_t\ar[l]^-{p_2}\ar[r]_-{p_1}\ar[d]^{\rho}&G\ar[d]^{t\cdot \psi (-)}\\
&X\ar[r]^f&T
}$$
tel que les colonnes soient  des $G_0$-torseurs. 

D'apr\`es \cite[Lem. 6.5 et Lem. 6.6]{S}, on a 
 $$\bk[G\times X_t]^{\times}/\bk^{\times}\cong \bk[G]^{\times}/\bk^{\times}\oplus \bk[X_t]^{\times}/\bk^{\times}
 \ \ \ \text{et}\ \ \ \Pic(G_{\bk}\times X_{t,\bk})\cong  \Pic(G_{\bk})\oplus \Pic(X_{t,\bk}).$$
Par la suite exacte de Sansuc \cite[Prop. 6.10 et Cor. 6.11]{S}, on a un diagramme commutatif de suites exactes:
$$\xymatrixcolsep{1.8pc}
\xymatrix{0\ar[r]&T^*\ar[r]\ar[d]^{f^*}&\frac{\bk[G]^{\times}}{\bk^{\times}}\ar[r]\ar[d]&
\frac{\bk[G_0]^{\times}}{\bk^{\times}}\ar[r]\ar[d]^=& \Pic(T_{\bk})\ar[r]\ar[d]&
\Pic(G_{\bk})\ar[r]\ar[d]&\Pic(G_{0,\bk})\ar[r]\ar[d]^=&0\\
0\ar[r]&\frac{\bk[X]^{\times}}{\bk^{\times}}\ar[r]&\frac{\bk[G\times X_t]^{\times}}{\bk^{\times}}\ar[r]&
\frac{\bk[G_0]^{\times}}{\bk^{\times}}\ar[r]&\Pic(X_{\bk})\ar[r]&\Pic((G\times X_t)_{\bk})\ar[r]&\Pic(G_{0,\bk})&.
}$$ 
Puisque $\Pic(T_{\bk})=0 $, une application du lemme du serpent donne l'isomorphisme et la premi\`ere suite exacte de l'\'enonc\'e.

 D'apr\`es (\ref{Dlemsuite1-e}), on a un isomorphisme:
 $\Br_e(G)\oplus \Br_{G_0}(X_t)\xrightarrow{(p_1^*,p_2^*)}\Br_{G\times G_0}(G\times X_t) $.
Le corollaire \ref{corbraueralgebraic} donne un diagramme commutatif de suites exactes:
$$\xymatrix{\Pic(G_0)\ar[r]\ar[d]^=&\Br_e(T)\ar[r]\ar[d]&\Br_e(G)\ar[r]^{\vartheta}\ar[d]^{p_1^*}&\Br_a(G_0)\ar[d]^=\\
\Pic(G_0)\ar[r]&\Br_G(X)\ar[r]^-{\rho^*}&\Br_{G\times G_0}(G\times X_t) \ar[r] \ar[d]^{i^*}&\Br_a(G_0)\\
&&\Br(X_t) &.
}$$
Puisque $ p_2\circ i=id$ et $i^*\circ p_1^*=0$, on a $\Br_{G_0}(X_t)= \Im (i^*)\cong \coker(p_1^*)$.
Une chasse au diagramme donne l'\'enonc\'e.
\end{proof}

\begin{cor}\label{prop-fibreinvariant}
Sous les hypoth\`eses de la proposition \ref{propbrauersuj}, soient $U\sbt X$ un $G$-ouvert et $B\sbt \Br_G(U)$ un sous-groupe. 
 Alors,  pour tout  $t\in T(k)$, de fibre $U_t\sbt X_t\stackrel{i_t}{\hookrightarrow}X$,
 on a :
 $$i_t^*(B\cap \Br(X))=(i_t^*(B)\cap \Br(X_t)).$$
\end{cor}

\begin{proof}
D'apr\`es la proposition \ref{propbrauersuj}, on a un diagramme de suites exactes:
$$\xymatrix{\Br_e(T)\ar[r]\ar[d]^=&\Br_G(X)\ar[r]\ar@{^{(}->}[d]&\Br_{G_0}(X_t)\ar[r]\ar@{^{(}->}[d]&\coker (\vartheta)\ar[d]^=\\
\Br_e(T)\ar[r]&\Br_G(U)\ar[r]^{i_t^*}&\Br_{G_0}(U_t)\ar[r]&\coker (\vartheta).
}$$
Une chasse au diagramme donne l'\'enonc\'e.
\end{proof}

\subsection{Pseudo espace homog\`ene}

Soit $G$ un groupe lin\'eaire connexe.
 La notion de pseudo $G$-espace homog\`ene g\'en\'eralise la notion de $G$-espace homog\`ene \`a stabilisateur g\'eom\'etrique connexe (cf. exemple \ref{examtoretorseur} (1)).

\begin{defi}\label{defpseudo}
Soit $G$ un groupe lin\'eaire connexe. Une $G$-vari\'et\'e $Z$ est appel\'ee \emph{pseudo $G$-espace homog\`ene} si $Z$ est lisse, g\'eom\'e\-triquement int\`egre,  $\Pic(Z_{\bk})$ est de type fini, $Z(k)\neq\emptyset$
 et $\frac{\Ker (\lambda )}{\Br(k)}, \Br_{G_{\bk}}(Z_{\bk})$ sont finis, o\`u $\Br_G(Z)\xrightarrow{\lambda}\Br_e(G)$ est l'homomorphisme de Sansuc (cf. D\'efinition \ref{defsansuc}).
\end{defi}

Fixons $z\in Z(k)$ et notons $\rho_z: G\to Z: g\mapsto g\cdot z$.
D'apr\`es (\ref{defsansuc-e}), le groupe $\frac{\Ker (\lambda )}{\Br(k)}$ est fini 
si et seulement si  $\Ker (\rho_z^*)\cap \Br_G(Z)$ est fini, o\`u $\Br(Z)\xrightarrow{\rho_z^*}\Br(G) $.

\begin{exam}\label{examtoretorseur}
Soit $G$ un groupe lin\'eaire connexe.

(1) Soit $G_0\sbt G$ un sous-groupe ferm\'e connexe. Alors $G/G_0$ est un pseudo $G$-espace homog\`ene.

(2) Soient $X$ une vari\'et\'e lisse g\'eom\'etriquement int\`egre et $Z\to X$ un $G$-torseur.
Si $\Pic(X)$ est de type fini, $\Br(X)/\Br(k), \Br(X_{\bk})$ sont finis et $Z(k)\neq \emptyset$,  alors $Z$ est un pseudo $G$-espace homog\`ene.
\end{exam}

\begin{proof}
L'\'enonc\'e (1) r\'esulte de \cite[Thm. 2.8]{BD} et de la proposition \ref{corbraueralgebraic1} (3).
L'\'enonc\'e (2) r\'esulte du corollaire \ref{corbraueralgebraic} (2).
\end{proof}

\begin{prop}\label{lemtoretorseur}
Soient $G$ un groupe lin\'eaire connexe et $Z$ un pseudo $G$-espace homog\`ene.
Soient $T$ un tore, $ Z'\to Z$ un $T$-torseur et $H$ le groupe lin\'eaire connexe d\'etermin\'e dans le th\'eor\`eme \ref{thmaction}.
Si $Z'(k)\neq\emptyset $, alors $Z'$ est un pseudo $H$-espace homog\`ene.
\end{prop}

\begin{proof}
Fixons $z'\in Z'(k)$ et $z\in Z(k)$ l'image de $z$.
Notons $\rho_{z'}: H\to Z': h\mapsto h\cdot z' $ et $\rho_z: G\to Z: g\mapsto g\cdot z$.
D'apr\`es la suite exacte de Sansuc \cite[Prop. 6.10]{S}, $\Pic({Z'}_{\bk})$ est de type fini.
Par le corollaire \ref{corbraueralgebraic} (2), on a un diagramme commutatif de suites exactes:
$$\xymatrix{\Pic(T)\ar[r]\ar[d]^=&\Br_G(Z)\ar[r]\ar[d]^{\rho_z^*}&\Br_H(Z')\ar[r]\ar[d]^{\rho_{z'}^*}&\Br_a(T)\ar[d]^=\\
\Pic(T)\ar[r]&\Br_1(G)\ar[r]&\Br_1(H)\ar[r]&\Br_a(T),
}$$
et un isomorphisme  $\Br_{G_{\bk}}(Z_{\bk})\to \Br_{H_{\bk}}(Z'_{\bk}) $.
Ainsi $\Br_{H_{\bk}}(Z'_{\bk}) $ est fini.
Puisque $\Ker (\rho_z^*)$ est fini, le groupe $\Ker (\rho_{z'}^*)$ est fini et $Z'$ est un pseudo $H$-espace homog\`ene.
\end{proof}

\begin{prop}\label{lembrauerinv-braueralg}
Soient $G$ un groupe lin\'eaire connexe et $Z$ un pseudo $G$-espace homog\`ene.
Alors $\Br_G(Z)/\Br_1(Z)$ est fini.
\end{prop}

\begin{proof}
Ceci vaut car $\Br_G(Z)/\Br_1(Z)$ est un sous-groupe de $\Br_{G_{\bk}}(Z_{\bk})$.
\end{proof}

 Pour un corps de nombres, le lemme suivant est bien connu.
 
 \begin{lem}\label{souslemexamtortorseur}
 Supposons que $k$ est un corps de nombres. Alors:
 
 (i) pour tout $\Gamma_k$-module de type fini sans torsion $M\neq 0$, le groupe $H^2(k,M)$ est infini;
 
 (ii) pour tout $\Gamma_k$-module fini $M\neq 0$, le groupe $H^1(k,M)$ est infini.
\end{lem}

\begin{lem}\label{lemexamtortorseur}
Supposons que $k$ est un corps de nombres.
Soit $\phi: M\to N $ un homomorphisme de $\Gamma_k$-modules de type fini sans torsion. 
Si $\Ker (H^2(k,M)\xrightarrow{\phi_*} H^2(k,N))$ est fini, alors $\phi$ est injectif et $\coker(\phi)$ est sans torsion.
\end{lem}

\begin{proof}
Notons $I:=\Im (\phi)$, $K:=\Ker(\phi)$ et $C:=\coker(\phi)$.
Alors $I$ et $K$ sont de type fini sans torsion et on a deux suites exactes:
$$H^1(k,I)\to H^2(k,K)\to H^2(k,M)\xrightarrow{\theta_1} H^2(k,I)$$
et
$$H^1(k,N)\to H^1(k,C)\to H^2(k,I)\xrightarrow{\theta_2}H^2(k,N).$$
Ainsi $H^1(k,I)$, $H^1(k,N)$ et $\Ker(\theta_1)$ sont finis.
Alors $H^2(k,K)$ est fini et donc $K=0$ (Lemme \ref{souslemexamtortorseur}).
Ainsi $\Ker(\theta_2)$ est fini.
Alors $H^1(k,C)$ est fini et donc  $C$ est sans torsion (Lemme \ref{souslemexamtortorseur}). 
\end{proof}

\begin{lem}\label{lemlemtorequotient}
Supposons que $k$ est un corps de nombres.
Soient $G$ un groupe lin\'eaire connexe, $T$ un tore et $G\xrightarrow{\psi}T$ un homomorphisme. Alors les \'enonc\'es suivants sont \'equivalents: 

(i) le groupe $\Ker (\Br_1(T)\xrightarrow{\psi^*}\Br_1(G))$ est fini; 

(ii) le morphisme $\psi$ est surjectif de noyau connexe; 

(iii) la $G$-vari\'et\'e $T$ est un pseudo $G$-espace homog\`ene.
\end{lem}

\begin{proof}
D'apr\`es l'exemple \ref{examtoretorseur} (1), on a (ii)$\Rightarrow$(iii). 
Par  d\'efinition, on a (iii)$\Rightarrow$(i).
Pour (i)$\Rightarrow$(ii), par la suite exacte de Sansuc, on peut remplacer $G$ par $G^{tor}$ et supposer que $G$ est un tore.
Dans ce cas, le noyau de $H^2(k,T^*)\xrightarrow{\psi^*} H^2(k,G^*)$ est fini.
Le lemme \ref{lemexamtortorseur} ci-dessus donne (ii).
\end{proof}

Soient $G$ un groupe lin\'eaire connexe et $Z$ un pseudo $G$-espace homog\`ene (cf. D\'efinition \ref{defpseudo}).
Soient $(Z^{tor})^*:=\bk[Z]^{\times}/\bk^{\times}$ un $\Gamma_k$-module libre de type fini et $Z^{tor}$ le tore correspondant.
Pour tout $z\in Z(k)$, on a un morphisme canonique $Z\xrightarrow{\pi_z}Z^{tor}$ tel que $\pi_z(z)=e_{Z^{tor}}$ (Rosenlicht).
Soit $\psi_z$ la composition $G\to G\cdot z\hookrightarrow Z\xrightarrow{\pi_z} Z^{\tor}$.

\begin{prop}\label{lemtorequotient}
Le morphisme $\psi_z$ ne d\'epend pas du choix de $z$, et c'est un homomorphisme tel que 
$Z^{tor}$ soit une $G$-vari\'et\'e et $\pi_z$ soit un $G$-morphisme.

De plus, si $k$ est un corps de nombres, la $G$-vari\'et\'e $Z^{\tor}$ est un pseudo $G$-espace homog\`ene.
\end{prop}

\begin{proof}
Notons $\rho: G\times Z\to Z$ l'action de $G$. 
Par le lemme de Rosenlicht, $\psi_z$ est un homomorphisme.
D'apr\`es le lemme \ref{lemaction}, on a un isomorphisme canonique
$$\xymatrix{H^0_{e_G}(G,Z^{tor})\oplus H^0(Z,Z^{tor})\ar@/^/[r]^-{p_1^*\cdot p_2^*}&H^0(G\times Z,Z^{tor}):\ar@/^/[l]^-{\theta}
&(\psi_z,\pi_z ) \ar@{|->}[r]& \psi_z\cdot \pi_z
}$$
tel que $\theta (\pi_z\circ \rho)=((\pi_z\circ \rho)|_{G\times z},(\pi_z\circ \rho)|_{e_G\times Z})=(\psi_z,\pi_z)$.
Alors $\pi_z\circ \rho=\psi_z\cdot \pi_z $ et $\pi_z$ est un $G$-morphisme.

Pour tout $z'\in Z(k)$, on a $\pi_{z'}=\pi_z(z')\cdot \pi_z$ et donc
$\theta (\pi_{z'}\circ \rho)=(\psi_z,\pi_z(z')\cdot \pi_z)$.
 Ainsi $\psi_{z'}=\psi_z$.

 La suite spectrale de Hochschild-Serre donne une suite exacte
$$\Pic(Z)\to \Pic(Z_{\bk})^{\Gamma_k}\to \Br_1(Z^{tor}) \xrightarrow{\pi_z^*|_{\Br_1}} \Br(Z).$$
Puisque $\Pic(Z_{\bk})$ est de type fini et $\Br_1(Z^{tor})$ est torsion, le groupe  $\Ker (\pi_z^*|_{\Br_1})$ est fini.
Puisque $\Ker(\Br_G(Z)\xrightarrow{\rho_z^*} \Br(G))$ est fini, 
le groupe  $\Ker (\Br_1(Z^{tor})\xrightarrow{\psi_z^*|_{\Br_1}} \Br(G))$ est fini.
Si $k$ est un corps de nombres, une application du lemme \ref{lemlemtorequotient} donne l'\'enonc\'e.
\end{proof}

\begin{defi}\label{defitorequotient}
Dans la proposition \ref{lemtorequotient}, une fois qu'on a  choisi   $z\in Z(k)$,
 le morphisme $Z\xrightarrow{\pi}Z^{tor}$ est appel\'e \emph{le quotient torique maximal}.
Il  ne d\'epend  du choix de $z$ qu'\`a translation pr\`es.
Si $k$ est un corps de nombres, le morphisme $G\to Z^{tor}$ dans la proposition \ref{lemtorequotient} est surjectif de noyau $G_0$ connexe.
Le groupe $G_0$ est appel\'e \emph{le stabilisateur de $G$ sur $Z^{tor}$}.
\end{defi}

 \section{L'approximation forte hors des places archim\'ediennes et la question \ref{ques2}}\label{4} 
 
  Dans toute cette section,  $k$ est un corps de nombres. 
Sauf  mention explicite, une vari\'et\'e est  une $k$-vari\'et\'e.
 Soit $G$ un groupe lin\'eaire connexe.
 Pour r\'epondre \`a la   question \ref{ques2}, on \'etablit le th\'eor\`eme \ref{mainthminfty}.
 Comme consequence, on montre le th\'eor\`eme \ref{thmgroupic} (1).
 
 Rappelons la notion de sous-groupe de Brauer invariant (cf. D\'efinition \ref{def-invariant}).
 
 \begin{lem}\label{mainleminfty}
Soit $G$ un groupe lin\'eaire connexe. Alors l'homomorphisme induit par l'accouplement de Brauer-Manin 
$G(\RA_k)_{\bullet}\xrightarrow{\theta_G} \Br_a(G)^D$ est ouvert, 
o\`u $(-)^D:=\Hom(-,\BQ/\BZ)$.
\end{lem}

\begin{proof}
Dans $G(\RA_k)_{\bullet}$, les sous-groupes ouverts compacts forment  une base topologique de $e_G$.
Pour tout tel sous-groupe $C$, l'image $\theta_G(C)\sbt \Br_a(G)^D$ est compacte, et donc ferm\'ee.
Il suffit  alors de montrer que cette image est d'indice fini.
Par la finitude de nombre de classes de $G$ (\cite[Thm. 5.1]{PR}),
il existe un tel sous-groupe $C_0$ tel que la classe double $C_0\backslash G(\RA_k)_{\bullet}/G(k) $ soit finie. 
Puisque $\Sha^1(G)$ est fini, d'apr\`es \cite[Thm. 5.1]{D11}, 
$\theta_G(C_0)$ est d'indice fini.
Pour tout tel sous-groupe $C$, le quotient $C_0/(C\cap C_0) $ est fini et donc $\theta_G(C)$ est d'indice fini.
\end{proof}

 \begin{thm}\label{mainthminfty}
Soient $G$ un groupe lin\'eaire connexe, $X$ une $G$-vari\'et\'e lisse g\'eom\'e\-triquement int\`egre et $U\sbt X$ un $G$-ouvert.
Soient $A\sbt \Br(X)$ un sous-groupe fini et  $B\sbt \Br_G(U)$ (cf. (\ref{defBr_GX})) un sous-groupe.
Supposons que $\frac{B\cap \Ker (\lambda)}{B\cap \Im \Br(k)}$ est fini, 
o\`u $\Br_G(U)\xrightarrow{\lambda}\Br_e(G)$ est  l'homomorphisme de Sansuc (D\'efinition \ref{defsansuc}).
Alors, pour tout ouvert $W\sbt X(\RA_k)$ satisfaisant $W^{(A+B)\cap \Br(X)}\neq\emptyset$, on a 
$W\cap U(\RA_k)^{A+B}\neq\emptyset.$
 \end{thm}
 
 \begin{proof}
 On peut supposer que $\Im \Br(k)\sbt B$.
 Apr\`es avoir r\'etr\'eci $W$, on peut supposer que tout \'el\'ement de $A$ s'annule sur $W$ 
 et $W\cong W_{\infty}\times W_f$ avec $W_{\infty}\sbt X(k_{\infty})$, $W_f\sbt X(\RA_k^{\infty})$
  tel que $W_f$ soit compact.
  
Pour tout $w\in W_f$, il existe un ouvert $W_w\sbt W_f$ contenant $w$ et un ouvert $C_w\sbt G(\RA^{\infty})$ tel que $C_w\cdot W_w\sbt W_f$.
Puisque $W_f$ est compact, 
il existe un sous-ensemble fini $I\sbt W_f $ tel que $\cup_{w\in I}W_w=W_f$.
Soit $C_f$ le sous-groupe de $G(\RA^{\infty})$ engendr\'e par $\cap_{w\in I}C_w$.
Alors $C_f\cdot W_f\sbt W_f$ et $C_f$ est ouvert dans $G(\RA^{\infty})$.
Soit $C:=(e_G)_{\infty}\times C_f\sbt G(\RA_k)$.
Alors $C\cdot W=W$ et l'image de $C$ dans $G(\RA_k)_{\bullet}$ est ouvert.

Notons $G(\RA_k)\xrightarrow{\theta_G} \Br_a(G)^D$ et $U(\RA_k)\xrightarrow{\theta_U}B^D$
 les applications induites par l'accouplement de Brauer-Manin, o\`u $(-)^D:=\Hom(-,\BQ/\BZ)$.
D'apr\`es le lemme \ref{mainleminfty}, $\theta_G (C)\sbt \Br_a(G)^D$ est un sous-groupe ouvert d'indice fini.
Par hypoth\`ese, il existe un sous-groupe fini $B_1\sbt B$ tel que 
\begin{equation}\label{mainthminfty-e}
\Ker((B)^D\xrightarrow{\vartheta} (B_1+\Im \Br(k))^D)\sbt (\lambda^D\circ \theta_G )(C),
\end{equation}
o\`u $\vartheta$ est induit par l'inclusion $B_1+\Im\Br(k)\sbt B$ et l'homomorphisme
$$\lambda: B\sbt \Br_G(U)\to \Br_e(G)=\Br_a(G) \ \ \ \text{induit} \ \ \ \lambda^D:=\Hom(\lambda,\BQ/\BZ): \Br_a(G)^D\to B^D.$$

D'apr\`es le lemme formel de Harari (\cite[Cor. 2.6.1]{Ha94}), $W\cap U(\RA_k)^{B_1}\neq\emptyset$.
On a un diagramme:
$$\xymatrix{G(\RA_k)\ar[d]^{\theta_G}&U(\RA_k)\ar[d]^{\theta_U}&\\
\Br_a(G)^D\ar[r]^-{\lambda^D}&B^D\ar[r]^-{\vartheta}&(B_1+\Im \Br(k))^D.
}$$
Soit $u\in W\cap U(\RA_k)^{B_1}$, alors $\vartheta(\theta_U(u))=0$ et, d'apr\`es (\ref{mainthminfty-e}), il existe $g\in C$ 
tel que $g^{-1}\cdot u\in W$ et $(\lambda^D\circ \theta_G)(g)=\theta_U(u)$.
D'apr\`es la proposition \ref{corbraueralgebraic1} (1), on a  $\theta_U(g^{-1}\cdot u)=0$.
 \end{proof}

\begin{rem}
 Dans le th\'eor\`eme \ref{mainthminfty}, si $G=1$, alors ce th\'eor\`eme est \'equivalent au lemme formel de Harari (\cite[Cor. 2.6.1]{Ha94}).
 \end{rem} 

Comme cons\'equence directe, on a :

\begin{cor}\label{main1cor}
 Avec les hypoth\`eses et notations du th\'eor\`eme \ref{mainthminfty}, pour tout sous-ensemble fini $S\sbt \Omega_k$, s'il existe un ouvert $X_1$ de $X$ tel que $U\sbt X_1$ 
 et $X_1$ satisfasse l'approximation forte par rapport \`a $\Br(X_1)\cap (A+ B)$ hors de $S $,
 alors $X$ satisfait l'approximation forte par rapport \`a $\Br(X)\cap (A+B)$ hors de $ S $.
 \end{cor}
 
 \begin{cor}\label{main1cor1}
 Avec les hypoth\`eses et notations du th\'eor\`eme \ref{mainthminfty}, s'il existe un ouvert $X_1$ de $X$ tel que $U\sbt X_1$ 
 et $X_1(k)$ soit dense dans $X_1(\RA_k)^{\Br(X_1)\cap (A+ B)}_{\bullet}$,
 alors $X(k)$ est dense dans $X(\RA_k)^{\Br(X)\cap (A+B)}_{\bullet}$.
 \end{cor}
 
Soit $G$ un groupe lin\'eaire connexe. 
 Rappelons la notion de pseudo $G$-espace homog\`ene (cf. D\'efinition \ref{defpseudo}).
 
\begin{thm}\label{main1thminfty}
 Soient $G$ un groupe lin\'eaire connexe et $Z$ un pseudo $G$-espace homog\`ene.
 Soient $X$ une $G$-vari\'et\'e lisse, g\'eom\'etriquement int\`egre 
 et $U\sbt X$ un $G$-ouvert  muni d'un $G$-morphisme $U\xrightarrow{f}Z$. 
 Soient $A\sbt \Br(X)$ et $B\sbt \Br_G(U)$ (cf. (\ref{defBr_GX})) deux sous-groupes finis. 
 Alors, pour tout ouvert $W\sbt X(\RA_k)$ satisfaisant  $W^{\Br(X)\cap (A+B+ f^*\Br_G(Z))}\neq\emptyset$,
 
 (1) on a $W\cap U(\RA_k)^{A+B+ f^*\Br_G(Z)}\neq\emptyset$;
 
 (2) si $G(k_{\infty})^+\cdot Z(k)$ est dense dans $Z(\RA_k)^{\Br_G(Z)}$, il existe $z\in Z(k)$ de fibre $U_z$
  tel que 
  $$(G(k_{\infty})^+\cdot W)\cap U_z(\RA_k)^{A+B}\neq\emptyset.$$
\end{thm}

\begin{proof}
Notons  $\Br_G(Z)\xrightarrow{\lambda_Z}\Br_a(G)$, $\Br_G(U)\xrightarrow{\lambda_U}\Br_a(G)$ les homomorphismes de Sansuc (cf. D\'efinition \ref{defsansuc}).
Puisque $\Ker (\lambda_Z)/\Im \Br(k)$ est fini et $\lambda_Z=\lambda_U\circ f^*$, 
le quotient $\frac{\Ker(\lambda_U)\cap (B+f^*\Br_G(Z))}{\Im\Br(k)}$ est fini.
Une application du th\'eor\`eme \ref{mainthminfty} donne (1).

D'apr\`es \cite[Thm. 4.5]{Co}, l'application $U(\RA_k)\to Z(\RA_k)$ est ouverte et on obtient (2). 
\end{proof}
 
 \begin{proof}[D\'emonstration du th\'eor\`eme \ref{thmgroupic} (1)]
 Par la proposition \ref{corbraueralgebraic1} (3) et l'approximation forte pour les espaces homog\`enes  \`a stabilisateur g\'eom\'etrique connexe (Borovoi et Demarche \cite[Thm. 1.4]{BD}), $U(k)$ est dense dans $U(\RA_k)^{\Br_G(U)}_{\bullet}$.
 Une application du th\'eor\`eme  \ref{main1thminfty} (2) donne l'\'enonc\'e.
 \end{proof}

\section{La descente par rapport au groupe de Brauer invariant}\label{5}

 Dans toute cette section,  $k$ est un corps de nombres. 
Sauf  mention explicite du contraire, une vari\'et\'e est  une $k$-vari\'et\'e.
La m\'ethode de descente des points ad\'eliques est \'etablie par Colliot-Th\'el\`ene et Sansuc dans \cite{CTS87}.
Dans \cite{CX2}, C. Demarche, F. Xu et l'auteur \'etudient la m\'ethode de descente des points ad\'eliques orthogonaux \`a certains groupes de Brauer dans le cas des torseurs sous un tore.
 On suit leur m\'ethode et  consid\`ere ici  le cas plus g\'en\'eral  des  torseurs sous un groupe lin\'eaire connexe (Th\'eor\`eme \ref{Dmain}).
 
 Rappelons la notion de sous-groupe de Brauer invariant (cf. D\'efinition \ref{def-invariant}).
 
\subsection{Torseur sous un groupe lin\'eaire dont le groupe de Shafarevich est trivial}

Soient $G$ un groupe lin\'eaire connexe, $X$ une vari\'et\'e lisse g\'eom\'e\-triquement int\`egre et $Y\xrightarrow{p}X$ un $G$-torseur. 
Pour tout $\sigma\in H^1(k,G)$, soient $P_{\sigma}$ le $G$-torseur correspondant et $G_{\sigma}$ le tordu de $G$ correspondant. 
Alors $P_{\sigma}$ est un $G_{\sigma}$-torseur et une $(G_{\sigma}\times G)$-vari\'et\'e.
Soit $Y_{\sigma}\xrightarrow{p_{\sigma}}X$ le tordu de $[Y]$, i.e. $Y_{\sigma}=P_{\sigma}\times^GY$.
Alors $[Y_{\sigma}]$ est un $G_{\sigma}$-torseur sur $X$. 
 Soit $\Br_{G_{\sigma}}(Y_{\sigma})$ le sous-groupe de Brauer  $G_{\sigma}$-invariant de $Y_{\sigma}$ (D\'efinition \ref{def-invariant}).

\begin{lem}\label{descent-lem2}
Supposons que  $\Sha^1(G)=1$.
Alors pour tout sous-groupe $B\sbt \Br_G(Y)$, on a:
$$p(Y(\RA_k)^{(p^*\circ (p^*)^{-1})B})=p(Y(\RA_k)^B),$$
o\`u $\Br(X)\xrightarrow{p^*}\Br(Y) $.
\end{lem} 

\begin{proof}
Puisque $(p^*\circ (p^*)^{-1})B\sbt B $, on a $p(Y(\RA_k)^B)\sbt p(Y(\RA_k)^{(p^*\circ (p^*)^{-1})B}).$

Par le th\'eor\`eme \ref{sansucthm}, on a un diagramme commutatif de suites exactes
$$\xymatrix{(p^*)^{-1}B\ar[r]\ar[d]&B\ar[r]\ar[d]&\Br_a(G)\ar[d]^=\\
\Br(X)\ar[r]&\Br_G(Y)\ar[r]^{\lambda}&\Br_a(G).
}$$
Il induit un diagramme commutatif avec suite exacte:
$$\xymatrix{G(\RA_k)\ar[d]^{a_G}&Y(\RA_k)\ar[d]^{a_Y}\ar[r]^p&X(\RA_k)\ar[d]^{a_X}\\
\Br_a(G)^D\ar[r]^{\lambda^D}&B^D\ar[r]^{p^{*D}}&((p^*)^{-1}B)^D
}$$
o\`u $(-)^D:=\Hom(-,\BQ/\BZ)$ et $a_G$, $a_Y$, $a_X$ sont induits par l'accouplement de Brauer-Manin.
Par la proposition \ref{corbraueralgebraic1}, pour chaque $y\in Y(\RA_k)$ et $g\in G(\RA_k)$, 
on a $a_Y(g\cdot y)=(\lambda^D\circ a_G)(g)+a_Y(y)$.

Puisque $\Sha^1(G)=1$, par la suite exacte de Poitou-Tate de $G$ (Demarche \cite[Thm. 5.1]{D11}), $a_G$ est surjectif.
Pour tout $y\in p(Y(\RA_k)^{(p^*\circ (p^*)^{-1})B})$, on a
$x:=p(y)\in X(\RA_k)^{(p^*)^{-1}B}$.
Alors  $(p^{*D}\circ a_Y)(y)=a_X(x)=0$ et il existe $g\in G(\RA_k)$ tel que $(\lambda^D\circ a_G)(g)=a_Y(y)$.
Donc $a_Y(g^{-1}\cdot y)=0$ et $p(g^{-1}\cdot y)=x$. 
\end{proof}

\begin{prop}\label{DmainProp}
Soient $G$ un groupe lin\'eaire connexe, $X$ une vari\'et\'e lisse g\'eom\'e\-triquement int\`egre et $Y\xrightarrow{p}X$ un $G$-torseur.  
Soit $A\sbt \Br(X)$ un sous-groupe et, pour chaque $\sigma\in H^1(k,G)$, soit $B_{\sigma}\sbt \Br_{G_{\sigma}}(Y_{\sigma})$ un sous-groupe.
Supposons que $\Sha^1(G)=1$ et que, pour tout $\sigma \in H^1(k,G)$, on a 
$(p_{\sigma}^*)^{-1}(B_{\sigma})\sbt A$,
o\`u $\Br(X)\xrightarrow{p_{\sigma}^*}\Br(Y_{\sigma}) $. 
Alors on a:
$$ X(\RA_k)^A=\cup_{\sigma\in H^1(k,G)}p_{\sigma}(Y_{\sigma}(\RA_k)^{B_{\sigma}+p_{\sigma}^*A }  ).$$
\end{prop}

 \begin{proof}
 Le r\'esultat d\'ecoule du \cite[Thm. 1.1]{CX2} et du lemme \ref{descent-lem2}.
 \end{proof}
 
\begin{cor} \label{descent}
Soit $T$ un tore quasi-trivial. Soient $X$ une vari\'et\'e lisse, g\'eom\'etriquement int\`egre et $Y\xrightarrow{p}X$ un $T$-torseur. Soient $U\sbt X$ un ouvert et  $V=p^{-1}(U)$.
 Alors pour tous sous-groupes $A\sbt \Br(U)$, $B\sbt \Br_T(V)$, si $(p^*)^{-1}(B)\sbt A$, o\`u $\Br(U)\xrightarrow{p^*}\Br(V)$, on a 
 $$X(\RA_k)^{\Br(X)\cap A}=p(Y(\RA_k)^{\Br(Y)\cap (B+ p^*A)}).$$
\end{cor}

\begin{proof}
Par le lemme \ref{brauertorseur}, on a $(p^*)^{-1}(\Br(Y)\cap (p^*A+B))=\Br(X)\cap A$. 
D'apr\`es la proposition \ref{prop-invariant}, on a $(\Br(Y)\cap (p^*A+B))\sbt \Br_{T}(Y)$.
L'\'enonc\'e r\'esulte de la proposition \ref{DmainProp}.
\end{proof} 
 
\medskip  
 
 \subsection{L'application de la r\'esolution coflasque}

Par \cite{CTS1}, un $\Gamma_k$-module $M$ de type fini est appel\'e \emph{coflasque} si $M$ est sans torsion et, pour tout sous-groupe ferm\'e $\Gamma\sbt \Gamma_k$, on a $H^1(\Gamma ,M)=0$.
Un $k$-tore $T$ est appel\'e \emph{coflasque} si $T^*$ est coflasque.
Alors $H^1(k,T^*)=0$ et, pour tout $v\in \Omega_k$, on a $H^1(k_v,T^*)=0 $.
Dans \cite[Prop. 1.3]{CTS1},  il y a des r\'esolutions par les tores coflasques et les tores quasi-triviaux.

\begin{lem}\label{Dlembrauersurj1}
Soit $T$ un tore. Alors $H^3(k,T^*)\cong \prod_{v\in \infty_k}H^3(k_v,T^*)\cong \prod_{v\in \infty_k}H^1(k_v,T^*)$.
\end{lem}

\begin{proof}
Pour tout $v\in  \infty_k$, on a $H^3(k_v,T^*)=H^1(k_v,T^*)$. 
Par la ligne 3 de la d\'emonstration de \cite[Prop. 5.9]{HS05},
on a $H^3(k,T^*)=\prod_{v\in \infty_k}H^3(k_v,T^*)$.
Le r\'esultat en d\'ecoule.
\end{proof}

\begin{lem}\label{Dlembrauersurj}
Soit 
$$1\to G \xrightarrow{\psi } H \xrightarrow{\phi } T \to 1$$
une suite exacte de groupes lin\'eaires connexes avec $T$ un tore. 
Supposons que $H^3(k,T^*)=0$.
Alors
le morphisme $\Br_a(H)\xrightarrow{\psi^*} \Br_a(G)$ est surjectif.
\end{lem}

\begin{proof}
Par la suite exacte de Sansuc \cite[Cor. 6.11]{S}, on a une suite exacte de $\Gamma_k$-modules
$$0\to T^* \to H^* \to G^* \to 0$$
et un isomorphisme de $\Gamma_k$-modules $\Pic(H_{\bk})\to \Pic(G_{\bk})$.
Par \cite[Lem. 2.1]{CTX09}, on a un diagramme commutatif de suites exactes:
$$\xymatrix{H^2(k,H^*)\ar[r]\ar[d]^{\psi_1}&\Br_a(H)\ar[r]\ar[d]&H^1(k,\Pic(H_{\bk}))\ar[d]^{\cong}\ar[r]&H^3(k,\bk^{\times}\oplus H^*)\ar[d]^{\psi_2}\\
H^2(k,G^*)\ar[r]&\Br_a(G)\ar[r]&H^1(k,\Pic(G_{\bk}))\ar[r]&H^3(k,\bk^{\times}\oplus G^*).
}$$
Puisque $H^3(k,T^*)=0$, le morphisme $\psi_1$ est surjectif et le morphisme $\psi_2$ est injectif.
Donc le morphisme $\Br_a(H)\to \Br_a(G)$ est surjectif.
\end{proof}

\begin{prop}\label{Dprolonger}
Soit $G$ un groupe r\'eductif connexe. Alors il existe un groupe lin\'eaire connexe $H$, un tore $T$ et une suite exacte:
$$1\to G \xrightarrow{\psi } H \xrightarrow{\phi } T \to 1$$
tels que $\Sha^1(H)=0$, $H^3(k,T^*)=0$ et $\Br_a(H)\xrightarrow{\psi^*}\Br_a(G)$ soit surjectif.

De plus, si $G$ est un tore, on peut imposer que $H$ soit un tore quasi-trivial.
\end{prop}

\begin{proof}
Puisque $G$ est r\'eductif, il existe une suite exacte
$0\to S\to G \to G^{ad}\to 0$
telle que $S$ soit le centre de $G$ et $G^{ad}$ soit le groupe adjoint de $G$.
Alors $S$ est  un groupe de type multiplicatif.
Par \cite[Prop. 1.3]{CTS1}, il existe un tore quasi-trivial $T_0$ et un homomorphisme injectif $S\xrightarrow{\chi}T_0$
tels que  $T:=T_0/S$ soit un tore coflasque.
Alors,  pour tout $v\in \infty_k$, on a $H^1(k_v,T^*)=0$. Par le lemme \ref{Dlembrauersurj1}, on a $H^3(k,T^*)=0$.

Soit $H:=G\times^ST_0$. Alors $H$ est un groupe lin\'eaire et on a deux suites exactes
$$1\to G \xrightarrow{\psi } H \xrightarrow{\phi } T \to 1\ \ \ \text{et}\ \ \ 
1\to T_0 \to H \to G^{ad} \to 1.$$
Par le lemme \ref{Dlembrauersurj}, le morphisme $\Br_a(H)\xrightarrow{\psi^*} \Br_a(G)$ est surjectif.
Par \cite[Cor. 5.4]{S}, on a $\Sha^1(G^{ad})=0$. 
Puisque $T_0$ est quasi-trivial, on a $\Sha^1(H)=0$.
\end{proof}

 \begin{prop}\label{proptorseurcoflasque}
 Soit $X$ une vari\'et\'e lisse g\'eom\'etriquement int\`egre. 
 Supposons que $X(k)\neq\emptyset$ et que $\Pic(X_{\bk})$ est de type fini.
 Alors il existe un tore quasi-trivial $T$ et un $T$-torseur $Y\to X$ 
 tels que $\Pic(Y_{\bk})=0$ et $H^3(k,\bk[Y]^{\times}/\bk^{\times})=0$.
 \end{prop}

\begin{proof}
Il existe un ouvert $U\sbt X$  tel que $\Pic(U_{\bk})=0$.
Soient $T_0$ un tore tel que $T_0^*:=\Div_{X_{\bk}\setminus U_{\bk}}(X_{\bk})$ 
et $Y_0\to X$ le $T_0$-torseur induit par l'homomorphisme $ \Psi$ de  la suite exacte (\ref{tor-e1}).
Par la suite exacte de Sansuc \cite[Prop. 6.10]{S}, $\Pic(Y_{0,\bk})=0$.
Soit $T_1$ le tore tel que $T_1^*\cong \bk[Y_0]^{\times}/\bk^{\times}$.
Puisque $X(k)\neq\emptyset$, on a $Y_0(k)\neq\emptyset$
 et il induit un morphisme $Y_0\xrightarrow{\pi} T_1$ tel que $T_1^*\xrightarrow{\pi^*}\bk[Y_0]^{\times}/\bk^{\times}$ soit un isomorphisme.
 
 D'apr\`es \cite[Prop. 1.3]{CTS1}, il existe une suite exacte $0\to T_2\to T_3\to T_1\to 0$ telle que 
 $T_2$ soit quasi-trivial et que $ T_3$ soit coflasque.
 Soit $Y:=Y_0\times_{T_1}T_3$ un $T_2$ torseur sur $Y_0$.
 Par la suite exacte de Sansuc \cite[Prop. 6.10]{S}, $\Pic(Y_{\bk})=0$, $\bk[Y]^{\times}/\bk^{\times}=T_3^*$ 
 et donc $H^3(k,\bk[Y]^{\times}/\bk^{\times})=0$.
 
Soit $T:=T_0\times T_2$. 
 Puisque $T_0, T_2$ sont quasi-triviaux, on a $H^1(T_0,T_2)=0$ et, d'apr\`es \cite[Cor. 5.7]{CT07}, 
 toute extension centrale de $T_0$ par $T_2$ est isomorphe \`a $T_0\times T_2$.
 D'apr\`es le th\'eor\`eme \ref{thmaction}, il existe une action de $T$ sur $Y$ 
 compatible avec les actions $T_2$ sur $Y$ et $T_0$ sur $Y_0$. 
 Alors $Y$ est un $T$-torseur sur $X$.
\end{proof}

\medskip

\subsection{La descente g\'en\'erale}

Soient $G$ un groupe lin\'eaire connexe, $X$ une vari\'et\'e lisse g\'eom\'e\-triquement int\`egre et $Y\xrightarrow{p}X$ un $G$-torseur. 
Pour tout $\sigma\in H^1(k,G)$, soient $P_{\sigma}$ le $G$-torseur correspondant et $G_{\sigma}$ le tordu de $G$ correspondant. 
Alors $P_{\sigma}$ est un $G_{\sigma}$-torseur et une $(G_{\sigma}\times G)$-vari\'et\'e.
Soit $Y_{\sigma}\xrightarrow{p_{\sigma}}X$ le tordu de $[Y]$, i.e. $Y_{\sigma}=P_{\sigma}\times^GY$.
Alors $[Y_{\sigma}]$ est un $G_{\sigma}$-torseur sur $X$. 
Soit $\Br_{G_{\sigma}}(Y_{\sigma})$ le sous-groupe de Brauer $G_{\sigma}$-invariant de $Y_{\sigma}$ (D\'efinition \ref{def-invariant}).

Notons $P_{\sigma}\times Y\xrightarrow{\theta_Y^{\sigma}} Y_{\sigma}$.
Le lemme  \ref{Dlemsuite1} donne un isomorphisme canonique:
$$\Br_a(P_{\sigma})\oplus \Br_G(Y)/\Im\Br(k) \xrightarrow{(p_1^*,p_2^*)}\Br_{G_{\sigma}\times G}(P_{\sigma}\times Y)/\Im\Br(k).  $$
Il induit un morphisme canonique 
$$\Theta_{Y}^{\sigma}:\ \Br_{G_{\sigma}}(Y_{\sigma})/\Im\Br(k)\xrightarrow{(\theta_Y^{\sigma})^*}   \Br_{G_{\sigma}\times G}(P_{\sigma}\times Y)/\Im\Br(k) \to \Br_G(Y)/\Im\Br(k).$$

\begin{lem}\label{Dlemsuite2}
Notons $\sigma^{-1}:=[P_{\sigma}]\in H^1(k,G_{\sigma})$. 
Alors $Y\cong (Y_{\sigma})_{\sigma^{-1}}$ et
$\Theta_{Y}^{\sigma}$ est un isomorphisme d'inverse $\Theta_{Y_{\sigma}}^{(\sigma^{-1})}$.
\end{lem} 

\begin{proof}
On a $(Y_{\sigma})_{\sigma^{-1}}=P_{\sigma}\times^{G_{\sigma}}P_{\sigma}\times^GY\cong G\times^GY\cong Y$.
On a un isomorphisme canonique $P_{\sigma}\times Y\xrightarrow{p_1\times \theta_Y^{\sigma}}P_{\sigma}\times Y_{\sigma} $ d'inverse $p_1\times \theta_{Y_{\sigma}}^{(\sigma^{-1})} $.
Par le lemme \ref{Dlemsuite1}, on a: 
$$\Br_a(P_{\sigma})\oplus \Br_G(Y)/\Im\Br(k) \cong \Br_{G_{\sigma}\times G}(P_{\sigma}\times Y)/\Im\Br(k)  $$
et
$$\Br_a(P_{\sigma})\oplus \Br_{G_{\sigma}}(Y_{\sigma})/\Im\Br(k) \cong \Br_{G_{\sigma}\times G}(P_{\sigma}\times Y_{\sigma})/\Im\Br(k).$$
Le r\'esultat en d\'ecoule.
\end{proof}

Pour un sous-groupe $B_{\sigma}\sbt \Br_{G_{\sigma}}(Y_{\sigma})$ contenant $\Im\Br(k)$, 
on note $\widetilde{\Theta_{Y}^{\sigma}}(B_{\sigma})\sbt \Br_G(Y)$ l'image inverse de $\Theta_{Y}^{\sigma}(B_{\sigma}) $.
Alors $\widetilde{\Theta_{Y_{\sigma}}^{(\sigma^{-1})}} (\widetilde{\Theta_{Y}^{\sigma}}(B_{\sigma}))=B_{\sigma}$. 

\begin{thm}\label{Dmain}
Soient $G$ un groupe lin\'eaire connexe, $X$ une vari\'et\'e lisse g\'eom\'e\-triquement int\`egre et $Y\xrightarrow{p}X$ un $G$-torseur.  
Soit $A\sbt \Br(X)$ un sous-groupe et, pour chaque $\sigma\in H^1(k,G)$, soit $B_{\sigma}\sbt \Br_{G_{\sigma}}(Y_{\sigma})$ un sous-groupe contenant $\Im \Br(k)$.
Supposons que, pour tout $\sigma \in H^1(k,G)$, on a 
$$(p_{\sigma}^*)^{-1}(\sum_{\sigma'\in \Sha^1(G_{\sigma})}\widetilde{\Theta_{Y_{\sigma}}^{\sigma'}}(B_{\sigma+\sigma'}))\sbt A$$
o\`u $\Br(X)\xrightarrow{p_{\sigma}^*}\Br(Y_{\sigma}) $ et $Y_{\sigma+\sigma'}:=(Y_{\sigma})_{\sigma'}$. 
Alors on a:
$$ X(\RA_k)^A=\cup_{\sigma\in H^1(k,G)}p_{\sigma}(Y_{\sigma}(\RA_k)^{B_{\sigma}+p_{\sigma}^*A }  ).$$
\end{thm}

\begin{proof}
Par \cite[Thm. 1.1]{CX2}, il suffit de montrer que, pour tout  $\sigma \in H^1(k,G)$, on a:
$$p_{\sigma}(Y_{\sigma}(\RA_k)^{p_{\sigma}^*A }  )\sbt \cup_{\sigma'\in \Sha^1(G_{\sigma})}p_{\sigma+\sigma'}(Y_{\sigma+\sigma'}(\RA_k)^{B_{\sigma+\sigma'}+p_{\sigma+\sigma'}^*A }  ). $$
On peut supposer que $\sigma=0$. 
Pour tout $v$, puisque $H^1(k_v,G^u)=0$, le morphisme $Y(k_v)\to (G^{red}\times^GY)(k_v)$ est surjectif.
Puisque $\Sha^1(G)\cong \Sha^1(G^{red})$  (\cite[Prop. 4.1]{S})  et que l'on a $\Br(Y)\cong \Br(G^{red}\times^GY) $  (\cite[Lem. 2.1]{CX2}),
on peut supposer que  $G$ est r\'eductif.

\medskip
 
 Par la proposition \ref{Dprolonger},  il existe un groupe lin\'eaire connexe $H$, un tore $T$ et une suite exacte:
$$1\to G \xrightarrow{\psi } H \xrightarrow{\phi } T \to 1$$
tels que $\Sha^1(H)=0$, $H^3(k,T^*)=0$ et $\Br_a(H)\xrightarrow{\psi^*}\Br_a(G)$ soit surjectif.
Par \cite[Prop. 36 du Chapitre 1]{S65}, on a une suite 
exacte
d'ensembles  point\'es:
$$1\to G(k)\to H(k)\to T(k)\xrightarrow{\partial} H^1(k,G)\to H^1(k,H),$$
telle que, pour tout $t\in T$, on ait $\partial (t):=[\phi^{-1}(t)]$. 
Alors $\Sha^1(G)\sbt \Im (\partial) $ et, pour tout $t\in T(k)$, on a $G_{\partial (t)}\cong G$.

Soient $Y_H:=H\times^GY$ et $Y\xrightarrow{i}Y_H$ le morphisme canonique. 
Alors $Y_H\xrightarrow{p_H}X$ est un $H$-torseur sur $X$ et $T\times^H Y_H\cong T\times X$.
Alors on a un morphisme canonique $Y_H\xrightarrow{\mu}T$ 
tel que, pour tout $t\in T(k)$, on a $\mu^{-1}(t)\cong  \phi^{-1}(t)\times^GY\cong Y_{\partial (t)}$.
Notons $\mu^{-1}(t)\xrightarrow{i_t}Y_H$. 
Par la d\'efinition de $\Theta_{Y}^{\partial (t)}$, on a: $\Theta_{Y}^{\partial (t)}\circ i_t^*=i^*$.
Par la proposition \ref{propbrauersuj} et le lemme \ref{Dlemsuite2},
 les morphismes $\Br_H(Y_H)\xrightarrow{i^*}\Br_G(Y)$ et $\Br_H(Y_H)\xrightarrow{i_t^*}\Br_G(\mu^{-1}(t))$ sont surjectifs.
  
Notons 
$$B:=\Br_H(Y_H)\cap (i^*)^{-1}(\sum_{\sigma\in \Sha^1(G)}\widetilde{\Theta_{Y}^{\sigma}}(B_{\sigma}))\sbt \Br_H(Y_H).$$ 
Alors $ (p_H^*)^{-1}(B)\sbt A,$ $B_{\partial (t)}\sbt i_t^*B$ et $ \mu^*\Br_1(T)\sbt B.$

Pour un $y\in Y(\RA_k)^{p^*A}$, par le lemme \ref{descent-lem2}, il existe $y_1\in Y_H(\RA_k)^{B+p_H^*A}$ tel que $p_H(y_1)=p(y)$.
Alors il existe $h\in H(\RA_k)$ tel que $y_1=h\cdot y$. 
Donc $\mu (y_1)\in \phi (H(\RA_k))\cap T(\RA_k)^{\Br_1(T)}$.

Par  \cite[Prop. 3.2]{CX2}, on a:
$$T(\RA_k)^{\Br_1(T)}=T(k)\cdot \phi (H(\RA_k)^{\Br_1(H)}).$$
Donc il existe $h_1\in H(\RA_k)^{\Br_1(H)}$ tel que 
$$t:=\mu (h_1\cdot y_1)\in T(k)\cap \phi (H(\RA_k)).$$
Donc $\partial (t)\sbt \Sha^1(G)$.
  Par la proposition \ref{corbraueralgebraic1} (1), 
  $$y_2:=h_1\cdot y_1\in Y_H(\RA_k)^{B+p_H^*A} \ \ \ \text{et donc}\ \ \  y_2\in \mu^{-1}(t)(\RA_k)^{B_{\partial (t)}+p_{\partial (t)}^*A}.$$
Le r\'esultat d\'ecoule de $p_{\partial (t)}(y_2)=p_H(y_2)=p(y)$.
\end{proof}

\begin{cor}\label{Dmaincor1}
Sous les hypoth\`eses du th\'eor\`eme \ref{Dmain}, soit 
$$X(k)\xrightarrow{\partial} H^1(k,G): \ z\mapsto [p^{-1}(z)]$$
le morphisme canonique. 
Supposons que pour tout $\sigma\in H^1(k,G)$, la vari\'et\'e $Y_{\sigma}$ satisfait le principe de Hasse par rapport \`a $B_{\sigma}+p_{\sigma}^*A$. 
Alors 
$$ X(\RA_k)^A=\cup_{\sigma\in \Im (\partial)}p_{\sigma}(Y_{\sigma}(\RA_k)^{B_{\sigma}+p_{\sigma}^*A }  ).$$
\end{cor}

\begin{proof}
Pour tout $\sigma\in H^1(k,G)$, si $Y_{\sigma}(\RA_k)^{B_{\sigma}+p_{\sigma}^*A }\neq\emptyset $, 
alors on a $Y_{\sigma}(k)\neq\emptyset$ et donc $\sigma\in \Im (\partial)$.
Le r\'esultat en d\'ecoule.
\end{proof}

\begin{cor}\label{Dmaincor}
Soient $G$ un groupe lin\'eaire connexe, $X$ une vari\'et\'e lisse g\'eom\'e\-triquement int\`egre et $Y\xrightarrow{p}X$ un $G$-torseur.  
Soit $\Br_1(X,Y):=\Ker (\Br(X)\to \Br(Y_{\bk})).$
Alors on a:
$$X(\RA_k)^{\Br_1(X,Y)}=\cup_{\sigma\in H^1(k,G)}p_{\sigma}(Y_{\sigma}(\RA_k)^{\Br_1(Y_{\sigma})} ) .$$
\end{cor}

\begin{proof}
Pour tout $\sigma\in H^1(k,G)$, puisque le morphisme $\Br(Y_{\bk})\to \Br((P_{\sigma}\times Y)_{\bk})$ est injectif, 
on a $\Br_1(X,Y_{\sigma})\sbt \Br_1(X,Y)$. Donc $\Br_1(X,Y_{\sigma})= \Br_1(X,Y)=(p^*)^{-1}\Br_1(Y)$.

Pour tout $\sigma\in H^1(k,G)$,
par le lemme \ref{Dlemsuite1} et le lemme \ref{Dlemsuite2}, 
le morphisme $\Theta_{Y}^{\sigma}$ induit un isomorphisme $\Br_a(Y_{\sigma})\to \Br_a(Y)$. 
Donc $\widetilde{\Theta_{Y}^{\sigma}}(\Br_1(Y_{\sigma}))=\Br_1(Y)$.
Le r\'esultat d\'ecoule du th\'eor\`eme \ref{Dmain}.
\end{proof}

\begin{rem}
La d\'emonstration du th\'eor\`eme \ref{Dmain} n'utilise pas  l'approximation forte pour les espaces homog\`enes  \`a stabilisateur g\'eom\'etrique connexe (Borovoi et Demarche \cite[Thm. 1.4]{BD}).
\end{rem}

\begin{cor}\label{Dmaincor2}
Soient $G$ un groupe lin\'eaire connexe, $G_0\sbt G$ un sous groupe ferm\'e connexe et $Z:=G/G_0$.
Notons $G\xrightarrow{\pi}Z$ la projection.
Alors on a
$$Z(\RA_k)^{\Br_G(Z)}=\pi (G(\RA_k)^{\Br_1(G)})\cdot Z(k).$$
\end{cor}

\begin{proof}
Puisque $G\xrightarrow{\pi}Z$ est un $G_0$-torseur \`a droite,
pour tout $P_{\sigma}\in H^1(k,G_0)$, le tordu $G\times^{G_0}P_{\sigma}$ est un $G$ torseur \`a gauche.
Donc  $G\times^{G_0}P_{\sigma}$ satisfait le principe de Hasse par rapport \`a $\Br_1(G\times^{G_0}P_{\sigma})$ (Sansuc, cf. \cite[Thm. 5.2.1]{sko}).
Le r\'esultat d\'ecoule du corollaire \ref{Dmaincor1}.
\end{proof}

 \section{La m\'ethode de fibration et la question \ref{ques1}}\label{6}
 
 Dans toute cette section,  $k$ est un corps de nombres. 
Sauf  mention explicite du contraire, une vari\'et\'e est une  $k$-vari\'et\'e. 
Pour traiter la   question  \ref{ques1}, on est amen\'e \`a \'etudier la question :

\begin{ques}\label{ques4}
 Soit $G$ un groupe lin\'eaire. Soient $X$ et $Z$ deux $G$-vari\'et\'es lisses g\'eom\'etriquement int\`egres.
 Soient $U\sbt X$ un $G$-ouvert et $U\xrightarrow{f}Z$ un $G$-morphisme.
 Soient $B\sbt \Br(U)$ un sous-groupe fini, $S\sbt \Omega_k$ un sous-ensemble fini non vide et $W\sbt X(\RA_k)$ un ouvert.
 Sous quelles conditions peut-on montrer que, si $W^{\Br(U)\cap (B+f^*\Br_G(Z))}\neq\emptyset $, 
 alors il existe $z\in Z(k)$ de fibre $U_z$ tel que 
 $((G(k_S)\cdot W)\cap U_z(\RA_k))^{B}\neq\emptyset$?
 \end{ques}
 
 Dans \cite{CTH}, Colliot-Th\'el\`ene et Harari \'etudient la m\'ethode de fibration au dessus de $\BA^1$. 
 On suit leur m\'ethode et  r\'epond \`a la question \ref{ques4} d'abord  dans le cas o\`u $Z$ est un tore quasi-trivial 
 et $f$ s'\'etend \`a un morphisme de $X$ vers une $Z$-vari\'et\'e torique standard  (Th\'eor\`eme \ref{thmfibration}).
Ensuite, en utilisant ce r\'esultat,  on r\'esoud cette question dans le cas o\`u $Z$ est un tore (Th\'eor\`eme \ref{mainlem}).
En utilisant la descente (\S 5), on \'etablit le th\'eor\`eme \ref{main1thm} dans le cas plus g\'en\'eral o\`u $Z$ est un pseudo $G$-espace homog\`ene.
Ceci sera utilis\'e dans la section \ref{7} pour traiter le cas d'un $G$-espace homog\`ene 
$Z$  \`a stabilisateur g\'eom\'etrique connexe.

Rappelons la notion de sous-groupe de Brauer invariant (cf. D\'efinition \ref{def-invariant}).
 
 \subsection{L'approximation forte raffin\'ee pour l'espace affine} \label{6.1}
 
Pour un ouvert $U$ d'un espace affine $\BA^n$ satisfaisant $\codim (\BA^n\setminus U,\BA^n)\geq 2$,
 l'approximation forte hors d'une place $v_0$ a \'et\'e \'etablie par  Fei Xu et l'auteur   dans \cite[Prop. 3.6]{CX}, 
  et raffin\'ee lorsque  la place  $v_0$ est archim\'edienne dans \cite[Prop. 4.6 et Cor. 4.7 ]{CX1}, 
 o\`u l'on montre que $U(k)\cap T(k_{v_0})^+\sbt \BA^n(k_{v_0})$ est dense dans $U(\RA_k^{v_0})$
   pour une vari\'et\'e torique $(T\hookrightarrow \BA^n)$ comme (\ref{strangapptoricstandthm-e}) ci-dessous.
 On g\'en\'eralise maintenant  ce r\'esultat au cas o\`u $v_0$ est une place   quelconque. 

 \begin{thm}\label{strangapptoricstandthm}
 Soient $K$ une $k$-alg\`ebre finie s\'eparable de degr\'e $n$ et 
\begin{equation}\label{strangapptoricstandthm-e}
(T\hookrightarrow \BA^n):= (\Res_{K/k}\BG_m\hookrightarrow \Res_{K/k}\BA^1)
\end{equation} 
 la  vari\'et\'e torique correspondante.
  Soit $U\sbt \BA^n$ un ouvert tel que $\codim(\BA^n\setminus U,\BA^n)\geq 2$.
 Soient $v_0$ une place et $D_{v_0}\sbt T(k_{v_0})$ un sous-groupe ouvert d'indice fini.
 Alors, pour tout ouvert $W\sbt U(\RA^{v_0})$  non vide  et tout ouvert $W_0\sbt  \BA^n(k_{v_0})$  non vide $D_{v_0}$-invariant, on a 
 $$U(k)\cap (W_0\times W)\neq \emptyset.$$
 \end{thm}

\begin{proof}
\'Etape (0).  Par approximation faible pour le tore quasi-trivial  $T$, il existe un $\alpha\in T(k)\cap W_0$. 
Ainsi on a $D_{v_0}\sbt \alpha^{-1}\cdot W_0$.
En rempla\c{c}ant  $U$ par $\alpha^{-1}U$ et $W$ par $\alpha^{-1}W$,
on peut supposer que $W_0=D_{v_0} $.
Si $v_0$ est complexe, on a $D_{v_0}=T(k_{v_0})$ et l'\'enonc\'e d\'ecoule de \cite[Prop. 3.6]{CX}.
On suppose que $v_0$ n'est pas complexe.

\'Etape (1). Supposons que $v_0\in \infty_k$ et $T=\BG_m^n$.
Puisque $D_{v_0}\sbt T(k_{v_0})$ est ouvert et donc ferm\'e,
le sous-groupe $D_{v_0} $ contient la composante connexe de l'identit\'e de $T(k_{v_0})$.
Ainsi l'\'enonc\'e est \'equivalent \`a \cite[Prop. 4.6]{CX1}. 

\'Etape (2). Supposons que $v_0<\infty_k$, $n=1$ et $T=\BG_m$.
On a $U=\BA^1$.
Par d\'efinition, Il existe un sous-ensemble fini $S\sbt \Omega_k\setminus \{v_0\} $ contenant $\infty_k$, 
un \'el\'ement $a_v\in k_v$ pour chaque $v\in S$ et $0<\epsilon<1$ tels que 
$$\prod_{v\in S}\{x\in k_v\ :\ |x-a_v|<\epsilon\}\times\prod_{v\notin S\cup \{v_0\} }\CO_v\sbt W.$$
On fixe une uniformisante $\pi_{v_0}$ de $k_{v_0}$.
Pour le sous-groupe d'indice fini $D_{v_0}\sbt k_{v_0}^{\times}$, il existe un entier $N>0$ tel que
$$ \bigcup_{i\in \BZ}(1+\pi_{v_0}^N\CO_{v_0})\cdot \pi_{v_0}^{iN}\sbt D_{v_0}.$$
Par approximation forte sur  $\BA^1$, il existe $\alpha , \beta \in k^{\times}$ tels que
$$\alpha\in k_{v_0}\times \prod_{v\in S}\{x\in k_v\ :\ |x-a_v|<\frac{1}{2}\epsilon\}\times\prod_{v\notin S\cup \{v_0\} }\CO_v$$
et 
$$\beta\in k_{v_0}\times \prod_{v\in S}\{x\in k_v\ :\ |x|<\frac{1}{2}\epsilon\}\times\prod_{v\notin S\cup \{v_0\} }\CO_v.$$
D'apr\`es la formule du produit, $l:=-v_0(\beta)>0$ et donc $\pi_{v_0}^l\beta\in \CO_{v_0}^{\times}$.
Alors il existe $m\in \BZ$ assez grand tel que 
$$(\beta \pi_{v_0}^l)^m\in (1+\pi_{v_0}^N\CO_{v_0})\ \ \ \text{et}\ \ \ v_0(\alpha \beta^{-m})>N. $$
Ainsi 
$$\alpha+\beta^{mN}=\pi_{v_0}^{-mlN}(\beta\pi_{v_0}^l)^{mN}(1+\alpha \beta^{-mN})\in  (1+\pi_{v_0}^N\CO_{v_0})\pi_{v_0}^{-mlN}\sbt D_{v_0}.$$
Ceci implique que $ \alpha+\beta^{mN}\in D_{v_0}\times W$.

\'Etape (3). Supposons que $v_0<\infty_k$, $T\cong \BG_m^n$ et, sous cet isomorphisme, $D_{v_0}\cong  \prod_{i=1}^nD_i$ avec $D_i\sbt k_{v_0}^{\times}$ un sous-groupe ouvert d'indice fini. 
On \'etabli le r\'esultat en utilisant projection $\BA^n\to \BA^1$ sur le premier facteur.
L'argument donn\'e pour $v_0$ r\'eel dans \cite[Prop. 4.6]{CX1} vaut pour tout $v_0$ en rempla\c{c}ant $\BR$ par $D_1$.

\'Etape (4). En g\'en\'eral, d'apr\`es (1) et (3), il reste \`a montrer qu'il existe
 des sous-groupes ouverts $\{D_i\}_{i=1}^n$ de $k_{v_0}^{\times}$ d'indice fini et un isomorphisme 
$$\psi:\ \Res_{K/k}\BA^1\to \BA^n\ \ \ \text{ tels que}\ \ \  \prod_{i=1}^nD_i\sbt \psi (D_{v_0}). $$
Si $v_0$ est r\'eel, ceci est \'etabli dans la d\'emonstration de \cite[Cor. 4.7]{CX1}.
On peut alors supposer que $v_0<\infty_k$ et $K=k(\theta)$ soit un corps.

Soit $\{w_j\}_j$ les  places de $K$ au-dessus de  $v_0$ et, pour tout $j$, 
soient $\pi_j$ une uniformisante de $K_{w_j}$ et $\CO_j$ l'anneau des entiers de $K_{w_j}$
Pour  $\pi_{v_0}$  une uniformisante de $k_{v_0}$ et pour chaque $j$,
soit $e_j:=w_j(\pi_{v_0})\leq n$. 
Puisque $D_{v_0}$ est un sous-groupe ouvert de $T(k_{v_0})\cong \prod_jK_{w_j}^{\times}$,
 il existe un entier $N\in \BZ_{>0}$ tel que
\begin{equation} \label{strang-e2}
 \prod_j [\bigcup_{l\in \BZ} (1+\pi_j^N\CO_j )\pi_j^{lN}]\sbt D_{v_0}. 
 \end{equation}
Apr\`es avoir remplac\'e $\theta$ par $\theta+\pi_{v_0}^{-cN} $ avec $c\gg 0$ suffisamment divisible,
on peut supposer que 
\begin{equation}\label{strang-e1}
w_j(\theta)=-e_jcN,\ \ \ \theta\pi_j^{e_jcN}\in (1+\pi_j^N\CO_j)\ \ \ \text{et}\ \ \  (\pi_{v_0}\pi_j^{-e_j})^c\in (1+\pi_j^N\CO_j).
\end{equation}
Soit $M:=cn^2N$.
Alors, pour tout $j$ et tous entiers $m,m'$ avec $0\leq m< m'\leq n-1$,
on a 
$$|m'-m ||w_j(\theta)|<ne_jcN \leq M\ \ \ \text{ et donc }\ \ \  0<(w_j(\theta^m)-w_j(\theta^{m'}))<M .$$
Donc, pour tous $\alpha,\alpha'\in k_{v_0}^{\times}$ satisfaisant $M|v_0(\alpha) $ et $M|v_0(\alpha')$,
on a 
\begin{equation}\label{strang-e}
w_j(\alpha \theta^{m})\neq w_j(\alpha' \theta^{m'})\ \ \ \text{et}\ \ \   
\begin{cases} w_j(\alpha')>w_j(\alpha) & \text{si} \ \ \ w_j(\alpha \theta^{m})< w_j(\alpha' \theta^{m'})\\
 w_j(\alpha)\geq w_j(\alpha')  & \text{si}\ \ \  w_j(\alpha \theta^{m})> w_j(\alpha' \theta^{m'}) 
\end{cases}.
\end{equation}
Donc
\begin{equation}\label{strang-e3}
\begin{cases}  w_j(\alpha' \theta^{m'})-w_j(\alpha \theta^{m})>M- (m'-m)(-w_j(\theta)) & \text{si} \ \ \ w_j(\alpha \theta^{m})< w_j(\alpha' \theta^{m'})\\
 w_j(\alpha \theta^{m})- w_j(\alpha' \theta^{m'})> (m'-m)(-w_j(\theta))   & \text{si}\ \ \  w_j(\alpha \theta^{m})> w_j(\alpha' \theta^{m'}) 
\end{cases}.
\end{equation}

L'\'el\'ement $\theta$ induit un isomorphisme $\psi: \Res_{K/k}\BA^1\to \BA^n$, qui est d\'efini par 
$$\Res_{K/k}\BA^1(A)=\sum_{i=0}^{n-1}A\theta^i\xrightarrow{\psi} \BA^n(A)=A^n:\ \ 
\ \sum_{i=0}^{n-1}a_i\theta^i\mapsto (a_0,\cdots , a_{n-1})$$
pour toute $k$-alg\`ebre $A$.
Soit 
$$D:=\bigcup_{l\in \BZ} (1+\pi_{v_0}^N\CO_{v_0})\pi_{v_0}^{lM}\sbt k_{v_0}^{\times}$$
un sous-groupe ouvert d'indice fini. 
D'apr\`es (\ref{strang-e1}), on a $D\sbt \cup_{l\in \BZ} (1+\pi_j^N\CO_j )\pi_j^{lN}$.
Pour tout $j$ et tout $(x_0,\cdots, x_{n-1})\in D^n$, il existe un $0\leq i_j\leq n-1$ tel que
$$w_j(x_{i_j}\theta^{i_j})=\min_{0\leq i\leq n-1}\{w_j(x_i\theta^i)\}=w_j(\sum_{i=0}^{n-1}x_i\theta^i),$$
o\`u la deuxi\`eme \'egalit\'e d\'ecoule  par (\ref{strang-e}). 
Pour $i\neq i_j$, d'apr\`es (\ref{strang-e3}), on a
$$w_j(x_i\theta^i)-w_j(x_{i_j}\theta^{i_j})\geq
 \begin{cases} (i-i_j)e_jcN\geq N & \text{si} \ \ \ i>i_j\\
 M-(i_j-i)e_jcN\geq N  & \text{si}\ \ \  i<i_j .
\end{cases} $$
Alors
$$\psi^{-1}(x_0,\cdots, x_{n-1})=\sum_{i=0}^{n-1}x_i\theta^i \in x_{i_j}\theta^{i_j}(1+\pi_j^N\CO_j) \sbt  \bigcup_{l\in \BZ} (1+\pi_j^N\CO_j )\pi_j^{lN}.$$
D'apr\`es (\ref{strang-e2}), on a $D^n\sbt \psi (D_{v_0})$.
\end{proof}

 On rappelle la d\'efinition des vari\'et\'es toriques standards (\cite[D\'ef. 2.12]{CX})

\begin{defi}\label{deftoricstandard}
Soit $K$ une $k$-alg\`ebre finie s\'eparable. La vari\'et\'e torique standard par rapport \`a $K/k$ est la sous-vari\'et\'e torique
$(\Res_{K/k}\BG_m\hookrightarrow Z)$ de $(\Res_{K/k}\BG_m\hookrightarrow \Res_{K/k}\BA^1)$ avec 
$$Z:=\Res_{K/k}\BA^1\setminus [(\Res_{K/k}\BA^1\setminus \Res_{K/k}\BG_m)_{sing}].$$
\end{defi}

\begin{cor}\label{strangapptoricstand}
Soit $(T\hookrightarrow Z)$ une vari\'et\'e torique standard. 
Soient $v_0$ une place de $k$ et $D_{v_0}\sbt T(k_{v_0})$ un sous-groupe ouvert d'indice fini.
Pour tout ferm\'e $F\sbt Z$ de codimension $\geq 2$ et tout ouvert non vide $W\sbt (Z\setminus F)(\RA_k)$, si $(D_{v_0}\cdot W)\cap (Z\setminus F)(\RA_k)=W$, alors on a $T(k)\cap W\neq\emptyset$. 
\end{cor}

\begin{proof}
Puisque $\codim( \Res_{K/k}\BA^1\setminus Z,\Res_{K/k}\BA^1)\geq 2$,
une application du th\'eor\`eme \ref{strangapptoricstandthm} donne le r\'esultat.
\end{proof}

\subsection{Fibration sur une vari\'et\'e torique standard}

On a besoin d'une g\'en\'eralisation du th\'eor\`eme de Tchebotarev ``g\'eom\'etrique" (Ekedahl  \cite[Lem. 2.1]{E}):

\begin{lem}\label{lemfibration}
Soient $X$, $Y$, $Z$ des sch\'emas int\`egres de type fini sur $\CO_k$, et $Y\xrightarrow{p}X$, $X\xrightarrow{f}Z$ deux $\CO_k$-morphismes.
Supposons que $f$ et $f\circ p$ sont lisses \`a fibres g\'eom\'etriquement int\`egres et $p$ est fini \'etale galoisenne de groupe de Galois $\Gamma$. 
Soit $C$ une classe de conjugaison de $\Gamma$ et $c:=|C|/|\Gamma |$. 
Pour chaque place $v\in\Omega_k$ et chaque $z\in Z(k(v))$, soit $N_C(z)$ le nombre  de $x\in X_z(k(v))$ tel que le Frobenius $Fr_x$ de $x$ soit dans $C$. Alors 
$$\frac{N_C(z)}{|X_z(k(v))|}-c=O(|k(v)|^{-1/2})$$
o\`u la constante dans $O(-)$ ne d\'epend ni de $v$ ni de $z$.
\end{lem}

\begin{proof}
Le r\'esultat d\'ecoule de la d\'emonstration standard de \cite[Lem. 2.1]{E} avec la formule des traces de Lefschetz
$$\sum_{x\in X_z(k(v))}\chi (Fr_x)=\sum_{0\neq i\neq 2dim(X_z)}(-1)^iTr(Fr^*,H_c^I(X_{\bar{z}}),V_{\chi}),$$
o\`u $V_{\chi}$ est d\'efini dans  la d\'emonstration de \cite[Lem. 2.1]{E}.
Puisque $R^if_{!}V_{\chi} $ est constructible et $H^i_c(X_{\bar{z}},V_{\chi})=(R^if_{!}V_{\chi} )_{\bar{z}}$, 
les dimensions des $H^i_c(X_{\bar{z}},V_{\chi}) $ sont born\'ees uniform\'ement.
De plus, la valeur absolue de la valeur propre de $Fr^*$ en $H^i_c(X_{\bar{z}},V_{\chi}) $ est inf\'erieure ou \'egale \`a $|k(v)|^{\frac{i}{2}}$. 
Le reste de la d\'emonstration est la m\^eme que celle de  \cite[Lem. 2.1]{E}.
\end{proof}

\begin{lem}\label{lem2fibration}
Soient $Z$ un ouvert de $\BA^n$ satisfaisant $\codim (\BA^n\setminus Z,\BA^n)\geq 2$ et $V$ un ouvert de $Z$.
Alors la suite exacte (\ref{purity}) induit un isomorphisme
\begin{equation}
\Br_a(V)\xrightarrow{\partial}H^1(k,\Div_{Z_{\bk}\setminus V_{\bk}}(Z_{\bk})\otimes \BQ/\BZ).
\end{equation}
\end{lem}

\begin{proof}
On peut supposer que $Z\setminus V$ est lisse. 
Puisque $\codim (\BA^n\setminus Z,Z)\geq 2$,
on a $\Br_a(Z)=0$ et $H^i(Z_{\bk},\BQ/\BZ(1))=0$ pour $i=1,2$. 
Puisque $V(k)\neq\emptyset$, le morphisme 
$$H^3(Z,\BQ/\BZ(1))\to H^3(V,\BQ/\BZ(1))\oplus H^3(Z_{\bk},\BQ/\BZ(1))$$
 est donc injectif.
Le r\'esultat d\'ecoule de la suite exacte (\ref{purity}).
\end{proof}

\begin{thm}\label{thmfibration}
Soient $(T\hookrightarrow Z)$ une vari\'et\'e torique standard, $G$ un groupe lin\'eaire connexe et 
$G\xrightarrow{\varphi} T$ un homomorphisme surjectif  de noyau connexe.
Soient $X$ une $G$-vari\'et\'e lisse, g\'eom\'etriquement int\`egre et
 $X\xrightarrow{f}Z$ un $G$-morphisme lisse surjectif \`a fibres g\'eom\'etriquement int\`egres.
Notons $U:=X\times_ZT$. 
Soit $B\sbt \Br_G(U)$ (cf. (\ref{defBr_GX})) un sous-groupe fini.
Soient $v_0\in \Omega_k$ une place et $G(k_{v_0})^{0}\sbt  G(k_{v_0})$ un sous-groupe d'indice fini.
 Alors, pour tout ferm\'e $F\sbt X$ de codimension $\geq 2$ et tout ouvert $W$ de $(X\setminus F)(\RA_k)$ satisfaisant $W^{ \Br(X)\cap (B+f|_U^*\Br_1(T))}\neq \emptyset$,  il existe $t\in T(k)$ de fibre $F_t\sbt X_t$,   tel que
 $$\codim(F_t,X_t)\geq 2\ \ \ \text{et}\ \ \  (X_t\setminus F_t)(\RA_k)^B\cap (G(k_{v_0})^{0}\cdot W)\neq\emptyset.$$
\end{thm}

\begin{proof}
Puisque $f$ est lisse, il existe un ferm\'e $F_1\sbt Z$ tel que $f(X\setminus F)=Z\setminus F_1$.

Soient $\{C_i\}_{i\in I}$ les composantes connexes de $Z\setminus T$ et $D_i:=X\times_ZC_i$ les composantes connexes de $X\setminus U$. 
Par hypoth\`ese, les vari\'et\'es $C_i$ et $D_i$ sont lisses int\`egres.
D'apr\`es (\ref{purity}), on a une suite exacte:
$$0\to  \Br(X)\cap (B+f|_U^*\Br_1(T))\to   (B+f|_U^*\Br_1(T))\xrightarrow{\partial } \oplus_i H^1(D_i,\BQ/\BZ)$$
Par le lemme \ref{lem-fibreinvariant}, pour tout $i$, 
il existe  un rev\^etement fini \'etale galoisien ab\'elien $D_i'\xrightarrow{\pi_i}D_i$ tel que 
$D_i'$ soit une $G$-vari\'et\'e int\`egre, $\pi_i$ soit un $G$-morphisme et $\partial (b)\in H^1(D_i'/D_i,\BQ/\BZ)$.

Soient $k_i$ la fermeture int\'egrale de $k$ dans $k(C_i)$ et $k_i'$ la fermeture int\'egrale de $k$ dans $k(D_i')$.
Par hypoth\`ese,  la fermeture int\'egrale $k$ dans $k(D_i)$ est $k_i$. 
D'apr\`es \cite[Prop. 2.2]{CX1}, le morphisme $D_i'\to C_i\times_{k_i}k_i'$ est lisse \`a fibres g\'eom\'etriquement int\`egres.

Soit $B_X:= \Br(X)\cap (B+f|_U^*\Br_1(T))$.
 D'apr\`es le lemme \ref{brauertorseur}, $B_X$ est fini et $B_X\sbt \Br_G(X)$.
D'apr\`es le lemme \ref{lem2fibration}, on a un diagramme commutatif avec suite exacte:
\begin{equation}\label{fibratione2}
\xymatrix{&&B\ar[r]^-{\partial}\ar[d]& \oplus_i H^1(D_i'/k_i'D_i,\BQ/\BZ)\ar[d]^{\cong}\\
0\ar[r]& (B_X+f|_U^*\Br_1(T))\ar[r]& (B+f|_U^*\Br_1(T))\ar[r]^-{\bar{\partial}} &\oplus_i H^1((D_i'\times_{k_i'}\bk)/(D_i\times_{k_i}\bk)  ,\BQ/\BZ).
}
\end{equation}

Soit $\Br_G(X)\xrightarrow{\lambda}\Br_e(G)$ l'homomorphisme de Sansuc (cf: D\'efinition \ref{defsansuc}).
Apr\`es avoir r\'etr\'eci $G(k_{v_0})^0$, on peut supposer que tout \'el\'ement de $\lambda (B_X) $ s'annule sur $G(k_{v_0})^0$.
Par la proposition \ref{corbraueralgebraic1}, pour tout $x\in X(\RA_k)$, 
on a que tout \'el\'ement de $B_X $ est constant sur $G(k_{v_0})^0\cdot x$.

Soient $l:=\dim (Z)$, $\Gamma_i:=\Gal(D_i'/k_i'D_i)$ et $W_1:=(G(k_{v_0})^{0}\cdot W)\cap (X\setminus F)(\RA_k)$.

Puisque $T$ est quasi-trivial, il existe des extensions de corps $L_j/k$ telles que $T\cong \Res_{\prod_jL_j/k}\BG_m$.
En fait, $\prod_jL_j\cong \prod_ik_i$, mais on ne l'utilise pas.
 Soit $\Omega_0$ l'ensemble des places $v\in \Omega_k$ totalement d\'ecompos\'ees dans la fermeture galoisienne de $k_i'/k$ pour tout $i$ et de $L_j/k$ pour tout $j$. 
Alors $\Omega_0$ est infini et pour tout $v\in \Omega_0$, on a $T_{k_v}\cong \BG^l_{m,k_v}$. 
Par la d\'efinition \ref{deftoricstandard}, on a
$$Z_{k_v}\cong \Spec\ k_v[t_1,\cdots, t_l]\setminus \cup_{n\neq m}V(t_n,t_m)\ \ \ \text{et}\ \ \ 
T_{k_v}\cong \Spec\ k_v[t_1,t_1^{-1},\cdots,t_l,t_l^{-1}].$$

\medskip

Soit $S$ un sous-ensemble fini de $\Ok$ tel que $v_0\cup \infty_k\sbt S$.
On agrandit $S$ de fa\c{c}on \`a avoir les propri\'et\'es suivantes:

(a) Le $k$-morphisme $f$ s'\'etend en un $\CO_S$-morphisme lisse \`a fibres g\'eom\'etriquement int\`egres $\CX\to \CZ$ de $\CO_S$-sch\'emas lisses, tel que pour tout point ferm\'e $z\in \CZ\setminus \CF_1$, la fibre $f^{-1}(z)$ poss\`ede un $k(z)$-point $x\notin \CF$,
o\`u $\CF$ est l'adh\'erence de $F$ dans $\CX$ et $\CF_1$ est l'adh\'erence de $F_1$ dans $\CZ$. 
Ceci est  possible par les estim\'ees de Lang-Weil \cite[Thm. 1, \'etape 3]{sko1}.
Par le lemme de Hensel,  pour tout $v\notin S$, l'application $(\CX\setminus \CF)(\CO_v)\to (\CZ\setminus \CF_1)(\CO_v)$ est donc surjective.

(b) Les extensions $k_i/k$ et $k_i'/k_i$ induisent des rev\^etements finis \'etales $\CO_{k_i,S}/\CO_S$ et $\CO_{k_i',S}/\CO_{k_i,S}$.

(c) Le sous-sch\'ema $\CC_i:=\overline{C_i}\sbt \CZ$ est lisse \`a fibres g\'eom\'etriquement int\`egres sur $\CO_{k_i,S}$. 
Soient $\CD_i:=\CC_i\times_{\CZ}\CX$, $\CT:=\CZ\setminus\cup_i \CC_i$ et $\CU:=\CT\times_{\CZ}\CX$.

(d) Les \'el\'ements de $B_X$ appartiennent \`a $\Br(\CX)$ et les \'el\'ements de $B$ appartiennent \`a $\Br(\CU)$.

(e)  Le rev\^etement $D_i'\xrightarrow{\pi_i}D_i$ s'\'etend en un $\CO_S$-rev\^etement fini \'etale galoisien ab\'elien $\CD_i'\to \CD_i$ tel que $\CD_i'$ soit un sch\'ema sur $\CO_{k_i',S}$,
 les r\'esidus des \'el\'ements de $B$ soient dans $H^1(\CD_i'/\CD_i,\BQ/\BZ)$ et
 $\CD_i'\to \CC_i\times_{\CO_{k_i,S}}\CO_{k_i',S}$ soit lisse \`a fibres g\'eom\'etriquement int\`egres. 

(f) Il existe un ouvert $W_2\sbt W_1\sbt (X\setminus F)(\RA_k)$ tel que 
$$W_2^{B_X}=W_2\neq\emptyset,\ \ \  W_2=(G(k_{v_0})^{0}\cdot W_2)\cap (X\setminus F)(\RA_k),\ \ \ \text{et}\ \ \ W_2=W_{v_0}\times W_{S\setminus v_0}\times \prod_{v\notin S}(\CX\setminus \CF)(\CO_v)$$
avec $W_{S\setminus v_0}\sbt \prod_{v\in S\setminus \{v_0\}}(X\setminus F)(k_v)$ un ouvert et $W_{v_0}\sbt (X\setminus F)(k_{v_0})$ un ouvert. 

(g) Pour tout $v\in \Omega_0\setminus (\Omega_0\cap S)$, 
on note $\CC_{i,v}:=\CC_i\times_{\CO_{k_i,S}}\CO_v $, $\CD_{i,v}:=\CD_i\times_{\CO_{k_i,S}}\CO_v $ et $\CD'_{i,v}:=\CD'_i\times_{\CO_{k'_i,S}}\CO_v $.
On a un diagramme commutatif de sch\'emas int\`egres:
\begin{equation}\label{fibratione1}
\xymatrix{\CD'_{i,v}\ar[r]\ar[d]\ar@{}[rd]|{\square}&\CD_{i,v}\ar[r]\ar[d]\ar@{}[rd]|{\square}&
\CC_{i,v}\ar[r]\ar[d]\ar@{}[rd]|{\square}& \CO_v\ar[d]\\
\CD_i'\ar[r]& \CD_i\times_{\CO_{k_i,S}}\CO_{k_i',S}\ar[r]&  \CC_i\times_{\CO_{k_i,S}}\CO_{k_i',S}\ar[r]& \Spec\ \CO_{k_i',S}.
}\end{equation}
De plus, on a $\Gamma_i\cong \Gal(\CD_i'/(\CD_i\times_{\CO_{k_i,S}}\CO_{k_i',S}))\cong \Gal(\CD_{i,v}'/\CD_{i,v})$.

(h) Pour tout $v\in \Omega_0\setminus (\Omega_0\cap S)$, tout $\sigma\in \Gamma_i$ et tout $c\in \CC_{i,v}(k(v)) $ avec $c\notin \CF_1$, la fibre $(\CD_{i,v})_c$ poss\`ede un $k(v)$-point $d$ avec $d\notin \CF$ dont le Frobenius est $\sigma$.
Ceci est possible en appliquant le lemme \ref{lemfibration} \`a (\ref{fibratione1}).

(i) Pour tout $v\in \Omega_0\setminus (\Omega_0\cap S)$, on note $(-)_{\CO_v}:=(-)\times_{\CO_{k,S}}\CO_v$ et on a
$$\CZ_{\CO_v}\cong \Spec\ \CO_v[t_1,\cdots ,t_l ]\setminus \cup_{n\neq m}V(t_n,t_m),\ \ \ 
\CT_{\CO_v}\cong \Spec\ \CO_v[t_1, t_1^{-1},\cdots ,t_l,t_l^{-1} ]$$
et il existe une partition $\{ 1,\cdots ,l\}=\coprod_i I_i$
 telle que $\CC_{i,\CO_v}=\cup_{n\in I_i}V(t_n)\sbt \CZ$ et que $V(t_n)\cong  \CC_{i,v}$.

\medskip

Pour chaque $i$, on choisit  une place $v_i\in \Omega_0\setminus (\Omega_0\cap S)$ et un $n_i\in I_i$ tels que pour $i\neq j$, on ait $v_i\neq v_j$.
Soient $\CC_{i,n_i}:=V(t_{n_i})\sbt \CC_{i,\CO_{v_i}}$, $\CD_{i,n_i}:=\CC_{i,n_i}\times_{\CZ}\CX$ et
$$E_i:=\{(t_1,\cdots ,t_l)\in \CO_{v_i}^l: t_{n_i}\in m_{v_i}\setminus m^2_{v_i}\ \ \ \text{et}\ \ \ t_n\in \CO_{v_i}^{\times}\ \ \ \text{pour} \ \ \ n\neq n_i   \}$$
 un ouvert de $  \CZ(\CO_{v_i})$. 
Alors  $\CC_{i,n_i}\cong \CC_{i,v_i}$, $\CD_{i,n_i}\cong \CD_{i,v_i}$ et 
$\CD_i'\times_{\CD} \CD_{i,n_i} \cong \bigsqcup_{[k_i':k_i]} \CD_{i,v_i}'$.
Donc le Frobenius en un point de $\CD_{i,n_i}(k(v_i))$ pour le rev\^etement $\CD_i'/\CD_i$ est dans $\Gamma_i$.

  Pour tout $b\in B$ et tout $P_i\in \CX(\CO_{v_i})$ avec $f(P_i)\in E_i$,
  on a  $\bar{P_i}:=P_i(k(v_i))\in \CD_{i,n_i}(k(v_i))$ et alors la formule \cite[Cor. 2.4.3 et p. 244-245]{Ha94} (voir \cite[Formule (3.6)]{CTH})
\begin{equation}\label{fibratione3}
b(P_i)=\partial_i (b)(Fr_{\bar{P_i}})\in \BQ/\BZ,
\end{equation} 
 o\`u  $Fr_{\bar{P_i}}\in \Gamma_i\sbt  \Gal(\CD_i'/\CD_i)$ est le Frobenius en $\bar{P_i}$ pour le rev\^etement $\CD_i'/\CD_i$
 et 
 $$\partial_i: \Br(\CU)\xrightarrow{\partial}H^1(\CD_i'/\CD_i,\BQ/\BZ)\to H^1(\CD_{i,v_i}'/\CD_{i,v_i},\BQ/\BZ) =\Hom(\Gamma_i,\BQ/\BZ).$$
Par   fonctorialit\'e, l'application $\bar{\partial}$ de (\ref{fibratione2}) satisfait $\bar{\partial}=\oplus_i \partial_i$.

\medskip

Notons $T(k_{v_0})^0:=\varphi (G(k_{v_0})^0) $.
Puisque $H^1(k_{v_0},\Ker(\varphi))$ est fini (\cite[Thm. 6.14]{PR}), le sous-groupe $T(k_{v_0})^0\sbt T(k_{v_0})$ est d'indice fini.
D'apr\`es \cite[Thm. 4.5]{Co}, $f(W_2)$ est un ouvert de $(Z\setminus F_1)(\RA_k)$.
Pour tout $z\in (Z\setminus F_1)(k_{v_0})$, l'ouvert $(X_z\setminus F_z)(k_{v_0})$ est dense dans $X_z(k_{v_0})$.
Donc 
$$(T(k_{v_0})^0\cdot f(W_{v_0}))\cap (Z\setminus F_1)(k_{v_0})=f(W_{v_0})\ \ \ \text{et}\ \ \ 
(T(k_{v_0})^{0}\cdot f(W_2))\cap (Z\setminus F_1)(\RA_k)=f(W_2).$$
Par le corollaire \ref{strangapptoricstand}, il existe $t\in T(k)\cap f(W_2)$ tel que $t|_{v_i}\in E_i$ pour tout $i$ et que $\codim(F_t,X_t)\geq 2$. 
Alors il existe $(P_v)\in W_2$ tel que $f(P_v)=t$.
 
 Soit $t_i\in \CZ(k(v_i))$ la sp\'ecialisation de $t$, alors $t_i\in \CC_{i,n_i}(k(v_i))$ et $t_i\notin \CF_1$.
D'apr\`es (\ref{fibratione2}), on a un diagramme avec suite exacte:
$$\xymatrix{\oplus_i ((\CD_{i,n_i})_{t_i}\setminus \CF_{t_i})(k(v_i))\ar@{->>}[d]^{Fr}&(X_t\setminus F_t)(\RA_k)\cap W_2\ar[d]^{a_U}&\\
\oplus_i\Gamma_i\ar[r]^{\bar{\partial}^D}&(B+f|_U^*\Br_1(T))^D\ar[r]^{\Res}&(B_X+f|_U^*\Br_1(T))^D,
}$$
o\`u  $a_U$ est l'accouplement de Brauer-Manin, $(-)^D:=\Hom(-,\BQ/\BZ)$ 
et, pour $u\in f^{-1}(t_i)$, l'\'el\'ement $Fr(u)$ est son Frobenius. Donc $Fr$ est surjectif par (h). 
Puisque $\Res\circ a_U=0$, il existe $\{\sigma_i\}_i\in \oplus_i\Gamma_i$ tel que $\bar{\partial}^D(\{\sigma_i\})=a_U(\{P_v\})$.
Alors il existe $u_i\in ((\CD_{i,n_i})_{t_i}\setminus \CF_{t_i})(k(v_i))$ tel que $Fr(u_i)=Fr(\bar{P_{v_i}})-\sigma_i$.
Par le lemme de Hensel, il existe un point $Q_{v_i}\in (\CX_t\setminus \CF_t)(\CO_{v_i})$  relevant  $u_i$. Soit $Q_v:=P_v$ pour tout $v$ distinct de l'un des $v_{i}$.
Par (\ref{fibratione3}), $\{Q_v\}$ satisfait les conditions.
\end{proof}

 \subsection{Fibration sur un tore}
 
 \begin{lem}\label{pointouvert}
 Soient $X$ et $Z$ deux vari\'et\'es lisses g\'eom\'etriquement int\`egres, et $X\xrightarrow{f}Z$ un morphisme lisse surjectif \`a fibres g\'eom\'etriquement int\`egres. Soit $U\sbt X$ un ouvert tel que $f|_U$ soit surjectif. Soient $W\sbt X(\RA_k)$ un ouvert et $x\in W$. Alors il existe $u\in W\cap U(\RA_k)$ tel que $f(u)=f(x)$.
 \end{lem}
 
 \begin{proof}
 Pour chaque $z\in Z$, la fibre $X_z$ est lisse int\`egre et l'ouvert $U_z\sbt X_z$ est donc dense.
 Soit $\{z_v\}_v=f(x)$. 
 Pour chaque $v$, l'ouvert $U_{z_v}(k_v)$ est dense en $X_{z_v}(k_v)$.
Puisque $f|_U$ est surjectif, le morphisme $f|_U$ est lisse \`a fibres g\'eom\'etriquement int\`egres.
 Apr\`es avoir fix\'e un mod\`ele int\`egre $\CU\rightarrow \CZ$ de $f|_U$, on a que, pour presque toute place $v$, le morphisme $\CU(\CO_v)\to \CZ(\CO_v)$ est surjectif. Le r\'esultat en d\'ecoule. 
 \end{proof}
 
 Le th\'eor\`eme suivante, d'\'enonc\'e un peu technique, joue un r\^{o}le cl\'e dans la d\'emonstration du
  th\'eor\`eme \ref{main1thm}.
 
 \begin{thm}\label{mainlem}
 Soient $T$, $T_0$ deux tores avec $T_0$ quasi-trivial, $G$ un groupe lin\'eaire connexe,  $G\xrightarrow{\varphi} T_0\times T$ un homomorphisme surjectif de noyau connexe et  $G_0\sbt G$ un sous-groupe ferm\'e connexe.
Soient $X$ une $G$-vari\'et\'e lisse g\'eom\'etriquement int\`egre,  $U\sbt X$ un $G$-ouvert et $U\xrightarrow{f}T_0\times T$ un $G$-morphisme. Soit $B\sbt \Br_G(U)$ (cf. (\ref{defBr_GX})) un sous-groupe fini.
Supposons que: 

(1)  la composition
$ T_0^* \xrightarrow{p_1^*}T_0^*\times T^* \xrightarrow{f^*} \bk[U]^{\times}/\bk^{\times} \xrightarrow{\div_X} \Div_{X_{\bk}\setminus U_{\bk}}(X_{\bk})$
est un isomorphisme;

(2) pour l'action de $G_0$ sur $X$, le morphisme $\bk[X]^{\times}/\bk^{\times} \to \bk[G_0]^{\times}/\bk^{\times} $ d\'efini par Sansuc (\cite[(6.4.1)]{S}) est injectif.

Alors, pour tout $v_0\in \Omega_k$, tout sous-groupe ouvert d'indice fini $G(k_{v_0})^0\sbt G(k_{v_0})$ et
tout ouvert $W\sbt X(\RA_k)$ satisfaisant $W^{ \Br(X)\cap (B+ f^*\Br_1(T_0\times T))}\neq\emptyset$,
il existe $t\in (T_0\times T)(k)$ de fibre $U_t$, tel que 
$$(G(k_{v_0})^0\cdot G_0(k_{\infty})^+\cdot W)\cap U_t(\RA_k)^{B}\neq\emptyset .$$
 \end{thm}
 
 \begin{proof}
 Par la proposition \ref{propextension}, apr\`es avoir remplac\'e $f$ par $\widetilde{\phi} \circ f$ et $\varphi $ par $\widetilde{\phi} \circ\varphi$ avec $\widetilde{\phi} $ un automorphisme de $T_0\times T$, on peut supposer que:
 
 (i) il existe une vari\'et\'e torique $(T_0\hookrightarrow \BA^l)$ satisfaisant (\ref{exten-e1});
 
 (ii) le morphisme $f$ s'\'etend \`a  un $G$-morphisme $X\xrightarrow{f_X} Z$ o\`u $Z:=\BA^l\times T\supset T_0\times T$;
 
 (iii) on a $f_X(U)\sbt T_0\times T$ et un isomorphisme
 $\Div_{Z_{\bk}\setminus (T_0\times T)_{\bk}}(Z_{\bk})\xrightarrow{f_X^*}\Div_{X_{\bk}\setminus U_{\bk}}(X_{\bk}).$ 
 
 Soit $Z_0:=\BA^l\setminus [(\BA^l\setminus T_0)_{sing}]$. Alors $(T_0\hookrightarrow Z_0)$ est une vari\'et\'e torique standard et
 $$Z_1:=Z_0\times T\cong Z\setminus [(Z\setminus T_0\times T)_{sing}].$$
 D'apr\`es la proposition \ref{propopensubset}, il existe un $G$-ouvert $X_1\sbt X$ tel que 
 $f(X_1)\sbt Z_1$, $X_1\cap f_X^{-1}(T)=U$, $\codim(X\setminus X_1,X)\geq 2$, $\Br(X)\cong \Br(X_1)$ 
 et que   le morphisme $X_1\xrightarrow{f_X|_{X_1}}Z_1$ soit lisse surjectif \`a fibres g\'eom\'etriquement int\`egres. 
   
Notons  $\phi: X\xrightarrow{f_X}\BA^l\times T\xrightarrow{p_2}T$.
 Pour chaque $t\in T(k)$, notons $X_{1,t}:=\phi^{-1}(t)\cap X_1$, $U_t:=U\cap X_{1,t}$,  $X_{1,t}\xrightarrow{i_t}X_1$ et $X_{1,t}\xrightarrow{f_t}Z_0$. 
 On a le diagramme:
 $$\xymatrix{X_{1,t}\ar@{^{(}->}[r]^{i_t}\ar[d]^{f_t}&X_1\ar@{^{(}->}[r]\ar[d]&X\ar[d]_{f_X}\ar@/^2pc/[dd]^{\phi}&G\ar[d]^{\varphi}\\
 Z_0\times t\ar[d]\ar@{^{(}->}[r]&Z_1\cong Z_0\times T\ar@{^{(}->}[r]\ar[d]&\BA^l\times T\ar[d]&T_0\times T\ar[d]^{p_2}\\
 t\ar@{^{(}->}[r]&T\ar[r]^=&T&T.
 }$$
 
On a les propri\'et\'es ci-dessous:

(a) le morphisme $f_t$ satisfait les hypoth\`eses g\'eom\'etriques du th\'eor\`eme \ref{thmfibration} 
  par rapport \`a $\Ker(G\xrightarrow{p_2\circ \varphi}T)\to T_0 $;

(b) l'homomorphisme $G_0\xrightarrow{p_2\circ \psi}T$ est surjectif et donc $(p_2\circ \psi)(G_0(k_{\infty})^+)=T(k_{\infty})^+$;

(c) Soient $B_1:=B+f^*\Br_1(T_0\times T)$ et $B_2\sbt (\Br(X)\cap B_1)$ un sous-groupe fini tels que le morphisme
$B_2\to \frac{\Br(X)\cap B_1}{\Br(X)\cap f^*\Br_1(T_0\times T)} $ soit surjectif, alors $ i_t^*B\cap \Br(X_{1,t})\sbt i_t^*B_2$.
 
 L'\'enonc\'e (a) est claire. 
 
Pour (b), d'apr\`es \cite[Prop. 2.2]{CX1}, $\phi$ est lisse surjectif \`a fibres g\'eom\'e\-triquement int\`egres.
Donc $\bk[T]^{\times}/\bk^{\times} \xrightarrow{\phi^*}\bk[X]^{\times}/\bk^{\times} $ est injectif.  
Notons $\bk[X]^{\times}/\bk^{\times} \xrightarrow{\theta_X} \bk[G_0]^{\times}/\bk^{\times} $ et 
$\bk[T]^{\times}/\bk^{\times} \xrightarrow{\theta_T} \bk[G_0]^{\times}/\bk^{\times} $ les morphismes
d\'efinis par Sansuc (\cite[(6.4.1)]{S}).
Ainsi $\theta_T=\theta_X\circ \phi^*$ est injectif.
Par l'argument de Sansuc (\cite[P. 39]{S}),
$\theta_T=(( p_2\circ \psi)|_{G_0})^*$.
Alors $G_0\xrightarrow{p_2\circ \psi}T$ est surjectif.

Pour (c), d'apr\`es le lemme \ref{brauertorseur}, on a
$$
  \Br(X_1)\cap f^*\Br_1(T_0\times T)\cong f_X^*\Br_1(Z_1)\cong \phi^*\Br_1(T)\ \ \ \text{et}\ \ \  B_1\cap \Br(X_1)=B_2+\phi^*\Br_1(T).
$$
Par  le corollaire \ref{prop-fibreinvariant}, on a 
$$
i_t^*B_2=i_t^*(B_2+\phi^*\Br_1(T))=i_t^*(B_1\cap \Br(X_1))=i_t^*B_1\cap \Br(X_{1,t})\supset i_t^*B\cap \Br(X_{1,t}).
$$
Ceci donne (c).
 
\medskip 
 
On consid\`ere l'ouvert $W$ de l'\'enonc\'e.  Apr\`es avoir r\'etr\'eci $W$, on peut supposer que tout \'el\'ement de $B_2$ s'annule sur $W$. 

 On note $W_1:=W\cap X_1(\RA_k)$. Soit $x\in W^{\phi^*\Br_1(T)}$.
En appliquant le lemme \ref{pointouvert} au triple $(X_1\sbt X,X\xrightarrow{\phi}T, W)$, 
on voit qu'il existe $x_1\in W_1$, tel que $\phi(x_1)=\phi(x)$.
Donc $W_1^{\phi^*\Br_1(T)}\neq \emptyset$ et $\phi (W_1)^{\Br_1(T)}\neq \emptyset$.

Soit $W_2:= G_0(k_{\infty})^+ \cdot W_1$.
D'apr\`es (b), on a $\phi (W_2)=T(k_{\infty})^+\cdot \phi (W_1)$.
Puisque $T$ satisfait l'approximation forte par rapport \`a $\Br_1(T)$ hors de $\infty_k$ (Harari \cite[Thm. 2]{Ha08}),
il existe $t\in T(k)\cap \phi(W_2)$. Donc $X_{1,t}(\RA_k)\cap W_2\neq\emptyset$. 
D'apr\`es (c), 
$(X_{1,t}(\RA_k)\cap W_2)^{i_t^*B\cap \Br(X_{1,t})}\neq\emptyset.$

Soit $W_3:=G(k_{v_0})^0\cdot W_2\supset W_2$.
 D'apr\`es le th\'eor\`eme \ref{thmfibration}, il existe 
  $u\in W_3\cap U_t(\RA_k)^{i_t^*B}$ tel que $f_t(u)\in T_0(k)$.
 \end{proof}

 \begin{cor}\label{mainlemcod2}
 Avec les hypoth\`eses et notations du th\'eor\`eme \ref{mainlem}, soit $F\sbt X$ un sous-sch\'ema ferm\'e $G_0$-invariant de codimension $\geq 2$.
Alors, pour tout $v_0\in \Omega_k$ et tout $\tilde{W}\sbt (X\setminus F)(\RA_k)$ satisfaisant $\tilde{W}^{ \Br(X)\cap (B+ f^*\Br_1(T_0\times T))}\neq\emptyset$, il existe un $t\in (T_0\times T)(k)$ tel que 
$$\codim(F\cap U_t,U_t)\geq 2\ \ \ \text{et}\ \ \ (G(k_{v_0})^0\cdot G_0(k_{\infty})^+\cdot \tilde{W})\cap (U_t\setminus (F\cap U_t))(\RA_k)^{B}\neq\emptyset.$$
 \end{cor}
 
 \begin{proof}
 Avec les constructions et notations de la d\'emonstration du th\'eor\`eme \ref{mainlem}, 
 soit 
 $$\tilde{W}_1:=\tilde{W}\cap (X_1\setminus F)(\RA_k),\ \ \  \tilde{W}_2:=G_0(k_{\infty})^+\cdot \tilde{W}_1\ \ \ \text{et}\ \ \  
 \tilde{W}_3:=(G(k_{v_0})^0\cdot \tilde{W}_2)\cap (X_1\setminus F)(\RA_k).$$
 Le r\'esultat d\'ecoule du m\^eme argument que dans la d\'emonstration du th\'eor\`eme \ref{mainlem},
  en rempla\c{c}ant $W_1,W_2,W_3$ par $\tilde{W}_1, \tilde{W}_2, \tilde{W}_3$.
 \end{proof}

\subsection{Fibration sur un pseudo espace homog\`ene }

Soit $G$ un groupe lin\'eaire connexe. 
 Rappelons la notion de pseudo $G$-espace homog\`ene (cf. D\'efinition \ref{defpseudo}).
 Soit $Z$ un pseudo $G$-espace homog\`ene, 
 on peut d\'efinir son quotient torique maximal $Z\xrightarrow{\pi} Z^{tor}$ et le stabilisateur de $G$ sur $Z^{tor}$
 (cf. D\'efinition \ref{defitorequotient}).

\begin{thm}\label{main1thm}
 Soient $G$ un groupe lin\'eaire connexe, $Z$ un pseudo $G$-espace homog\`ene,
 $Z\xrightarrow{\pi} Z^{tor}$  le quotient torique maximal et $G_0$ le stabilisateur de $G$ sur $Z^{tor}$.
 Soit $X$ une $G$-vari\'et\'e lisse g\'eom\'etriquement int\`egre telle que $\bk[X]^{\times}/\bk^{\times}=0$.
  Soient $U\sbt X$ un $G$-ouvert et $U\xrightarrow{f}Z$ un $G$-morphisme. 
 Soient $A\sbt \Br(X)$, $B\sbt \Br_G(U)$ (cf. (\ref{defBr_GX})) deux sous-groupes finis.
Pour tout ouvert $W\sbt X(\RA_k)$ satisfaisant  $W^{\Br(X)\cap (A+B+ f^*\Br_G(Z))}\neq\emptyset$, on a:

(1) pour toute place $v_0\in \Omega_k$ et tout sous-groupe ouvert d'indice fini $G(k_{v_0})^0\sbt G(k_{v_0}) $,
 il existe un $t\in Z^{\tor}(k)$ de fibre $U_t\xrightarrow{f_t}Z_t$,  tel que 
$$(G(k_{v_0})^0\cdot W)\cap U_t(\RA_k)^{A+B+f_t^*\Br_{G_0}(Z_t)}\neq\emptyset;$$

(2) s'il existe un sous-ensemble fini non vide $S\sbt \Omega_k$ tel que, 
pour tout sous-groupe ouvert d'indice fini $G_0(k_S)^0\sbt G_0(k_S)$  et tout $t\in Z^{tor}(k)$ de fibre $Z_t$, 
l'adh\'erence  $\overline{G_0(k_S)^0\cdot Z_t(k)}$ contient $Z_t(\RA_k)^{\Br_{G_0}(Z_t)}$, 
alors, pour tout sous-groupe ouvert d'indice fini $G(k_S)^0\sbt G(k_S)$,
 il existe un $z\in Z(k)$ de fibre $U_z$ tel que 
 $(G(k_S)^0\cdot W)\cap U_z(\RA_k)^{A+B}\neq\emptyset. $
 \end{thm}

\begin{proof}
On consid\`ere (1).

Soient $T_0$ un tore tel que $T_0^*\cong \Div_{X_{\bk}\setminus U_{\bk}}(X_{\bk})$ et $Y_0\to X$ le $T_0$-torseur induit par l'homomorphisme $ \Psi$ de  la suite exacte (\ref{tor-e1}).

D'apr\`es la proposition \ref{proptorseurcoflasque}, il existe un tore quasi-trivial $T_1$ et un $T_1$-torseur $Z_1\xrightarrow{p_{1,Z}} Z$
 tels que $\Pic(Z_{1,\bk})=0$ et $H^3(k,\bk[Z_1]^{\times}/\bk^{\times})=0$.
 Par le th\'eor\`eme \ref{thmaction} et le corollaire \ref{coraction}, 
 il existe un groupe lin\'eaire connexe $H$ muni d'un homomorphisme surjectif $H\to G$ de noyau central $T_1$ 
 tel que $Z_1$ soit une $H$-vari\'et\'e et que $p_{1,Z}$ soit un $H$-morphisme.
De plus, $Z_1(k)\neq\emptyset$ et, d'apr\`es la proposition \ref{lemtoretorseur}, $Z_1$ est un pseudo $H$-espace homog\`ene.

Soient $T_2:=T_0\times T_1$ et $V_1:=U\times_ZZ_1$ un $T_1$-torseur sur $U$.
 Puisque $H^1(X,T_1)\to H^1(U,T_1)$ est surjectif, il existe un $T_1$-torseur $Y_1\to Y$ tel que $[Y_1]_U=[V_1]$.
L'isomorphisme canonique
 $H^1(X,T_0)\oplus H^1(X,T_1)\xrightarrow{\theta}H^1(X,T_2)$
donne un $T_2$-torseur $Y\to X$ tel que $[Y]=\theta ([Y_0],[Y_1])$.
 Maintenant on obtient des $T_1$-torseurs $Z_1\to Z$, $V_1\to U$ et des $T_2$-torseurs $Y\to X$, $V\to U$ 
 tels que $f^*[Z_1]=[V_1]$, $[Y]|_U=[V]$ et $[V]=[T_0\times V_1]$. 
 
 Par le th\'eor\`eme \ref{thmaction}, le corollaire \ref{cor1action} et le corollaire \ref{coraction}, 
 il existe un homomorphisme surjectif $T_0\times H\xrightarrow{\psi}G$ de noyau central $T_2$ et un diagramme commutatif de $T_0\times H$-vari\'et\'es et de $T_0\times H$-morphismes:
 \begin{equation}
 \xymatrix{ Y\ar[d]^p\ar@{}[rd]|{\square}&V\ar@{_{(}->}[l]\ar[d]^p\ar[r]_-{\tau}\ar@/^1pc/[rrr]^{f_V}&
 T_0\times V_1\ar[r]_-{f_1} \ar[d]\ar@{}[rd]|{\square}&
 T_0\times Z_1\ar[d]^{p_Z}\ar[r]_-{id\times \pi_1}& T_0\times Z_1^{\tor}\ar[d]^{p_{Z^{tor}}}\\
X& U\ar@{_{(}->}[l]\ar[r]^=&U\ar[r]^f& Z\ar[r]^-{\pi}&Z^{tor},
 }
 \end{equation}
 o\`u  $\tau$ est une trivialisation, $Z_1\xrightarrow{\pi_1}Z_1^{tor}$ est le quotient torique maximal, $f_V:=(id\times\pi_1)\circ f_1\circ \tau$ est la composition et $p_Z:=p_{1,Z}\circ p_2$.  
 
\medskip 
 
 On a des propri\'et\'es:
 
 (a) on peut supposer que la composition
$$T_0^*\xrightarrow{p_1^*}\bk[T_0\times V_1]^{\times}/\bk^{\times}\xrightarrow{\tau^*}\bk[V]^{\times}/\bk^{\times}\xrightarrow{\div} \Div_{Y_{\bk}\setminus V_{\bk}}(Y_{\bk})$$
est un isomorphisme. 

(b) le stabilisateur $H_0$ de $H$ sur $Z_1^{tor}$ est connexe et donc
 les morphismes $f_V$, $\pi$, $\pi_1$ et $\pi\circ f$ sont lisses \`a fibres g\'eom\'e\-triquement int\`egres (\cite[Prop. 2.2]{CX1}).

(c) on a  
\begin{equation}\label{main1thm-e1}
X(\RA_k)^{\Br(X)\cap (B+ f^*\Br_G(Z))}=p(Y(\RA_k)^{\Br(Y)\cap [p^*B+ (f_1\circ \tau )^*\Br_{T_0\times H}(T_0\times Z_1)]}).
\end{equation}

(d) il existe un sous groupe fini $B_1\sbt \Br_{T_0\times H}(V)$ tel que 
$$B_1+f_V^*\Br_1(T_0\times Z_1^{tor})=(f_1\circ \tau )^*\Br_{T_0\times H}(T_0\times Z_1)\sbt \Br(V).$$
 
(e) pour tout $(t_0,t_1)\in (T_0\times Z_1^{tor})(k)$,  notons 
 $V_{(t_0,t_1)}\xrightarrow{f_1|_{(t_0,t_1)}}t_0\times Z_{1,t_1}\to (t_0,t_1)$
 la fibre de $V\xrightarrow{f_1\circ \tau} T_0\times Z_1\to T_0\times Z_1^{tor} $ et on a que
 la restriction $ \Br_{T_0\times H}(T_0\times Z_1)\twoheadrightarrow \Br_{H_0}(Z_{1,t_1})$ est surjective. 

L'\'enonc\'e (a) r\'esulte de la proposition \ref{proptor}.
La proposition \ref{lemtorequotient} et le lemme \ref{lemlemtorequotient} donne (b).

Pour (c), puisque $\Pic(T_2)=0$, par le corollaire \ref{corbraueralgebraic} et la suite exacte de Sansuc \cite[Prop. 6.10]{S}, on a deux diagrammes commutatifs de suites exactes
 $$\xymatrix{0\ar[r]&\Br_G(Z)\ar[r]^-{p_Z^*|_{\Br_G}}\ar[d]^{f^*|_{\Br_G}}&\Br_{T_0\times H}(T_0\times Z_1)\ar[d]^-{(f_1\circ \tau)^*|_{\Br_{T_0\times H}}}\ar[r]&\Br_a(T_2)\ar[d]^=\\
 0\ar[r]&\Br_G(U)\ar[r]^-{p^*|_{\Br_G}}&\Br_{T_0\times H}(V)\ar[r]&\Br_a(T_2)
 } \ \ \ \text{et}\ \ \ 
 \xymatrix{0\ar[r]&\Br(Z)\ar[r]^-{p_Z^*}\ar[d]^-{f^*}&\Br(T_0\times Z_1)\ar[d]^{(f_1\circ \tau)^*}\\
 0\ar[r]&\Br(U)\ar[r]^{p^*}&\Br(V).
 }$$
 et $(p^*)^{-1}\Br_{T_0\times H}(V)=\Br_G(U)$.
Donc 
$(p^*)^{-1}((f_1\circ \tau )^*\Br_{T_0\times H}(T_0\times Z_1))=f^*\Br_G(Z).$
Une application du corollaire \ref{descent} au torseur $Y\xrightarrow{p}X$ et aux sous-groupes: 
 $$ (B+ f^*\Br_G(Z))\sbt \Br(U)\ \ \ \text{ et}\ \ \  (f_1\circ \tau )^*\Br_{T_0\times H}(T_0\times Z_1)\sbt \Br_{T_0\times H}(V)$$
donne (c).

 Pour (d), par la construction, on a $\bk[Z_1^{tor}]^{\times}\cong \bk[Z_1]^{\times}$, $\Pic(Z_{1,\bk})=0$ et $Z_1(k)\neq\emptyset$. 
Par la suite spectrale de Hochschild-Serre et \cite[Lem. 6.6]{S}, on a $ \Br_1(T_0\times Z_1^{tor})\cong \Br_1(T_0\times Z_1)$.
L'\'enonc\'e (d) d\'ecoule de la proposition \ref{lembrauerinv-braueralg}.

Pour (e), puisque $H^3(k,Z_1^{tor,*})=0$, d'apr\`es le lemme \ref{Dlembrauersurj}, le morphisme $\Br_a(H)\to \Br_a(H_0)$ est surjectif.
  La proposition \ref{propbrauersuj} donne (e).

\medskip

On consid\`ere l'ouvert $W$ de l'\'enonc\'e.  Apr\`es avoir r\'etr\'eci $W$, on peut supposer que tout \'el\'ement de $A$ s'annule sur $W$.
D'apr\`es (c) et (d), on a $(p^{-1}(W))^{\Br(Y)\cap (p^*B+B_1+ f_V^*\Br_1(T_0\times Z_1^{tor}))}\neq \emptyset$. 
  
 Par la suite exacte de Sansuc \cite[Prop. 6.10]{S}, le morphisme canonique $\bk[Y]^{\times}/\bk^{\times}\to \bk[T_2]^{\times}/\bk^{\times}  $ est injectif.
 Puisque $p^*B+B_1\sbt \Br_{T_0\times H}(V)$ est fini, 
une application du th\'eor\`eme \ref{mainlem} au triple
   \begin{equation}\label{triple-e}
  (T_0\times H\to T_0\times Z_1^{tor}, V\sbt Y, V\xrightarrow{f_V}T_0\times Z_1^{tor})
  \end{equation}
   montre qu'il existe
   $$v\in [(T_0\times H)(k_{v_0})^0\cdot T_2(k_{\infty})^+\cdot p^{-1}(W)]\cap V(\RA_k)^{p^*B+B_1}\ \ \  \text{tel que} \ \ \ (t_0,t_1):=f_V(v)\in (T_0\times Z_1^{tor})(k),$$
o\`u $(T_0\times H)(k_{v_0})^0:=\psi^{-1}( G(k_{v_0})^0 )$.

D'apr\`es (d) et (e)  on a 
$v\in V_{(t_0,t_1)}(\RA_k)^{p^*B+f_1|_{(t_0,t_1)}^*\Br_{H_0}(Z_{1,t_1}))}$. Donc
$t:=\pi_{Z^{tor}}((t_0,t_1))\in Z^{tor}(k)$ et
$u:=p(v)\in (G(k_{v_0})^0\cdot W)\cap U_t(\RA_k)^{B+f^*\Br_{G_0}(Z_t)}.$
  Ce qui donne (1).
  
  \medskip
  
On consid\`ere (2).

Fixons $v_0\in S$. On a le plongement  canonique 
de groupes $G(k_{v_0})\sbt G(k_S)$.
Puisque $G(k_S)^0\sbt G(k_S)$ est ouvert d'indice fini,  les sous-groupes 
$$G(k_{v_0})^0:=G(k_S)^0\cap G(k_{v_0})\sbt G(k_{v_0})  \ \ \ \text{et}\ \ \ G_0(k_S)^0:=G(k_S)^0\cap G_0(k_S)\sbt G_0(k_S) $$
 sont ouverts d'indice fini.
 Pour tout $t\in Z^{tor}(k)$, l'ensemble $W_t:=(G(k_{v_0})^0\cdot W)\cap U_t(\RA_k)^{A+B}$ est ouvert dans $U_t(\RA_k)$.
D'apr\`es (1), il existe $t\in Z^{tor}(k)$ tel que $W_t^{f^*\Br_{G_0}(Z_t)}\neq\emptyset$ et donc $f_t(W_t)^{\Br_{G_0}(Z_t)}\neq\emptyset$.
D'apr\`es \cite[Thm. 4.5]{Co}, $f_t(W_t)\sbt Z_t(\RA_k)$ est ouvert.
Par hypoth\`ese, il existe $z\in Z_t(k)\cap f_t(G_0(k_S)^0\cdot W_t)$ et ceci \'etablit (2).
\end{proof}

\begin{rem}
On peut \'etablir le th\'eor\`eme \ref{main1thminfty} par la m\'ethode de la d\'emonstration du th\'eor\`eme \ref{main1thm}.
Mais l'argument dans \S \ref{4} est plus simple.
\end{rem}

\begin{cor}
Avec les hypoth\`eses et notations du th\'eor\`eme \ref{main1thm}, soit $F\sbt X$ sous-sch\'ema ferm\'e de codimension $\geq 2$.
Alors,  pour tout $v_0\in \Omega_k$, tout sous-groupe ouvert d'indice fini $G(k_{v_0})^0\sbt G(k_{v_0})$
 et tout ouvert $\tilde{W}\sbt (X\setminus F)(\RA_k)$ satisfaisant $\tilde{W}^{\Br(X)\cap (A+B+ f^*\Br_G(Z))}\neq \emptyset$,
  il existe un $t\in Z^{\tor}(k)$ tel que 
$$\codim(F\cap U_t,U_t)\geq 2\ \ \ \text{et}\ \ \ (G(k_{v_0})^0\cdot \tilde{W})\cap (U_t\setminus (F\cap U_t))(\RA_k)^{B+ A+f_t^*\Br_{G_0}(Z_t)}\neq\emptyset.$$
\end{cor}

\begin{proof}
Avec les constructions et notations de la d\'emonstration du th\'eor\`eme \ref{main1thm}, d'apr\`es (\ref{main1thm-e1}), 
on a 
$$(X\setminus F)(\RA_k)^{\Br(X)\cap (B+ f^*\Br_G(Z))}=p((Y\setminus p^{-1}F)(\RA_k)^{\Br(Y)\cap (p^*B+ B_1+f_V^*\Br_1(T_0\times Z_1^{tor}))}).$$
Une application du corollaire \ref{mainlemcod2} au triple (\ref{triple-e}) donne le r\'esultat.
\end{proof}

 \section{Le r\'esultat principal}\label{7}
 Dans toute cette section,  $k$ est un corps de nombres. 
Sauf  mention explicite du contraire, une vari\'et\'e est une  $k$-vari\'et\'e. 
Dans cette section, on \'etablit le r\'esultat principal:  le th\'eor\`eme \ref{main1cor2} (ou le th\'eor\`eme \ref{main2thm} sur la version de la fibration).
 
Rappelons la notion de sous-groupe de Brauer invariant (cf. D\'efinition \ref{def-invariant}). 
 
 Soit $G$ un groupe lin\'eaire connexe.
 Soit $S\sbt \Omega_k$ un sous-ensemble fini.
 On consid\`ere tout sous-groupe ouvert d'indice fini $ G(k_S)^0$ de  $G(k_S)$.
 Ainsi $G(k_S)^0$ est ferm\'e dans  $G(k_S)$ et on a directement:
 
 \begin{lem}\label{lemsousgroupeouvert}
 Si $S=\infty_k$, alors $G(k_{\infty})^+\sbt G(k_{\infty})$ est un sous-groupe ouvert d'indice fini 
 et tout tel sous-groupe $G(k_S)^0$ contient $G(k_{\infty})^+$.
 \end{lem}

\begin{lem}\label{main2lem}
Soit $G$ un groupe lin\'eaire connexe et simplement connexe.
Soit $v\in \Omega_k$ une place. 
Supposons que $G$ est unipotent ou que $G$ est semi-simple et simple avec $G(k_v)$ non compact.
Alors $G(k_v)$ ne poss\`ede pas de sous-groupe ouvert d'indice fini non-trivial.
\end{lem}

\begin{proof}
Si $G\cong \BG_a$, ceci vaut car $G(k_v)\cong k_v$ est uniquement divisible.
Dans le cas o\`u $G$ est unipotent, ceci vaut car il existe une filtration de $G$ de facteurs $\BG_a$ (\cite[Cor. 15.5 (ii)]{Bo}).
Ceci vaut aussi pour tout tel $G$ d\'efini sur $k_v$.

Dans le cas o\`u $G$ est semi-simple, simplement connexe et simple avec $G(k_v)$ non compact,
 si $v\in \infty_k$, ceci vaut par E. Cartan (cf. \cite[Prop. 7.6]{PR}).
 Si $v\notin \infty_k$, ceci vaut car $G(k_v)$ est engendr\'e par les $k_v$-points des sous-groupes unipotents de $G$ sur $k_v$ (la conjecture de Kneser-Tits \'etablie par Platonov, cf. \cite[Thm. 7.6]{PR}).
\end{proof}

\begin{prop}\label{main2lem1}
Soit $G$ un groupe lin\'eaire connexe et simplement connexe.
Soit $S\sbt \Omega_k$ un sous-ensemble fini non vide  tel que $G'(k_S)$ soit non compact pour chaque facteur simple $G'$ du groupe $G^{sc}$.
Alors, pour tout sous-groupe ouvert d'indice fini $G(k_S)^0\sbt G(k_S)$ et tout $G$-torseur $P$ sur $k$, l'ensemble $G(k_S)^0\cdot P(k)$ est dense dans $P(\RA_k)$.
\end{prop}

\begin{proof}
On peut supposer que $P(\RA_k)\neq\emptyset$.
 Puisque $\Br_a(G)=0$, par le principe de Hasse pour un $G$-torseur (Kneser, Harder et Chernousov, cf. \cite[Thm. 5.1.1 (e)]{sko}), on a $P(k)\neq\emptyset$. 
 Alors on peut supposer que $G\cong P$.

Si $G$ est soit unipotent soit semi-simple, simplement connexe et simple,  par hypoth\`ese
 il existe une place $v\in S$ tel que $G(k_v)$ soit non compact.
D'apr\`es le lemme \ref{main2lem}, $G(k_v)^0:=G(k_S)^0\cap  G(k_v)$ est exactement $G(k_v)$.
Une application de l'approximation forte de $G$ (Kneser, Platonov, cf. \cite[Thm. 7.12]{PR}) donne l'\'enonc\'e.

En g\'en\'eral, le groupe $G$ poss\`ede une filtration de facteurs soit unipotents soit semi-simples simplement connexes et simples.
Une application de la m\'ethode de fibration (\cite[Prop. 3.1]{CTX13}) donne l'\'enonc\'e.
\end{proof}

\begin{prop}\label{main2prop}
Soient $G$ un groupe lin\'eaire connexe, et $Z$ un $G$-espace homog\`ene \`a stabilisateur g\'eom\'etrique connexe.
Soit $S\sbt \Omega_k$ un sous-ensemble fini non vide  tel que $G'(k_S)$ soit non compact pour chaque facteur simple $G'$ du groupe $G^{sc}$.
Supposons que $S=\infty_k$ ou que $\bk[Z]^{\times}=\bk^{\times}$.
Alors, pour tout sous-groupe ouvert d'indice fini $G(k_S)^0\sbt G(k_S)$ et tout ouvert $W\sbt Z(\RA_k)$ satisfaisant $W^{\Br_G(Z)}\neq\emptyset$, 
on a $Z(k)\cap (G(k_S)^0\cdot W)\neq \emptyset $.
\end{prop}

\begin{proof}
Le cas o\`u $S=\infty_k$ a \'et\'e \'etabli par Borovoi et Demarche (\cite[Thm. 1.4]{BD}). 
Ici, on donne une d\'emonstration unifi\'ee des deux cas consid\'er\'es.

Une application du principe de Hasse pour un espace homog\`ene \`a stabilisateur g\'eom\'etrique connexe (Borovoi \cite{B96}, cf. \cite[Thm. 5.2.1 (a)]{sko})
montrer que $Z(k)\neq \emptyset$. 
Il induit un $G$-morphisme $\pi: G\to Z$ tel que $Z\cong G/G_0$ avec $G_0\sbt G$ un sous-groupe ferm\'e connexe.

Par la r\'esolution flasque \cite[Porp. 5.4]{CT07}, il existe un groupe lin\'eaire connexe $H$ 
et un homomorphisme surjectif $H\xrightarrow{\psi} G$
tels que $\Ker(\psi)$ soit un tore et $H$ soit quasi-trivial, i.e. $H^{tor}$ soit quasi-trivial et $H^{sc}=H^{ss}$.
Alors $Z$ est une $H$-espace homog\`ene \`a stabilisateur g\'eom\'etrique connexe, $H^{sc}\cong G^{sc}$ 
et, d'apr\`es le corollaire \ref{corbraueralgebraic} et la proposition \ref{corbraueralgebraic1}, on a: $\Br_G(Z)=\Br_H(Z).$
Le sous-groupe $H(k_S)^0:=\psi^{-1}(G(k_S)^0)\sbt H(k_S)$ est ouvert d'indice fini.
Alors on peut remplacer $G$ par $H$ et supposer que $G$ est quasi-trivial.

Notons $G\xrightarrow{\phi} G^{tor}$ le quotient torique maximal. 
Alors $\phi$ est lisse \`a fibres g\'eom\'e\-triquement int\`egres et donc $G(\RA_k)\to G^{tor}(\RA_k)$ est ouvert (\cite[Thm. 4.5]{Co}).
D'apr\`es le corollaire \ref{Dmaincor2}, il existe un ouvert $W_1\sbt G(\RA_k)$ et un point $z\in Z(k)$
tels que $\pi (W_1)\cdot z\sbt W$ et $W_1^{\Br_1(G)}\neq\emptyset$.
Puisque $G^{tor}$ satisfait l'approximation forte par rapport \`a $\Br_1(G^{tor})$ hors de $\infty_k$ (Harari \cite[Thm. 2]{Ha08}),
il existe
$$t\in G^{tor}(k)\cap (G^{tor}(k_{\infty})^+\cdot \phi (W_1)) .$$

Notons $G_t$ la fibre de $\phi$ au-dessus de $t$ et $G^{ssu}(k_S)^0:=G(k_S)^0\cap G^{ssu}(k_S)$.
Ainsi  $G_t$ est un $G^{ssu}$-torseur.
D'apr\`es la proposition \ref{main2lem1}, l'ensemble $G^{ssu}(k_S)^0\cdot G_t(k)$ est dense dans $G^{ssu}(\RA_k)$.

Dans le cas o\`u $S=\infty_k$, puisque l'homomorphisme $G(k_{\infty})^+\to G^{tor}(k_{\infty})^+$ est surjectif,
 il existe 
 $$a\in G_t(\RA_k)\cap (G(k_{\infty})^+\cdot W_1)\ \ \ \text{ et donc}\ \ \  g\in G_t(k)\cap (G^{ssu}(k_S)^0\cdot G(k_{\infty})^+\cdot W_1).$$ 
 Par le lemme \ref{lemsousgroupeouvert}, on a $G^{ssu}(k_S)^0\cdot G(k_{\infty})^+\sbt  G(k_S)^0$ et donc $g\cdot z\in Z(k)\cap (G(k_S)^0\cdot W)$.

Dans le cas o\`u $\bk[Z]^{\times}=\bk^{\times}$, par la suite exacte de Sansuc \cite[Prop. 6.10]{S}, 
$G_0^{tor}\to G^{tor}$ est surjectif et donc $G_0\to G \to G^{tor}$ est surjectif.
Ainsi $G_0(k_{\infty})^+\to G^{tor}(k_{\infty})^+$ est surjectif.
Alors il existe  $a\in G_t(\RA_k)\cap (G_0(k_{\infty})^+\cdot W_1)$ et donc 
$$g\in G_t(k)\cap (G^{ssu}(k_S)^0\cdot G_0(k_{\infty})^+\cdot W_1).$$
Ainsi $g\cdot z\in Z(k)\cap (G(k_S)^0\cdot W)$.
\end{proof}

 \begin{thm}\label{main2thm}
Soient $G$ un groupe lin\'eaire connexe, $G_0\sbt G$ un sous-groupe ferm\'e connexe et $Z:=G/G_0$.
 Soient $X$ une $G$-vari\'et\'e lisse g\'eom\'etriquement int\`egre,  $U\sbt X$ un $G$-ouvert  
 et $U\xrightarrow{f}Z$ un $G$-morphisme. 
 Soient $A\sbt \Br(X)$ et $B\sbt \Br_G(U)$ (cf. (\ref{defBr_GX})) deux sous-groupes finis. 
 Soit $S\sbt \Omega_k$ un sous-ensemble fini non vide  tel que $G'(k_S)$ soit non compact pour chaque facteur simple $G'$ du groupe $G^{sc}$.
 Supposons que $S=\infty_k$ ou que $\bk[X]^{\times}/\bk^{\times}=0$.
 Alors, pour tout sous-groupe ouvert d'indice fini $G(k_S)^0\sbt G(k_S)$
 et tout ouvert $W\sbt X(\RA_k)$ satisfaisant $W^{\Br(X)\cap (A+B+ f^*\Br_G(Z))}\neq\emptyset ,$  
il existe   un $z\in Z(k)$ de fibre $U_z$, tel que 
 $$(G(k_S)^0\cdot W)\cap U_z(\RA_k)^{B+ A}\neq\emptyset .$$
 \end{thm}
 
 \begin{proof}
Le cas o\`u $S=\infty_k$ d\'ecoule du th\'eor\`eme \ref{main1thminfty} (2) et de la proposition \ref{main2prop}.
 
 Soit $\pi: Z\xrightarrow{\pi} T$ le quotient torique maximal de $Z$ et $G_1\sbt G$ le stabilisateur de $G$ sur $T$ (cf. D\'efinition \ref{defitorequotient}).
 Pour tout $t\in T(k)$, notons $Z_t$ la fibre de $\pi$ au-dessus de $t$.
 Alors $Z_t$ est un $G_1$-espace homog\`ene \`a stabilisateur g\'eom\'etrique connexe.
 Par la proposition \ref{propbrauersuj}, on a $\bk[Z_t]^{\times}/\bk^{\times}=0$.
 
 Le cas o\`u $\bk[X]^{\times}/\bk^{\times}=0$ d\'ecoule du  th\'eor\`eme \ref{main1thm} (2) et de la proposition \ref{main2prop}.
 \end{proof}

\begin{thm}\label{main1cor2}
Soient $G$ un groupe lin\'eaire connexe, $G_0\sbt G$ un sous-groupe ferm\'e connexe et $Z:=G/G_0$.
 Soient $X$ une $G$-vari\'et\'e lisse g\'eom\'etriquement int\`egre, 
  $U\sbt X$ un $G$-ouvert  et $U\xrightarrow{f}Z$ un $G$-morphisme. 
 Soient $A\sbt \Br(X)$ et $B\sbt \Br_G(U)$ (cf. (\ref{defBr_GX})) deux sous-groupes finis.
 Soit $S\sbt \Omega_k$ un sous-ensemble fini non vide tel que $G'(k_S)$ soit non compact pour chaque facteur simple $G'$ du groupe $G^{sc}$.

 (1) Si $S= \infty_k$ et, pour tout $z\in Z(k)$ de fibre $U_z$, l'ensemble $U_z(k)$ est dense dans $U_z(\RA_k)_{\bullet}^{A+B}$,
 alors $X(k)$ est dense dans $X(\RA_k)_{\bullet}^{\Br(X)\cap (A+B+f^*\Br_G(Z))}$.
 
 (2) Si $\bk[X]^{\times}=\bk^{\times}$ et, pour tout $z\in Z(k)$, la fibre $U_z$ satisfait l'approximation forte de Brauer-Manin   par rapport \`a $A+B$ hors de $S$,
alors $X$ satisfait l'approximation forte de Brauer-Manin par rapport \`a $\Br(X)\cap (A+B+f^*\Br_G(Z))$ hors de $S$.
\end{thm}

\begin{proof}
Ceci suit imm\'ediatement du th\'eor\`eme \ref{main2thm}. 
\end{proof}

\bigskip

\noindent{\bf Remerciements.}
Je remercie tr\`es chaleureusement Jean-Louis Colliot-Th\'el\`ene et Fei Xu pour  plusieurs discussions.
Je remercie \'e\-galement Cyril Demarche, Qifeng Li et Giancarlo Lucchini Arteche pour leurs commentaires.
Projet soutenu par l'attribution d'une allocation de recherche R\'egion Ile-de-France.

\bibliographystyle{alpha}
\end{document}